\documentclass[11pt,a4paper]{article}
\usepackage{multirow}
\usepackage{epsfig}
\usepackage{epstopdf}
\usepackage[T1]{fontenc}
\usepackage{geometry}
\usepackage{amsbsy,amsmath,latexsym,amsfonts, epsfig, color, authblk, amssymb, graphics, bm}
\usepackage{epsf,slidesec,epic,eepic}
\usepackage{fancybox}
\usepackage{fancyhdr}
\usepackage{setspace}
\usepackage{cases}
\usepackage{subfigure}
\usepackage{epsfig}
\usepackage{epstopdf}
\usepackage{nccmath}
\usepackage{tikz}
\usepackage{algorithm}
\usepackage{algorithmic}
\usepackage[colorlinks, citecolor=blue]{hyperref}

\newtheorem{theorem}{Theorem}[section]
\newtheorem{lemma}{Lemma}[section]
\newtheorem{definition}{Definition}[section]

\newtheorem{proposition}{Proposition}[section]
\newtheorem{corollary}{Corollary}[section]

\newtheorem{remark}{Remark}[section]

\newenvironment{proof}{{\noindent \bf Proof:}}{\hfill$\Box$\medskip}

\definecolor{lred}{rgb}{1,0.8,0.8}
\definecolor{lblue}{rgb}{0.8,0.8,1}
\definecolor{dred}{rgb}{0.6,0,0}
\definecolor{dblue}{rgb}{0,0,0.5}
\definecolor{dgreen}{rgb}{0,0.5,0.5}

\title{Tilt stability of Ky-Fan $\kappa$-norm composite optimization}
\author{Yulan Liu\footnote{(ylliu@gdut.edu.cn), School of Mathematics and Statistics, Guangdong University of Technology, Guangzhou},
\ \ Shaohua Pan\footnote{(shhpan@scut.edu.cn), School of Mathematics, South China University of Technology, Guangzhou.}
\ \ {\rm and}\ \ Wen Song\footnote{(wsong@hrbnu.edu.cn), School of Mathematical and Sciences, Harbin Normal University, Harbin.}}
\date{}

\begin{document}

\maketitle

\begin{abstract}
 This paper concerns the tilt stability for the minimization of the sum of a twice continuously differentiable matrix-valued function and the Ky-Fan $\kappa$-norm. To achieve this goal, we first provide a sufficient and necessary condition for a local minimizer of the composite $f=\varphi+g$ to be tilt-stable with the second subderivative of $g$, where $g$ is a closed proper convex function, and $\varphi$ is a twice continuously differentiable function that is locally convex at the local minimizer. Then, we apply the sufficient and necessary condition to the concerned Ky-Fan $\kappa$-norm composite problem, and employ the expression of second subderivative of the Ky-Fan $\kappa$-norm to derive a verifiable criterion to identify the tilt stability of a local minimum for this class of nonconvex and nonsmooth problems. As a byproduct, a practical criterion is obtained for identifying the tilt stablity of solutions to the nuclear-norm and spectral norm regularized  minimization problems. 
\end{abstract}


\section{Introduction}

 Let $\mathbb{R}^{n\times m}\,(n\!\le\! m)$ represent the space of all $n\times m$ real matrices endowed with the trace inner product $\langle \cdot,\cdot\rangle$ and its induced Frobenius norm $\|\cdot\|_F$. For an integer $1\!\le\!\kappa\!\le\! n$, let $\Psi_{\kappa}(X)\!:=\!\sum_{i=1}^{\kappa}\sigma_i(X)$ denote the Ky-Fan $\kappa$-norm, where $\sigma_i(X)$ means the $i$th largest singular value of $X\!\in\!\mathbb{R}^{n\times m}$. We focus on  the following composite optimization problem  
 \begin{equation}\label{KnormRegular}
 \min_{X\in\mathbb{R}^{n\times m}}\Theta_{\nu,\kappa}(X):=\nu\vartheta(X)+\Psi_{\kappa}(X),
 \end{equation}
 where $\vartheta\!:\mathbb{R}^{n\times m}\to\mathbb{R}$ is a twice continuously differentiable function, and $\nu>0$ is a regularization parameter. Let $\mathcal{B}\!:\mathbb{R}^{n\times m}\to\mathbb{S}^{n+m}$ be the linear mapping defined by
 \begin{equation}\label{Bmap}
 	\mathcal{B}(X):=\begin{pmatrix}
 		0 & X\\
 		X^{\top}& 0
 	\end{pmatrix}\quad{\rm for}\ X\in\mathbb{R}^{n\times m},  
 \end{equation}
 and let $\Phi_{\kappa}(Z):=\sum_{i=1}^{\kappa}\lambda_i(Z)$ denote the sum function of the first $\kappa$ largest eigenvalues of matrices from $\mathbb{S}^{n+m}$, where $\lambda_i(Z)$ is the $i$th largest eigenvalue of $Z$.
 Apparently, $\Psi_{\!\kappa}$ is the composition of $\Phi_{\!\kappa}$ and the mapping $\mathcal{B}$, i.e., $\Psi_{\kappa}(\cdot)=\Phi_{\kappa}(\mathcal{B}(\cdot))$. The composite optimization problem \eqref{KnormRegular} has an extensive application in many fields such as matrix norm approximation \cite{Chen2016}, matrix completion and sensing \cite{Candes12,Fazel09}, control and system identification \cite{Sun13}, signal and image processing \cite{Koliada20}, and so on. 
 \subsection{Related work}\label{sec1.1}

 Tilt stability, first introduced by Poliquin and Rockafellar \cite{Poliquin98}, is a kind of single-valued Lipschitzian behavior of local minimizers of an extended real-valued function with respect to one-parametric linear or tilt perturbation. Until now, many different characterizations have been delved for it. In the original paper \cite{Poliquin98}, the limiting coderivative of the subdifferential mapping/generalized Hessian of objective functions is used to characterize tilt stability, and a quantitative version of this characterization was also given by Mordukhovich and Nghia \cite{Mordu2015}. Drusvyatskiy et al. \cite{DruLewis2014,Drusvy2014} established that tilt stability, uniform second-order growth, and strong metric regularity of subdifferential mapping are equivalent for any lower semicontinuous (lsc) extended real-valued functions. Chieu et al. \cite{ChienSiam2018} characterized tilt stability by the positive definiteness of subgradient graphical derivatives. Recently, for the minimization of the sum of a twice continuously differentiable convex function and a proper lsc convex function $g$, Nghia \cite{Nghia2025} presented a novel characterization for tilt stability by leveraging the second subderivative of the sum function under the assumption that $g$ satisfies the quadratic growth condition around the reference point $(\overline{x},\overline{y})\in{\rm gph}\,\partial g$ for a set $\mathcal{M}$ and the subdifferential mapping $\partial g$ has a relative approximations onto $\mathcal{M}$ around $(\overline{x},\overline{y})$, and obtained the verifiable criterion for three classes of specific $g$. For the applications of the equivalent characterization of tilt stability from \cite{Poliquin98} in different setting, see \cite{Lewis2013, Gfrerer2015, Mordu2021, Mordu2015}. During the review process of our paper, we notice that Cui et al. \cite{Cui2024} investigated the Lipschitz stability of the problem
 \begin{equation}\label{para-prob}
  \min_{x\in\mathbb{X}} f(x,\overline{p})+g(x),
 \end{equation}
 where $\overline{p}$ is a fixed vector of the parameteric space $\mathbb{P}$, $f:\mathbb{X}\times\mathbb{P}\to\overline{\mathbb{R}}$ is twice continuously differentiable around $(\overline{x},\overline{p})$ with $\overline{x}$ being an optimal solution of \eqref{para-prob} and $f(\cdot,\overline{p})$ is a convex function, and $g:\mathbb{X}\to\overline{\mathbb{R}}$ is a closed proper convex function. They proved in \cite[Theorem 4.4]{Cui2024} that, if the conjugate $g^*$ of $g$ is $\mathcal{C}^2$-cone reducible at $\overline{z}:=-\nabla_{x}f(\overline{x},\overline{p})\in\partial g(\overline{x})$, then $\overline{x}$ is a tilt stable optimal solution of problem \eqref{para-prob} if and only if 
 \begin{equation}\label{Cui-cond}
 {\rm Ker}\,\nabla_{\!xx}^2\!f(\overline{x},\overline{p})\cap{\rm par}\,\partial g^*(\overline{z})=\{0\}
 \end{equation}
 where ${\rm par}\,\partial g^*(\overline{z})$ denotes the paralle subspace of $\partial g^*(\overline{z})$. Their result recovers the corresponding findings in \cite{Nghia2025} when $g$ is the $\ell_1$-norm, the $\ell_1/\ell_2$-norm and the nuclear norm. 
 
 \subsection{Main contribution}
 
 Inspired by \cite{Nghia2025}, this work aims at providing a verifiable characterization for the tilt-stability of the Ky-Fan $\kappa$-norm regularization problem \eqref{KnormRegular}. To attain this goal, we first provide a sufficient and necessary condition for a local minimizer of the composite $f$ in \eqref{prob} to be tilt-stable with the second subderivative of $g$; see Proposition \ref{TiltProp2}. Then, we characterize the expression of the second subderivative of $\Psi_{\kappa}$ by  the chain rule developed in \cite{MohMSSiam2020} and the fact that $\Psi_{\kappa}$ is the composition of $\Phi_{\kappa}$ and the linear mapping $\mathcal{B}$, and apply Proposition \ref{TiltProp2} to the composite $\Theta_{\nu,\kappa}$ to obtain a sufficient and necessary criterion for identifying the tilt stability of a local minimum of problem \eqref{KnormRegular}; see Theorem \ref{ThTilt}. This criterion is point-based and checkable. As a byproduct, we recover the practical criterion obtained in \cite{Nghia2025} for the nuclear norm regularized problem. Different from \cite{Nghia2025}, our work establishes the equivalent characterization of tilt stability for problem \eqref{prob} by operating directly the expression of the second subderivative of $g$, rather than finding a set $\mathcal{M}$ such that $g$ satisfies the quadratic growth condition around the reference point for $\mathcal{M}$ and  $\partial g$ has a relative approximation onto $\mathcal{M}$ around the reference point. Our condition for $\overline{X}$ to be a tilt stable solution of \eqref{KnormRegular} with $\kappa=1$ is demonstrated in Remark \ref{main-remark41} to be weaker than the above \eqref{Cui-cond} applied to \eqref{KnormRegular} in some cases, though it is same as the above \eqref{Cui-cond} when $\kappa=n$. In addition, it is worth pointing out that unlike  \cite{Nghia2025,Cui2024}, our criterion is also applicable to the case that $\vartheta$ is nonconvex but locally convex at the reference point.
 \subsection{Notation}
 
 Throughout this paper, $\mathbb{S}^p$ represents the space of all $p\times p$ real symmetric matrices, $\mathbb{S}_{+}^p$ denotes the set of all positive semidefinite matrices in $\mathbb{S}^p$, and $\mathbb{O}^{p\times k}$ denotes the set of all $p\times k$ real matrices with orthonormal columns and write $\mathbb{O}^{p}\!:=\!\mathbb{O}^{p\times p}$. The notation $I_p$ means a $p\times p$ identity matrix,  $I^{\uparrow}_p$ denotes the $p\times p$ anti-diagonal matrix whose anti-diagonal entries are all ones and others are zeros, and $e$ stands for a vector of all ones whose dimension is known from the context. For an integer $k\ge 1$, write $[k]:=\{1,\ldots,k\}$. For a vector $x\in\mathbb{R}^n$ and an index set $J\subset[n]$, $x_{\!J}$ denotes the vector in $\mathbb{R}^{|J|}$ obtained by deleting those entries $x_j$ with $j\notin J$. For a matrix $Z\in\mathbb{S}^{p}$, $\lambda_i(Z)$ means the $i$th largest eigenvalue of $Z$, and $\mathbb{O}^p(Z):=\{P\in\mathbb{O}^p\,|\,Z=P{\rm Diag}(\lambda(Z))P^{\top}\}$ with  $\lambda(Z)\!:=(\lambda_1(Z),\ldots,\lambda_p(Z))^{\top}$. For a matrix $X\in\mathbb{R}^{n\times m}$, $\sigma_i(Z)$ represents the $i$th largest singular value of $Z$, and  $\mathbb{O}^{n,m}(X):=\{(U,V)\in\mathbb{O}^n\times\mathbb{O}^m\ |\ X=U[{\rm Diag}(\sigma(X)\ \ 0]V^{\top}\}$ with $\sigma(X)=(\sigma_1(X),\ldots,\sigma_n(X))^{\top}$. For a matrix $X\in\mathbb{R}^{n\times m}$, $\|X\|$ and $\|X\|_*$ denote the spectral norm and nuclear norm of $X$, respectively, $X^\dagger$ means the Moore-Penrose pseudo-inverse of $X$, 
 $X_{IJ}\in\mathbb{R}^{|I|\times |J|}$ for index sets $I\subset[n]$ and $J\subset[m]$ denotes the submatrix obtained by removing all rows not in $I$ and all columns not in $J$, and write $X_{J}:=X_{IJ}$ if $I=[n]$. Let 
 $\mathcal{S}(X):=(X+X^\top)/2$ and $\mathcal{T}(X):=(X-X^\top)/2$ for $X\in\mathbb{R}^{n\times n}$.
 \section{Preliminaries}
 
 Let $\mathbb{X}$ be a finite dimensional real vector space equipped with the inner product $\langle \cdot,\cdot\rangle$ and its induced norm $\|\cdot\|$. An extended real-valued function $h\!:\mathbb{X}\to \mathbb{\overline{R}}\!:=\mathbb{R}\cup \{+\infty\}$ is said to be proper if its domain ${\rm dom}\,h\!:=\{x\in \mathbb{X}\,|\, h(x)<\infty\}$ is nonempty. We first recall from the monograph \cite{RW98} the regular and basic subdifferentials of $h$ at a point $x\in {\rm dom}\,h$.
 \begin{definition}\label{gsubdiff}
 Consider a function $h\!:\mathbb{X}\to\overline{\mathbb{R}}$ and a point $x\in{\rm dom}\,h$. The regular subdifferential of $h$ at $x$, denoted by
 $\widehat{\partial}h(x)$, is defined as
 \[
 \widehat{\partial}h(x):=\bigg\{v\in\mathbb{X}\ \big|\
  \liminf_{x'\to x\atop x'\ne x}
 \frac{h(x')-h(x)-\langle v,x'-x\rangle}{\|x'-x\|}\ge 0\bigg\},
 \]
 and the basic (known as limiting or Morduhovich) subdifferential of $h$ at $x$ is defined as
 \[
 	\partial h(x):=\Big\{v\in\mathbb{X}\ |\  \exists\,x^k\to x\ {\rm with}\ h(x^k)\to h(x)\ {\rm and}\ v^k\to v\ {\rm with}\
 	v^k\in\widehat{\partial}h(x^k)\Big\}.
 \]
 \end{definition}
 
 By Definition \ref{gsubdiff}, $\widehat{\partial}h(x)\subset\partial h(x)$,  $\widehat{\partial}h(x)$ is a closed convex set, and $\partial h(x)$ is closed but generally nonconvex. When $h$ is convex, $\widehat{\partial}h(x)=\partial h(x)$, and they reduce to the subdifferential in the sense of convex analysis.
  Next we recall the subderivative of $h$.
 \begin{definition}
 (see \cite[Definition 7.20]{RW98}) Consider a function $h\!:\mathbb{X}\to\mathbb{\overline{R}}$ and a point $x\in{\rm dom}\,h$. The subderivative function $dh(x)\!:\mathbb{X}\to[-\infty,\infty]$ of $h$ at $x$ is defined as
 \[
  dh(x)(w):=\liminf_{\tau\downarrow 0\atop w'\to w}\frac{h(x+\tau w')-h(x)}{\tau}\quad{\rm for}\ w\in \mathbb{X},
 \]
 and $h$ is said to be (properly) epi-differentiable at $x$ if the first-order quotient function $\Delta_{\tau}h(x)(\cdot)\!:=\tau^{-1}[h(x+\tau \cdot)-h(x)]$ epi-converges to the (proper) function $dh(x)$ as $\tau\downarrow 0$. 
 \end{definition}
 
 When $h\!:\mathbb{X}\to \mathbb{\overline{R}}$ is directionally differentiable, $dh(x)(\cdot)\le h'(x;\cdot)$, and the inequality becomes an equality if $h$ is strictly continuous at $x$. With the subderivative function, we recall the critical cone to $h$ at a point $(x,v)$ with $ v\in\partial h(x)$:
 \begin{equation*}
 \mathcal{C}_{h}(x,v):=\big\{w\in\mathbb{X}\ |\ dh(x)(w)=\langle v,w\rangle\big\}.
 \end{equation*} 
 \subsection{Second subderivative and graphical derivative}\label{sec2.1}
 
 Next we introduce the second subderivative of an extended real-valued function. 
 \begin{definition}\label{DefSubder}
  (see \cite[Definition 13.3]{RW98}) Given a function $h\!:\mathbb{X}\to\mathbb{\overline{R}}$, a point $x\in{\rm dom}\,h$ and a vector $v\in\mathbb{X}$. The second subderivative of $h$ at $x$ for $v$ and $w$ is defined as
  \[
 	d^2h(x|v)(w):=\liminf_{\substack{\tau\downarrow 0\\ w'\to w}}\frac{h(x\!+\!\tau w')-h(x)-\tau\langle v,w'\rangle}{\tau^2/2},
  \]
 and $h$ is said to be (properly) twice epi-differentiable at $x$ for $v$ if the second-order quotient
 \[
  \Delta^2_{\tau}h(x|v)(\cdot):=2\tau^{-2}[h(x+\tau \cdot)-h(x)-\tau\langle v,\cdot\rangle] 
 \]
 epi-converges to the (proper) function $d^2h(x|v)$ as $\tau\downarrow 0$.
 \end{definition} 

 From \cite[Proposition 13.5]{RW98}, $d^2h(x|v)$ is lsc and positively homogeneous of degree $2$, and ${\rm dom}\,d^2h(x|v)\subset\{w\in\mathbb{X}\,|\,dh(x)(w)\le\langle v,w\rangle\}$, and moreover, the properness of $d^2h(x|v)$ implies that $v\in\widehat{\partial}h(x)$ and
 ${\rm dom}\,d^2h(x|v)\subset\mathcal{C}_{h}(x,v)$. Considering that the twice epi-differentiability of extended real-valued functions is often established under parabolic regularity, we here introduce the definition of this regularity condition. 
 \begin{definition}\label{def-pregular}
 A function $h\!:\mathbb{X}\to\overline{\mathbb{R}}$ is parabolically regular at $\overline{x}$ for $\overline{v}$ if $h(\overline{x})$ is finite and if for any $w$ such that $d^2h(\overline{x}|\overline{v})(w)<\infty$, there exist, among the sequences $\tau_k\downarrow 0$ and $w^k\to w$ with $\Delta_{\tau_k}^2h(\overline{x}|\overline{v})(w^k)\to d^2h(\overline{x}|\overline{v})(w)$, those with the additional property that $\limsup_{k\to\infty}\frac{\|w^k-w\|}{\tau_k}<\infty$.
 \end{definition}
  
 The second subderivative of $h\!:\mathbb{X}\to\overline{\mathbb{R}}$ at a point $(x,v)\in{\rm gph}\,\partial h$, the graph of its subdifferential mapping, has a close relation with the graphical derivative of $\partial h$ at this point. According to \cite[Definition 8.33]{RW98}, the graphical derivative of the subdifferential mapping $\partial h$ at $(x,v)\in{\rm gph}\,\partial h$ is the mapping $D\partial h(x|v):\mathbb{X}\rightrightarrows \mathbb{X}$ defined by
 \[
 	D\partial h(x|v)(w):=\big\{z\in \mathbb{X}\,|\,(w,z)\in \mathcal{T}_{{\rm gph}\,\partial h}(x,v)\big\},
 \]
 where $\mathcal{T}_{{\rm gph}\,\partial h}(x,v)$ is the tangent cone to ${\rm gph}\,\partial h$ at $(x,v)$ in the sense of \cite[Eq. 6(3)]{RW98}. The following lemma discloses the connection between the second subderivative of $h$ at $(x,v)\in{\rm gph}\,\partial h$ with its subgradient graphical derivative at this point.
 \begin{lemma}\label{relation}(see \cite[Lemma 2.4]{Nghia2025}) Let $h$ be a proper lsc convex function on $\mathbb{X}$, and denote its conjugate function as $h^*$. For any $(x,v)\in{\rm gph}\,\partial h$, it holds
  \[
    2\langle z,w\rangle\ge d^2h(x|v)(w)+d^2h^*(v|x)(z)\ \ {\rm whenever}\ z\in D\partial h(x|v)(w).
  \]
 \end{lemma}

 In addition, second subderivative has a nontrivial
 connection with parabolic subderivative. Let $h\!:\mathbb{X}\to\overline{\mathbb{R}}$, and let $\overline{x}\in{\rm dom}\,h$ and $w\in\mathbb{X}$ with $dh(\overline{x})(w)$ finite. The parabolic subderivative of $h$ at $\overline{x}$ for $w$ with respect to $z$ is defined by
 \[
 d^2h(\overline{x})(w|z):=\liminf_{\tau\downarrow 0\atop z'\to z}\frac{h(\overline{x}+\tau w+\frac{1}{2}\tau^2z')-h(\overline{x})-\tau dh(\overline{x})(z')}{\frac{1}{2}\tau^2}.
 \]
 From \cite[Theorem 13.64]{RW98}, for any $(\overline{v},w)\in\mathbb{X}\times\mathbb{X}$ with
 $dh(\overline{x})(w)=\langle\overline{v},w\rangle$, it holds $d^2h(\overline{x}|v)(w)\le\inf_{z\in\mathbb{X}}\big\{d^2h(\overline{x})(w|z)-\langle\overline{v},z\rangle\big\}$.
 Recall from \cite[Definition 13.59]{RW98} that $h$ is called parabolically epi-differentiable at $\overline{x}$ for $w$ if ${\rm dom}\,d^2h(\overline{x})(w|\cdot)\ne\emptyset$ and for every $z\in\mathbb{X}$ and every sequence $\tau_k\downarrow 0$ there exists a sequences $z^k\to z$ such that 
 \[
   d^2h(\overline{x})(w|z)=\lim_{k\to\infty}\frac{h(\overline{x}+\tau_k w+\frac{1}{2}\tau_k^2z^k)-h(\overline{x})-\tau_k dh(\overline{x})(z^k)}{\frac{1}{2}\tau_k^2}.
 \]

 \subsection{Tilt stability of a class of composite problems}\label{sec2.2}
 
 We first recall the formal definition of tilt-stable local minimum from the work \cite{Poliquin98}.
 \begin{definition}
 For a proper lsc function  $f\!:\mathbb{X}\to\overline{\mathbb{R}}$, a point $\overline{x}\in{\rm dom}\,f$ is called its tilt-stable local minimum if there exists $\delta>0$ such that the following mapping
 \[
  M_{\delta}(v):=\mathop{\arg\min}_{\|x-\overline{x}\|\le\delta}\big\{f(x)-f(\overline{x})-\langle v,x-\overline{x}\rangle\big\}
 \]
 is single-valued and Lipschitz continuous on some neighborhood of $v\!=\!0$ with $M_{\delta}(0)\!=\!\{\overline{x}\}$.   	
 \end{definition}

 The following lemma provides some characterizations for the tilt stability of a local minimizer of $f$, whose proofs can be found in \cite[Theorem 3.3]{ChienSiam2018} and \cite[Theorem 3.3]{DruLewis2014}.
 \begin{lemma}\label{tilt-lemma1}
  Let $f\!:\mathbb{X}\to\overline{\mathbb{R}}$ be an lsc proper function with $\overline{x}\in{\rm dom}\,f$ and $0\in\partial f(\overline{x})$. Assume that $f$ is both prox-regular and subdifferentially continuous at $\overline{x}$ for $\overline{v}=0$. Then the following assertions are equivalent. 
  \begin{itemize}
  \item[(i)] The point $\overline{x}$ is a tilt-stable local minimizer of $f$ with modulus $\gamma>0$.
  
  \item[(ii)] There is a constant $\eta>0$ such that for all $w\in\mathbb{X}$, 
  \[
    \langle z,w\rangle\ge\frac{1}{\gamma}\|w\|^2\ \ {\rm whenever}\ z\in D\partial\!f(x|v)(w),(x,v)\in{\rm gph}\,\partial\!f\cap\mathbb{B}((\overline{x},0),\eta),
  \]
  where $\mathbb{B}((\overline{x},0),\eta)$ denotes the closed ball centered at $(\overline{x},0)$ with radius $\eta>0$. 
 
 \item[(iii)] There exists $\alpha>0$ and neighborhoods $\mathcal{U}$ of $\overline{x}$ and $\mathcal{V}$ of $0$ so that the localization of $(\partial\!f)^{-1}$ relative to $\mathcal{V}$ and $\mathcal{U}$ is single-valued and for all $x\in\mathcal{U}$,
 \[
   f(x)\ge f(\widetilde{x})+\langle\widetilde{v},x-\widetilde{x}\rangle+\alpha\|x-\widetilde{x}\|^2\ {\rm whenever}\ (\widetilde{x},\widetilde{v})\in[\mathcal{U}\times\mathcal{V}]\cap{\rm gph}\,\partial\!f.
 \]
 \end{itemize}
 \end{lemma}
 
 As previously mentioned, Nghia \cite{Nghia2025} recently used the second subderivative to provide a characterization for tilt stability of a minimizer of $f=\varphi+g$, where $\varphi:\mathbb{X}\to\mathbb{R}$ is a twice continuously differentiable convex function, and $g:\mathbb{X}\to\overline{\mathbb{R}}$ is a proper lsc convex function. Here we are interested in the characterization of tilt stability for the problem 
 \begin{equation}\label{prob}
 \min_{x\in\mathbb{X}}f(x):=\varphi(x)+g(x),
 \end{equation}
 where $\varphi\!:\mathbb{X}\to\overline{\mathbb{R}}$ is a proper lsc function that is twice continuously differentiable on an open set $\mathcal{O}\supset{\rm dom}\,g$, and $g$ is the same as above. Obviously, $f$ is prox-regular and subdifferentially continuous, and  model \eqref{prob} covers the case that $g$ is weakly convex. 
 From \cite[Exercise 13.18]{RW98}, the second subderivative of $f$ at $x\in{\rm dom}\,g$ for $v$ has the form
\begin{equation}\label{SecDerForm}
 d^2\!f(x|v)(w)=\langle w, \nabla^2\varphi(x)(w)\rangle+d^2g(x|\,v\!-\!\nabla\varphi(x))(w)\quad{\rm for\ all}\ w\in \mathbb{X}.
\end{equation}
By Lemma \ref{relation} and Lemma \ref{tilt-lemma1}, we present a characterization for the tilt stability of a local minimizer of $f$ given in \eqref{prob} by leveraging its second subderivative.
 \begin{proposition}\label{TiltProp}
 Consider a local minimizer $\overline{x}$ of $f$. If $\overline{x}$ is a tilt-stable solution of \eqref{prob}, then there exist $\ell>0$ and an open neighborhood $\mathcal{U}\times\mathcal{V}$ of $(\overline{x},0)$ such that 
 \begin{equation}\label{d2-ineq0}
 d^2\!f(x|v)(w)\ge\ell\|w\|^2\quad{\rm for\ all}\ (x,v)\in{\rm gph}\,\partial\! f\cap[\mathcal{U}\times\mathcal{V}]\ {\rm and}\ w\in\mathbb{X}.
 \end{equation}
 Conversely, $\overline{x}$ is tilt-stable whenever either of the following conditions holds:
 \begin{itemize}
 \item[(a)] There exist $\eta>0$ and $\gamma>0$ such that for all $(x,v)\in{\rm gph}\,\partial\! f\cap\mathbb{B}((\overline{x},0),\eta)$ and $w\in\mathbb{X}$, the function $g$ is twice epi-differentiable at $x$ for $v-\!\nabla\varphi(x)$ and 
 \[
 \langle \nabla^2\varphi(x)w+z,w\rangle\ge\frac{1}{\gamma}\|w\|^2\ \ {\rm whenever}\ z\in \partial h(w)\ {\rm with}\ h=\frac{1}{2}d^2g(x|\,v-\!\nabla\varphi(x)),
 \]
 where $\mathbb{B}((\overline{x},0),\eta)$ denotes the closed ball centered at $(\overline{x},0)$ with radius $\eta>0$.  
 
 \item[(b)] $\nabla^2\varphi(\cdot)\succeq 0$ on an open neighborhood $\mathcal{N}\subset\mathcal{O}$ of $\overline{x}$ and the condition \eqref{d2-ineq0} is satisfied.
 \end{itemize}
 \end{proposition}
 \begin{proof}
 Let $\overline{x}$ be a tilt-stable solution of \eqref{prob}. By Lemma \ref{tilt-lemma1} (iii), there exist $\ell'>0$ and an open neighborhood $\mathcal{U}\times\mathcal{V}$ of $(\overline{x},0)$ such that for any $(z,v)\in{\rm gph}\,\partial\!f\cap[\mathcal{U}\times\mathcal{V}]$ and any $x\in\mathcal{U}$, 
 \[
 f(x)-f(z)-\langle v,x-z\rangle\ge({\ell'}/{2})\|x-z\|^2,
 \] 
 which by Definition \ref{DefSubder} implies that $d^2\!f(z|v)(w)\ge({\ell'}/{2})\|w\|^2$ for all $w\in\mathbb{X}$. Consequently, the inequality \eqref{d2-ineq0} holds for all $(z,v)\in{\rm gph}\,\partial\! f\cap[\mathcal{U}\times\mathcal{V}]$ and $w\in\mathbb{X}$.   
 
 In what follows, we focus on the proof of the converse conclusion. Suppose that the condition (a) is satisfied. We claim that Lemma \ref{tilt-lemma1} (ii) holds. Fix any $(x,v)\in{\rm gph}\,\partial\! f\cap\mathbb{B}((\overline{x},0),\eta)\cap[\mathcal{O}\times\mathbb{X}]$ and $w\in\mathbb{X}$. Pick any $z'\in D\partial\!f(x|v)(w)$.
 Since $g$ is twice epi-differentiable at $x$ for $v-\nabla\varphi(x)$, from \cite[Exercise 13.18]{RW98}, $f$ is twice epi-differentiable at $x$ for $v$, which along with \cite[Theorem 13.40]{RW98} implies $D\partial\!f(x|v)(w)=\nabla^2\varphi(x)w+\partial h(w)$ with $h=\frac{1}{2}d^2g(x|\,v-\!\nabla\varphi(x))$. Thus, there exists $z\in\partial h(w)$ such that $z'=\nabla^2\varphi(x)w+z$. From the condition (a), it follows $\langle z',w\rangle=\langle\nabla^2\varphi(x)w+z,w\rangle\ge\frac{1}{\gamma}\|w\|^2$. This shows that Lemma \ref{tilt-lemma1} (ii) holds, so $\overline{x}$ is a tilt-stable solution. Now suppose that the condition (b) is satisfied. Then, there exist $\ell>0$ and an open neighborhood $\mathcal{U}\times\mathcal{V}$ of $(\overline{x},0)$ such that the inequality \eqref{d2-ineq0} holds for all $(x,v)\in{\rm gph}\,\partial\! f\cap[\mathcal{U}\times\mathcal{V}]$ and $w\in\mathbb{X}$. Fix any $(x,v)\in{\rm gph}\,\partial\! f\cap[(\mathcal{U}\cap\mathcal{N})\times\mathcal{V}]$ and $w\in\mathbb{X}$. Pick any $z\in D\partial\!f(x|v)(w)$. From the twice continuous differentiability of $\varphi$ on $\mathcal{O}$ and the expression of $f$, it follows $z-\nabla^2\varphi(x)(w)\in D\partial g(x|\,v-\!\nabla\varphi(x))(w)$. Invoking Lemma \ref{relation} leads to
 \[
  \langle z-\nabla^2\varphi(x)(w),w\rangle\ge \frac{1}{2}d^2 g(x|\,v-\!\nabla\varphi(x))(w),
 \]   
 which by the positive semidefinitness of $\nabla^2\varphi(x)$ and the above \eqref{SecDerForm} implies that 
 \begin{align*}
  \langle z,w\rangle&\ge\langle w, \nabla^2\varphi(x)(w)\rangle+\frac{1}{2}d^2g(x|\,v-\!\nabla\varphi(x))(w)\\
  &\ge\frac{1}{2}\big[\langle w, \nabla^2\varphi(x)(w)\rangle+d^2g(x|\,v-\!\nabla\varphi(x))(w)\big]\\
  &\stackrel{\eqref{SecDerForm}}{=}\frac{1}{2}d^2\!f(x|v)(w)\stackrel{\eqref{d2-ineq0}}{\ge}\frac{1}{2}\ell\|w\|^2.
 \end{align*}
 That is, $\langle z,w\rangle\ge\frac{1}{2}\ell\|w\|^2$ when  $z\in D\partial f(x|v)(w)$ with $(x,v)\in{\rm gph}\,\partial\!f\cap[(\mathcal{U}\cap\mathcal{N})\times\mathcal{V}]$. By virtue of Lemma \ref{tilt-lemma1} (ii), $\overline{x}$ is a tilt-stable solution of \eqref{prob}. 
 \end{proof}

 Now we are ready to characterize the tilt stability of a local minimizer of the composite $f$ in \eqref{prob} by leveraging the second subderivative of $g$, which will be used in Section \ref{sec4} to establish the main result of this paper. 
 \begin{proposition}\label{TiltProp2}
 Let $\overline{x}$ {\color{blue}be} a local minimizer of the function $f$ in \eqref{prob}. Then $\overline{x}$ is  tilt-stable if $\nabla^2\varphi(\cdot)\succeq 0$ on an open neighborhood $\mathcal{N}$ of $\overline{x}$ and
 \(
   {\rm Ker}\,\nabla^2\varphi(\overline{x})\cap \mathcal{W}=\{0\}
 \)
 with 
 \begin{align*}
 \mathcal{W}=\big\{w\in \mathbb{X}\,|\, \exists (x^k,y^k)\in {\rm gph}\,\partial g {\;\rm and \;} w^k\in \mathbb{X} {\;\rm with \; } \lim\limits_{k\to \infty} d^2 g(x^k|y^k)(w^k)=0 \\
 {\, \rm and \,} \lim\limits_{k\to \infty}(x^k,y^k,w^k)= (\overline{x},-\nabla \varphi(\overline{x}),w)\big\}.\qquad\qquad
 \end{align*}
 Conversely, if $\overline{x}$ is a tilt-stable solution of \eqref{prob}, then ${\rm Ker}\,\nabla^2\varphi(\overline{x})\cap \mathcal{W}=\{0\}$. 
 \end{proposition}
 \begin{proof}
 Suppose on the contrary that $\overline{x}$ is not a tilt-stable solution of \eqref{prob}. By Proposition \ref{TiltProp}, there exist sequences $(x^k,v^k)\in{\rm gph}\,\partial\!f\cap[\mathcal{N}\times\mathbb{X}]$ and $w^k\in\mathbb{X}$ with $\|w^k\|\!=\!1$ such that $(x^k,v^k)\to (\overline{x},0)$ and
  \(
  d^2\!f(x^k|v^k)(w^k)<\frac{1}{k}\|w^k\|^2.
  \)
 From \eqref{SecDerForm}, for each $k\in\mathbb{N}$, 
 \begin{equation}\label{Prop2Temp}
 \langle w^k, \nabla^2\varphi(x^k)w^k\rangle+d^2 g(x^k|y^k)(w^k)<\frac{1}{k}\|w^k\|^2\ \ {\rm with}\ \ y^k:=v^k-\nabla\varphi(x^k).
 \end{equation}
 Obviously, $(x^k,y^k)\in {\rm gph}\,\partial g$ for each $k$, and $y^k\to -\nabla\varphi(\overline{x})$ as $k\to\infty$. Since $\|w^k\|\!=\!1$ for all $k\in\mathbb{N}$, if necessary by taking a subsequence, we can assume $\lim_{k\to \infty} w^k=w$. Note that $d^2 g(x^k|y^k)(w^k)\geq 0$ by the convexity of $g$, so the  inequality \eqref{Prop2Temp} implies that $\langle w^k, \nabla^2\varphi (x^k)w^k\rangle<\frac{1}{k}\|w^k\|^2$. Passing the limit $k\to \infty$ leads to $\langle w,\nabla^2 \varphi(\overline{x})w\rangle\leq 0$. This along with $\nabla^2\varphi(\overline{x})\succeq 0$ implies $w\in {\rm Ker}\,\nabla^2\varphi(\overline{x})$. On the other hand, since $x^k\in \mathcal{N}$, we have $\nabla^2\varphi(x^k)\succeq 0$, which along with the above \eqref{Prop2Temp} implies that $d^2 g(x^k|y^k)(w^k)<\frac{1}{k}\|w^k\|^2$. Passing the limit $k\to \infty$ yields $\lim_{k\to \infty} d^2 g(x^k|y^k)(w^k)\leq 0$. This along with $d^2 g(x^k|y^k)(w^k)\geq 0$ implies that $\lim_{k\to \infty} d^2 g(x^k|y^k)(w^k)= 0$. Recalling that $(x^k,y^k)\in {\rm gph}\,\partial g$ and $(x^k,y^k,w^k)\to (\overline{x},-\nabla\varphi(\overline{x}),w)$, we have  $w\in\mathcal{W}$. This, together with $w\in {\rm Ker}\,\nabla^2\varphi(\overline{x})$ and the assumption ${\rm Ker}\,\nabla^2 \varphi(\overline{x})\cap \mathcal{W}=\{0\}$, yields a contradiction to $\|w\|=1$. Hence, $\overline{x}$ is a tilt-stable solution of \eqref{prob}.
  
 Now let $\overline{x}$ {\color{blue}be} a tilt-stable solution of \eqref{prob}. Pick any $w\in{\rm Ker}\,\nabla^2\varphi(\overline{x})\cap \mathcal{W}$. From $w\in\mathcal{W}$, there exist sequences $\{(x^k,y^k)\}_{k\in\mathbb{N}}\subset {\rm gph}\,\partial g$ with $x^k\in\mathcal{N}$ and $w^k\in\mathbb{X}$ such that
 \begin{equation}\label{Prop2Temp2}
 \lim_{k\to \infty} d^2 g(x^k|y^k)(w^k)=0\ \ {\rm and}\ \
 \lim_{k\to \infty}(x^k,y^k,w^k)=(\overline{x},-\nabla \varphi(\overline{x}),w).
 \end{equation}
 Write $v^k:=y^k+\nabla\varphi(x^k)$ for each $k\in\mathbb{N}$. Then $v^k\to 0$ and $(x^k, v^k)\in {\rm gph}\,\partial f$. By the tilt-stability of $\overline{x}$, Proposition \ref{TiltProp} and \eqref{SecDerForm}, there is $\ell>0$ such that for all $k$ large enough,
 \[
  \ell \|w^k\|^2\leq d^2 g(x^k|y^k)(w^k)+\langle w^k, {\color{blue}\nabla^2}\varphi(x^k) w^k).
 \]
 Passing the limit $k\to \infty$ and using the limits in \eqref{Prop2Temp2}
 and the fact $w\!\in\! {\rm Ker}\,\nabla^2\varphi(\overline{x}) $ leads to $\ell\|w\|^2\leq \langle w,\nabla^2 \varphi(\overline{x})w\rangle=0$, so $w=0$ follows. The desired result is obtained.
 \end{proof}

 \begin{remark}\label{local-convex}
 The condition that $\nabla^2\varphi(\cdot)\succeq 0$ on an open convex neighborhood of $\overline{x}$ is equivalent to the convexity of $\varphi$ on this neighborhood (i.e., $\varphi$ is locally convex at $\overline{x}$); see \cite[Proposition B.4]{Bertsekas99} and its proof. Obviously, it is implied by  $\nabla^2\varphi(\overline{x})\succ 0$, which is a point-based and checkable condition. In addition, for many nonconvex optimization problems, it is very common that their objective functions are locally convex at a local optimal solution even a stationary point. Thus, the sufficient condition in Proposition \ref{TiltProp2} is applicable to some nonconvex composite problems.
 \end{remark}
 
 \subsection{Twice epi-differentiability of $\Phi_{\kappa}$}\label{sec2.3}

 Fix any $Z\in\mathbb{S}^{n+m}$. Let $\mu_1(Z)>\cdots>\mu_{\varsigma}(Z)$ be the distinct eigenvalues of $Z$ and define
 \begin{align}\label{thetalDef}
 \theta_l(Z)\!:=\big\{i\in[p]\ |\ \lambda_i(Z)=\mu_l(Z)\big\}\quad {\rm for\ each}\ l\in[\varsigma],
 \end{align}
 where $l_i(Z)$ denotes the number of eigenvalues that rank before $\lambda_i(Z)$ and equal $\lambda_i(Z)$ (including $\lambda_i(Z)$). That is, the eigenvalues of $Z$ take the following form 
 \begin{equation*}
 \lambda_1(Z)\geq \cdots\geq \lambda_{i-l_i(Z)}(Z)>\lambda_{i-l_i(Z)+1}(Z)=\cdots=\lambda_i(Z)\geq\cdots\geq \lambda_{n+m}(Z).
 \end{equation*}
 The following lemma characterizes the subdifferential of $\Phi_{\kappa}$ and its second subderivative, whose proofs can be found in \cite[Theorem 3.5]{Overton93} and \cite[Theorem 2.5]{Torki99}, respectively.   
 \begin{lemma}\label{Phik-lemma}
 Fix any $Z\in\mathbb{S}^{n+m}$ with $\varsigma$ distinct eigenvalues. Pick any $Q\in\mathbb{O}^{n+m}(Z)$. For each $l\in[\varsigma]$, let $\theta_{l}=\theta_{l}(Z)$ with $\theta_{l}(Z)$ defined by \eqref{thetalDef}. Let $r\in[\varsigma]$ be such that $\kappa\in\theta_{r}$. Then, the following assertions hold true.
 \begin{itemize}
  \item [(i)] The subdifferential of $\Phi_{\kappa}$ at $Z$ takes the following form
 		\begin{equation*}
 		\partial \Phi_{\kappa}(Z)
 			=\sum_{l=1}^{r-1}Q_{\theta_l} Q^{\top}_{\theta_l}+\Big\{Q_{\theta_{r}}{\rm Diag}(\xi) Q^{\top}_{\theta_{r}}\ |\ 
 			\xi\in\Omega_{r}\Big\}, 
 		\end{equation*}
 		where $\Omega_{r}:=\big\{z\in \mathbb{R}^{|\theta_{r}|}\ |\ 0\leq z_i\leq 1\ {\rm for\ each}\ i\in [|\theta_{r}|]\ {\rm and}\  \sum_{i=1}^{|\theta_{r}|}z_i=l_{\kappa}(Z) \big\}$.
 		
  \item [(ii)] The function $\Phi_{\kappa}$ is semi-differentiable at $Z$, and for any $H\in \mathbb{S}^{n+m}$,
 		\[
 		 d\Phi_k(Z)(H)=\Phi_{\kappa}'(Z;H)= \sum_{l=1}^{r-1}{\rm tr}(Q^{\top}_{\theta_{l}}H Q_{\theta_{l}})+\Phi_{l_{\kappa}(Z)}(Q^{\top}_{\theta_{r}}H Q_{\theta_{r}}).
 		\]  
 		
 \item[(iii)] $\Phi_{\kappa}$ is properly twice epi-differentiable at $Z$, and for any $S\in\partial\Phi_{\kappa}(Z)$, if a matrix $H\in\mathbb{S}^{n+m}$ is such that $\Phi_{l_{\kappa}(Z)}(Q^{\top}_{\theta_{r}}H Q_{\theta_{r}})=\langle \widetilde{S}, H\rangle$ with $\widetilde{S}\!:=S-\!\sum_{l=1}^{r-1}Q_{\theta_l} Q^{\top}_{\theta_l}$, then 
 		\[
 	    \frac{1}{2}d^2 \Phi_{\kappa}(Z|S)(H)=\sum_{l=1}^{r-1}{\rm tr}( Q^{\top}_{\theta_{l}}H(\mu_{l}(Z)I_{n+m}-Z)^{\dagger}HQ_{\theta_{l}})+\langle \widetilde{S}, 
 	    H(\mu_{r}(Z)I_{n+m}\!-\!Z)^{\dagger}H\rangle.
 	    \]
  \end{itemize}
 \end{lemma} 
 \section{Characterization of second subderivative of \texorpdfstring{$\Psi_{\!\kappa}$}{}}
    
 Recall that $\Psi_{\!\kappa}$ is the composition of $\Phi_{\kappa}$ and the mapping $\mathcal{B}$ in \eqref{Bmap}. We use the chain rule in \cite[Theorem 5.4]{MohMSSiam2020} to characterize the second subderivative of $\Psi_{\!\kappa}$. Before doing this, we introduce some notation for the subsequent analysis.
 For any $X\in\mathbb{R}^{n\times m}$, define
 \begin{equation}\label{abc}
 a(X)\!:=\!\big\{i\in[n]\ |\ \sigma_i(X)>0\big\},\
 b(X)\!:=\!\big\{i\in[n]\ |\ \sigma_i(X)\!=0\big\}
 	\ \ {\rm and}\ \ c\!:=\!\{n+1,\ldots,m\}.
 \end{equation}
 Let $\nu_1(X)>\cdots>\nu_{s(X)}(X)$ be the nonzero distinct singular values of $X$, and write
 \begin{equation}\label{al}
  a_l(X)\!:=\left\{i\in[n]\ |\ \sigma_i(X)=\nu_l(X)\right\}\ \ \forall\, 
  l\in[ s(X)]\ \ {\rm and}\ a_{s(X)+1}(X)\!:=b(X).
 \end{equation}
 For any $(U,V)\in\mathbb{O}^{n,m}(X)$, with $a=a(X),b=b(X)$, let $U^{\uparrow}_{a}:=U_a I^{\uparrow}_{|a|},
 V^{\uparrow}_{a}:=V_a I^{\uparrow}_{|a|}$,
 \begin{equation}\label{Pmatrix}
  P:=\frac{1}{\sqrt{2}}\left[\begin{matrix}
  U_{a}&U_{b} &0 &U_b &U^{\uparrow}_a\\
  V_{a}&V_{b} &\sqrt{2}V_{\!c} &-V_b &-V^{\uparrow}_a\\
  \end{matrix}\right]\ \ {\rm and}\ \  
  P_{0}:=
  \frac{1}{\sqrt{2}}\left[\begin{matrix}
   U_{b} &0 &U_b \\
  V_{b} &\sqrt{2}V_{\!c} &-V_b \\
  \end{matrix}\right].
 \end{equation}
 It is easy to check that $P\in\mathbb{O}^{n+m}$. Also, 
 from \cite[Theorem 7.3.7]{Horn90}, it follows that 
 \begin{equation}\label{BX-svd}
  P^{\top}\mathcal{B}(X)P
 =\left[\begin{matrix}
 	{\rm Diag}(\sigma(X)) & 0 & 0\\
 	0 & 0 & 0\\
 	0 & 0  & -{\rm Diag}(I^{\uparrow}_{n}\sigma(X))
 \end{matrix}\right].
 \end{equation} 
 
 Note that $\Psi_{\!\kappa}(X)\!=h_{\kappa}(\sigma(X))$ with  $h_{\kappa}(x)=\sum_{i=1}^{\kappa}|x|_i^{\downarrow}$ for $x\in\mathbb{R}^n$. Since $h_{\kappa}$ is absolutely symmetric, i.e., $h_{\kappa}(Qx)=h_{\kappa}(x)$ for any $n\times n$ signed permutation matrix $Q$, according to \cite[Corollary 2.5]{Lewis95}, $\Gamma\in\partial\Psi_{\kappa}(X)$ if and only if $\sigma(\Gamma)\in\partial h_{\kappa}(\sigma(X))$ and there exists $(U,V)\in\mathbb{O}^{n,m}(X)\cap\mathbb{O}^{n,m}(\Gamma)$, i.e., a simulatenous ordered SVD of the form 
 \begin{equation}\label{BarXGamma-SVD}
  X\!=U\big[{\rm Diag}(\sigma(X))\ \ 0\big]V^{\top}
  \ {\rm and}\  
 \Gamma\!=U\big[{\rm Diag}(\sigma(\Gamma))\ \ 0\big]V^{\top}.
 \end{equation} 
 Together with \cite[Lemma 2.3]{Wu14}, we have the following characterization for the subdifferential of $\Psi_{\kappa}$ at a point $X\in\mathbb{R}^{n\times m}$, which was also given in \cite[Lemma 3]{Ding17}.
 \begin{lemma}\label{KnorSubdifLemma}
 Consider any $X\!\in\mathbb{R}^{n\times m}$. Let $a_l=a_l(X)$ for each $l\in[s(X)\!+1]$ with $a_l(X)$ defined by \eqref{al}, and let $r\in[s(X)\!+1]$ be the integer such that $\kappa\in a_{r}$. Then, $\Gamma\!\in\partial\Psi_{\!\kappa}(X)$ if and only if the following assertions hold:
 \begin{itemize}
 \item [(i)] when $r\in[ s(X)]$, there exist integers $0\le \kappa_0\le\kappa-1$ and $\kappa\le\kappa_1\le n$ such that  
 		\begin{align*}
 		 \sigma_1(X)\geq \cdots\geq \sigma_{\kappa_0}(X)>\sigma_{\kappa_0+1}(X)=\cdots=\sigma_{\kappa}(X)=\cdots=\sigma_{\kappa_1}(X)\\
 		 >\sigma_{\kappa_1+1}(X)\ge\cdots\ge\sigma_n(X)\geq 0,\qquad\\	
 		\sigma_{\alpha}(\Gamma)=e_{\alpha},\ {\textstyle\sum_{i\in \beta }}\sigma_i(\Gamma)=\kappa-\kappa_0\ {\rm with}\  0\leq \sigma_{\beta}(\Gamma)\leq e_{\beta}, \ \sigma_{\gamma}(\Gamma)=0,
 		\end{align*}
 		where $\alpha\!:=\![\kappa_0],\, \beta\!:=\!\{\kappa_0\!+\!1,\ldots,\kappa_1\}=a_{r}, \gamma\!:=\{\kappa_1\!+\!1,\ldots,n\}=\bigcup_{l=r+1}^{ s(X)+1}a_l$; 
 		
 \item [(ii)] when $r= s(X)\!+\!1$, there exists an integer $\kappa_0$ with $0\le\kappa_0\le\kappa-1$ such that  
 		\begin{align*}
 		\sigma_1(X)\geq \ldots\geq \sigma_{\kappa_0}(X)>\sigma_{\kappa_0+1}(X)=\cdots=\sigma_{\kappa}(X)=\cdots=\sigma_n(X)=0,\\
 		\sigma_{\alpha}(\Gamma)=e_{\alpha}\ \ {\rm and} \ \ {\textstyle\sum_{i\in \beta}}\sigma_{i}(\Gamma)\leq\kappa-\kappa_0\ \ {\rm with}\ \ 0\leq \sigma_{\beta}(\Gamma)\leq e_{\beta},\quad
 		\end{align*}
 		where $\alpha:=[\kappa_0]=\bigcup_{l=1}^{r-1}a_{l}$ and $\beta:=\{\kappa_0\!+1,\ldots,n\}=b(X)=a_{s(X)+1}$.
 	\end{itemize}
 \end{lemma}  
 
 For any $(X,\Gamma)\in{\rm gph}\,\partial\Psi_{\!\kappa}$, define the multiplier set associated with $(X,\Gamma)$ by
 \begin{equation*}
  \Lambda(X,\Gamma)
  :=\Big\{M\in\partial\Phi_{\!\kappa}(\mathcal{B}(X))\ |\ \mathcal{B}^*(M)=\Gamma\Big\},
 \end{equation*} 
 where $\mathcal{B}^*\!:\mathbb{S}^{m+n}\to\mathbb{R}^{n\times m}$ stands for the ajoint of the linear operator $\mathcal{B}$. The following lemma provides a specific characterization for such a multiplier set.
 \begin{lemma}\label{multiplier-lemma}
 Fix any $X\in\mathbb{R}^{n\times m}$ and  $\Gamma\in\partial\Psi_{\!\kappa}(X)$. Pick any $(U,V)\in\mathbb{O}^{n,m}(X)$, and let $P$ and $P_0$ be the matrices defined by \eqref{Pmatrix} with such $(U,V)$ and the index sets $a=a(X)$, $b=b(X)$ and $c$ from \eqref{abc}. Let $a_l=a_l(X)$ for each $l\in[ s(X)+1]$ with $a_l(X)$ defined by \eqref{al}, and let $r\in[ s(X)\!+\!1]$ be the integer such that $\kappa\in a_{r}$. Then, $M\in \Lambda(X,\Gamma)$ iff $M=\frac{1}{2}\begin{pmatrix}
 2M_{11} & \Gamma\\
 \Gamma^{\top}&2M_{22}
 \end{pmatrix}$ with $M_{11}\in\mathbb{S}^{n}$ and $M_{22}\in\mathbb{S}^m$ and there exists $\xi\in\Omega_{r}$ such that
 \begin{align}\label{aim-equa}
 M-{\textstyle\sum_{l=1}^{r-1}} P_{\!a_l} P^{\top}_{\!a_l}=\left\{\begin{array}{cl} P_{\!a_{r}}{\rm Diag}(\xi) P^{\top}_{\!a_{r}} &{\rm if}\ r\in [s(X)], \\
 P_{0}{\rm Diag}(\xi) P^{\top}_{0} & {\rm if}\ r= s(X)\!+\!1,
 \end{array}\right.
 \end{align}
 where $\Omega_{r}=\big\{z\in \mathbb{R}^{|a_{r}|}\ |\ 0\leq z_i\leq 1\ {\rm for\ each}\ i\in [|a_{r}|]\ {\rm and}\  \sum_{i=1}^{|a_{r}|}z_i=l_{\kappa}(\mathcal{B}(X)) \big\}$.
 When $r\in[s(X)]$, the above equality \eqref{aim-equa} can equivalently be written as 
 \begin{subnumcases}{}
  U_{\!a_{r}}{\rm Diag}(\xi)U_{\!a_{r}}^{\top}
  =2M_{11}-{\textstyle\sum_{l=1}^{r-1}}\,U_{\!a_l}U_{\!a_l}^{\top},\nonumber\\
  V_{\!a_{r}}{\rm Diag}(\xi)V^{\top}_{\!a_{r}}
   =2M_{22}-{\textstyle\sum_{l=1}^{r-1}}\,V_{\!a_l}V_{\!a_l}^{\top},\nonumber\\
  U_{\!a_{r}}{\rm Diag}(\xi)V^{\top}_{\!a_{r}}=\Gamma-{\textstyle\sum_{l=1}^{r-1}}\,U_{\!a_l}V_{\!a_l}^{\top};\nonumber
  \end{subnumcases}
 when $r= s(X)\!+\!1$, by writing $\xi:=(\xi_1;\xi_2;\xi_3)$ with $\xi_1,\xi_3\in \mathbb{R}^{|b|}$ and $\xi_2\in \mathbb{R}^{|c|}$, the above equality \eqref{aim-equa} can equivalently be written as 
 \begin{subnumcases}{}
 U_{b}{\rm Diag}(\xi_1+\xi_3)U^{\top}_{\!b}
 = 2M_{11}-{\textstyle\sum_{l=1}^{r-1}}\,U_{\!a_l}U_{\!a_l}^{\top},\nonumber\\ V_{b}{\rm Diag}(\xi_1+\xi_3)V^{\top}_{b}+ 2 V_{\!c}{\rm Diag}(\xi_2)V_{\!c}^\top
 =2M_{22}-{\textstyle\sum_{l=1}^{r-1}}V_{\!a_l}V_{\!a_l}^{\top},\nonumber\\
 U_{b}{\rm Diag}(\xi_1-\xi_3)V^{\top}_{b}=\Gamma-{\textstyle\sum_{l=1}^{r-1}}\,U_{\!a_l}V_{\!a_l}^{\top}.\nonumber
 \end{subnumcases}
 \end{lemma}
 \begin{proof}
 From the definition of the linear mapping $\mathcal{B}$ in \eqref{Bmap}, it is not difficult to obtain 
 \[
   \mathcal{B}^*(M)=2M_{12}\ \ {\rm for}\ M=\begin{pmatrix}
   	M_{11} & M_{12}\\
   	M_{12}^{\top}&M_{22}
   \end{pmatrix}\ {\rm with}\ M_{11}\in\mathbb{S}^n,M_{22}\in\mathbb{S}^m,M_{12}\in\mathbb{R}^{n\times m}.
 \]
 Together with the defintion of $\Lambda(X,\Gamma)$, $M\in\Lambda(X,\Gamma)$ if and only if $M\in\partial \Phi_{\kappa}(\mathcal{B}(X))$ with	
 $M=\begin{pmatrix}
 M_{11} & \Gamma/2\\
  (\Gamma/2)^{\top}&M_{22}
 \end{pmatrix}$ for $M_{11}\!\in\!\mathbb{S}^{n}$ and $M_{22}\in\mathbb{S}^m$. By Lemma \ref{Phik-lemma} (i), $M\in\partial\Phi_{\kappa}(\mathcal{B}(X))$ iff there exists $\xi\in\Omega_{r}$ such that \eqref{aim-equa} holds. The first part then follows. The second part is immediate by using equality \eqref{aim-equa} and the expressions of $P$ and $P_0$ in \eqref{Pmatrix}. 
 \end{proof}

 Now we are ready to give a compact expression of the second subderivative of $\Psi_{\!\kappa}$.
 \begin{proposition}\label{tepi-diff}
  Fix any $X\!\in\mathbb{R}^{n\times m}$ and  $\Gamma\!\in\partial\Psi_{\kappa}(X)$. Pick any $(U,V)\in\mathbb{O}^{n,m}(X)$, and let $P$ be defined by \eqref{Pmatrix} with such $(U,V)$ and the index sets $a=a(X),b=b(X)$ and $c$ from \eqref{abc}. Let $a_l=a_l(X)$ for each $l\in[s(X)+1]$ with $a_{l}(X)$ defined by \eqref{al}, and let $r\in[ s(X)+1]$ be the integer such that $\kappa\in a_{r}$. For each $l\in[ s(X)]$, define
  \begin{equation}\label{Xial} \Xi_{a_l}(X,\cdot):=2P^{\top}_{\!a_{l}}\mathcal{B}(\cdot)(\nu_{l}(X)I_{n+m}\!-\!\mathcal{B}(X))^{\dagger}\mathcal{B}(\cdot)P_{\!a_{l}}.
  \end{equation}  
  Then, $\Psi_{\kappa}$ is properly twice epi-differentiable at $X$ for $\Gamma$ 
  with ${\rm dom\,}d^2\Psi_{\kappa}(X|\Gamma)=\mathcal{C}_{\Psi_{\kappa}}(X,\Gamma)$, and moreover, for any $G\in \mathcal{C}_{\Psi_{\kappa}}(X,\Gamma)$, when $r\in[ s(X)]$, 
 \begin{equation}\label{KNormSubDer1}
  \!d^2\Psi_{\kappa}(X|\Gamma)(G)
  =\!\sum_{l=1}^{r-1}{\rm tr}\big(\Xi_{a_{l}}(X,G)\big)+\Big\langle U_{\!a_{r}}^{\top}\Big(\Gamma-\!\sum_{l=1}^{r-1} U_{\!a_l}V^{\top}_{\!a_l}\Big)V_{\!a_{r}},\Xi_{a_{r}}(X,G)\Big\rangle,
 \end{equation}  
 and when $r= s(X)+1$, with $\Sigma_a={\rm Diag}(\sigma_{a(X)}(X))$ it holds
 \begin{equation}\label{KNormSubDer2}
  d^2\Psi_{\kappa}(X|\Gamma)(G)
  =\!\sum_{l=1}^{r-1}{\rm tr}\big(\Xi_{a_{l}}(X,G)\big)-
  2\Big\langle\Gamma-\sum_{l=1}^{r-1}U_{\!a_l}V_{\!a_l}^\top,GV_{\!a} \Sigma_a^{-1}U^\top_{\!a}G\Big\rangle.
 \end{equation}
 \end{proposition} 
 \begin{proof}
By \cite[Proposition 2.2]{Torki01}, the function $\Phi_{\kappa}$ is parabolically epi-differentiable at any $Z$ for any $H\in \mathbb{S}^{n+m}$. Note that the constraint system $\mathcal{B}(\cdot)\in{\rm dom}\,\Phi_{\kappa}$ satisfies the metric subregularity constraint qualification at any feasible point. From \cite[Proposition 5.1]{MohMSSiam2020}, $d^2\Psi_{\kappa}(X|\Gamma)$ is a proper lsc function with
 ${\rm dom}\,d^2\Psi_{\kappa}(X|\Gamma)=\mathcal{C}_{\Psi_{\kappa}}(X,\Gamma)$. In addition, from the proof of \cite[Theorem 2.4]{Torki99} (see also \cite[Example 3.3]{MohMSSiam2020}), the function $\Phi_{\kappa}$ is parabolically regular at any $Z\in\mathbb{S}^{n+m}$ for every $W\in\partial\Phi_{\kappa}(Z)$. Invoking \cite[Corollary 5.5]{MohMSSiam2020} with $g=\Phi_{\kappa}$ and $F(\cdot)=\mathcal{B}(\cdot)$ leads to the twice epi-differentiability $\Psi_{\kappa}$ at $X$ for $\Gamma$. 
 Thus, the rest only needs to achieve the expression of $d^2\Psi_{\kappa}(X|\Gamma)(G)$ for $G\in\mathcal{C}_{\Psi_{\kappa}}(X,\Gamma)$.   
 Pick any $G\in\mathcal{C}_{\Psi_{\kappa}}(X,\Gamma)$. Using \cite[Theorem 5.4]{MohMSSiam2020} with $g=\Phi_{\kappa}$ and $F(\cdot)=\mathcal{B}(\cdot)$ leads to
 \begin{equation}\label{d2NuNorm}
 d^2\Psi_{\kappa}(X|\Gamma)(G)
 \!=\!\max_{M'\in\Lambda(X,\Gamma)}d^2 \Phi_{\kappa}(Z|M')(\mathcal{B}(G))\ \ {\rm with}\ Z=\mathcal{B}(X).
 \end{equation}
 Let $M\in\Lambda(X,\Gamma)$ be an optimal solution of the maximum problem in \eqref{d2NuNorm}. Then, 
 \begin{equation}\label{d2Phik}
  d^2\Psi_{\kappa}(X|\Gamma)(G)
  =d^2 \Phi_{\kappa}(Z|M)(\mathcal{B}(G)).
 \end{equation}
 {\bf Case 1: $r\in[ s(X)]$.}
 Since $M\in\Lambda(X|\Gamma)$, by Lemma \ref{multiplier-lemma}, there exists $\xi\in\Omega_{r}$ such that 
 \begin{equation}\label{temp-equa41}
 P_{\!a_{r}}{\rm Diag}(\xi)P^{\top}_{\!a_{r}}=M-\sum_{l=1}^{r-1}P_{\!a_l} P^{\top}_{\!a_l}\ \ {\rm and}\ \ U_{\!a_{r}}{\rm Diag}(\xi)V^{\top}_{\!a_{r}}=\Gamma-\sum_{l=1}^{r-1}U_{\!a_l}V_{\!a_l}^{\top}. 
 \end{equation}
 From $\langle\Gamma,G\rangle 
 =d\Psi_{\kappa}(X)(G)$ and the second equality in \eqref{temp-equa41}, it follows that 
 \begin{align}\label{temp-equa42}
 \Big\langle\sum_{l=1}^{r-1}U_{\!a_l} V^{\top}_{\!a_l}+U_{\!a_{r}}{\rm Diag}(\xi) V^{\top}_{\!a_{r}},G\Big\rangle 
 &\!=\!\langle\Gamma,G\rangle\!=\!d\Psi_{\kappa}(X)(G)\!=\! d\Phi_{\kappa}(\mathcal{B}(X))(\mathcal{B}(G))\nonumber\\
 &=\sum_{l=1}^{r-1}{\rm tr}(P^{\top}_{\!a_{l}}\mathcal{B}(G) P_{\!a_{l}})\!+\!\Phi_{l_{\kappa}(Z)}(P^{\top}_{\!a_{r}}\mathcal{B}(G)P_{\!a_{r}}),
 \end{align}
 where the third equality is obtained by using \cite[Proposition 4.3]{MohMSSiam2020} and the last one is due to Lemma \ref{Phik-lemma} (ii). By the expression of $P$ in \eqref{Pmatrix}, $\langle \sum_{l=1}^{r-1}U_{\!a_l} V^{\top}_{\!a_l},G\rangle=\sum_{l=1}^{r-1}{\rm tr}(P^{\top}_{\!a_{l}}\mathcal{B}(G)P_{\!a_{l}})$ and $\langle U_{\!a_{r}}{\rm Diag}(\xi) V^{\top}_{\!a_{r}},G\rangle=\langle P_{\!a_{r}}{\rm Diag}(\xi)P^{\top}_{\!a_{r}}, \mathcal{B}(G) \rangle$. Then equation \eqref{temp-equa42} is equivalent to   
 \begin{equation*}
 \Phi_{l_k(Z)}(P^{\top}_{\!a_{r}}\mathcal{B}(G)P_{\!a_{r}})=\langle P_{\!a_{r}}{\rm Diag}(\xi)P^{\top}_{\!a_{r}}, \mathcal{B}(G) \rangle.
 \end{equation*}
 Together with the first equality of \eqref{temp-equa41}, we conclude that $H=\mathcal{B}(G)$ satisfies the condition of  Lemma \ref{Phik-lemma} (iii) with $Q=P$ and $\widetilde{S}=P_{\!a_{r}}{\rm Diag}(\xi)P^{\top}_{\!a_{r}}$. Then, by noting that $\mu_{l}(Z)$ for each $l\in[ s(X)]$ is precisely $\nu_{l}(X)$, from Lemma \ref{Phik-lemma} (iii) it follows   
 \begin{align*}
  d^2 \Phi_{\kappa}(Z|M)(\mathcal{B}(G))
  &=\sum_{l=1}^{r-1}{\rm tr}(2 P^{\top}_{\!a_{l}}\mathcal{B}(G)(\nu_{l}(X)I_{n+m}\!-\!\mathcal{B}(X))^{\dagger}\mathcal{B}(G)P_{\!a_{l}})\nonumber\\
 &\quad+\langle{\rm Diag}(\xi),2P_{\!a_{r}}^{\top}\mathcal{B}(G)(\nu_{r}(X)I_{n+m}\!-\!\mathcal{B}(X))^{\dagger}\mathcal{B}(G)P_{\!a_{r}}\rangle.
 \end{align*}
 Note that ${\rm Diag}(\xi)=U_{\!a_{r}}^{\top}(\Gamma-\!\sum_{l=1}^{r-1}U_{\!a_l}V_{\!a_l}^{\top})V_{\!a_{r}}$ by the second equality of \eqref{temp-equa41}. Along with the above equality and equation \eqref{d2Phik}, we obtain the desired result.  

 \noindent
 {\bf Case 2: $r= s(X)+\!1$.} Since $M\in\Lambda(X,\Gamma)$, according to Lemma \ref{multiplier-lemma}, there exists a vector $\xi=(\xi_1;\xi_2;\xi_3)\in\Omega_{r}$ with $\xi_1,\xi_3\in \mathbb{R}^{|b|}$ and $\xi_2\in \mathbb{R}^{|c|}$ such that 
 \begin{equation}\label{MstruCase2}
 P_{0}{\rm Diag}(\xi)P^{\top}_{0}=M-\sum_{l=1}^{r-1}\!P_{\!a_l}P^{\top}_{\!a_l}\ \ {\rm and} \ \ 
 U_{b}{\rm Diag}(\xi_1\!-\!\xi_3)V^{\top}_{\!b}=\Gamma-\!\sum_{l=1}^{r-1}U_{\!a_l}V_{\!a_l}^{\top}.
 \end{equation}
 Using the second equality in \eqref{MstruCase2} and the similar arguments as above leads to
 \begin{equation}\label{Case2temp-equa42}
 \Big\langle\sum_{l=1}^{r-1}U_{\!a_l} V^{\top}_{\!a_l}\!+U_{b}{\rm Diag}(\xi_1\!-\!\xi_3)V^{\top}_{b},G\Big\rangle=\!\sum_{l=1}^{r-1}{\rm tr}(P^{\top}_{\!a_{l}}\mathcal{B}(G) P_{\!a_{l}})\!+\Phi_{l_{\kappa}(Z)}(P_{\!a_{r}}^{\top}\mathcal{B}(G)P_{\!a_{r}}).
 \end{equation}
 By the expressions of $P$ and $P_0$ in equation \eqref{Pmatrix}, $\langle \sum_{l=1}^{r-1}U_{\!a_l} V^{\top}_{\!a_l},G\rangle=\sum_{l=1}^{r-1}{\rm tr}(P^{\top}_{\!a_{l}}\mathcal{B}(G)P_{\!a_{l}})$ and $\langle U_{b}{\rm Diag}(\xi_1\!-\!\xi_3)V^{\top}_{\!b},G\rangle=\langle P_{0}{\rm Diag}(\xi)P^{\top}_{0}, \mathcal{B}(G) \rangle$. Then equation \eqref{Case2temp-equa42} is equivalent to   
 \begin{equation*}
 \Phi_{l_{\kappa}(Z)}(P_{\!a_{r}}^{\top}\mathcal{B}(G)P_{\!a_{r}})=\langle P_{0}{\rm Diag}(\xi)P_{0}^{\top}, \mathcal{B}(G)\rangle.
 \end{equation*}
 Along with the first equality in \eqref{MstruCase2}, we conclude that $H=\mathcal{B}(G)$ satisfies the condition of Lemma \ref{Phik-lemma} (iii) with $Q=P$ and $\widetilde{S}=P_{0}{\rm Diag}(\xi)P^{\top}_{0}$. Then, by noting that $\mu_{l}(Z)$ for each $l\in[s(X)\!+1]$ is precisely $\nu_{l}(X)$, from Lemma \ref{Phik-lemma} (iii) it follows  
 \begin{align*}
 d^2\Phi_{\kappa}(Z|M)(\mathcal{B}(G))&\!=\!\sum_{l=1}^{r-1}{\rm tr}(2 P^{\top}_{\!a_{l}}\mathcal{B}(G)(\nu_{l}(X)I_{n+m}-\mathcal{B}(X))^{\dagger}\mathcal{B}(G)P_{\!a_{l}})\\
 &\quad-\langle {\rm Diag}(\xi),P_{0}^{\top}\mathcal{B}(G)(\mathcal{B}(X))^{\dagger}\mathcal{B}(G)P_{0}\rangle.
 \end{align*}
 Substituting the expression of $P_0$ into the second term on the right hand side, we have
 \begin{align*}
  &\langle {\rm Diag}(\xi),P^{\top}_{0}\mathcal{B}(G)(\mathcal{B}(X))^{\dagger}\mathcal{B}(G)P_{0}\rangle \nonumber\\
  &=\langle {\rm Diag}(\xi_1-\xi_3),U_{\!b}^\top GV_{\!a}\Sigma_a^{-1}U^\top_{\!a} G V_{\!b}+V^\top_{\!b} G^\top U_a \Sigma_a^{-1}V^\top_{\!a} G^\top U_{\!b} \rangle\\
  &=2 \langle U_{b} {\rm Diag}(\xi_1-\xi_3)V^\top_{b}, GV_{\!a} \Sigma_a^{-1}U^\top_{\!a} G \rangle
 =2\Big\langle\Gamma-\sum_{l=1}^{r-1}U_{\!a_l}V_{\!a_l}^\top,G V_{\!a}\Sigma_a^{-1}U^\top_{\!a} G\Big\rangle
 \end{align*}
 where the last equality is due to the second equality in \eqref{MstruCase2}. 
 Combining the above two equations with \eqref{d2Phik} yields the desired result. The proof is completed.
 \end{proof}

When $(U,V)\in\mathbb{O}^{n,m}(X)$ in Proposition \ref{tepi-diff} also satisfies $(U,V)\in\mathbb{O}^{n,m}(\Gamma)$ (see \cite[Corollary 2.5]{Lewis95} for the  existence of such $(U,V)$), the second subderivative of $\Psi_{\kappa}$ has the expression as stated in the following proposition.
 \begin{proposition}\label{MainTh2}
 Fix any $X\in\mathbb{R}^{n\times m}$ and $\Gamma\in\partial\Psi_{\kappa}(X)$. Pick any $(U,V)\in\mathbb{O}^{n,m}(X)\cap\mathbb{O}^{n,m}(\Gamma)$ with $V=[V_{1}\ \ V_{\!c}]$ for $V_{1}\in\mathbb{O}^{m\times n}$. Let $P$ be the  matrix defined by \eqref{Pmatrix} with such $(U,V)$ and $a=a(X),b=b(X)$ and $c$ from \eqref{abc}, let $a_l=a_l(X)$ for each $l\in[ s(X)+1]$ with $a_{l}(X)$ defined by \eqref{al}, and let $r\in[ s(X)\!+\!1]$ be the integer such that $\kappa\in a_{r}$. Let $\zeta_1(\Gamma)>\cdots>\zeta_{q}(\Gamma)$ be the nonzero distinct entries of the set $\big\{\sigma_i(\Gamma)\ |\ i\in a_{r}\big\}$, and write
 \begin{equation}\label{beta_l}
 \beta_l(\Gamma)\!:=\big\{i\in a_{r}\,|\,\sigma_i(\Gamma)=\zeta_{l}(\Gamma)\big\}\ \ {\rm for}\ l\in[q]\ \ {\rm and}\ \  \beta_{0}(\Gamma)\!:=\big\{i\in a_{r}\,|\,\sigma_i(\Gamma)=0\big\}.
 \end{equation}
 Then, for any $G\in\mathcal{C}_{\Psi_{\kappa}}(X,\Gamma)$, 
 when $r\in[ s(X)]$,  $d^2\Psi_{\kappa}(X|\Gamma)(G)$ is equal to  
 \begin{align*}
 &\sum_{l=1}^{r-1}\sum_{l'=r+1}^{ s(X)+1}
 \frac{2\|[\mathcal{S}(U^{\top}\!G V_{1})]_{a_{l}a_{l'}}\|_F^2}{\nu_l(X)\!-\!\nu_{l'}(X)}+\sum_{l=1}^{r-1}\sum_{j=1}^{q}
 \frac{2(1\!-\!\zeta_j(\Gamma))}{\nu_l(X)\!-\!\nu_{r}(X)}\| [\mathcal{S}(U^{\top}\!G V_{1})]_{a_{l}\beta_{j}}\|_F^2\nonumber\\
 & +\sum_{j=1}^{q}\sum_{l'=r+1}^{ s(X)+1}\frac{2\zeta_{j}(\Gamma)\| [\mathcal{S}(\overline{U}^{\top}\! G\overline{V}_{1})]_{\beta_{j}a_{l'}}\|_F^2}{\nu_{r}(X)\!-\!\nu_{l'}(X)}+\sum_{l=1}^{r-1}\sum_{l'=1}^{ s(X)+1}\frac{2\|[\mathcal{T}(U^{\top}\!G V_{1})]_{a_{l}a_{l'}}\|_F^2}{\nu_l(X)\!+\!\nu_{l'}(X)} \nonumber\\
 & +\sum_{j=1}^{q}\sum_{l'=1}^{ s(X)+1}
 	\frac{2\zeta_{j}(\Gamma)\|[\mathcal{T}(U^{\top}\! GV_{1})]_{\beta_ja_{l'}}\|_F^2}{\nu_{r}(X)+\nu_{l'}(X)}+\sum_{j=1}^{q}\frac{\zeta_{j}(\Gamma)}{\nu_{r}(X)}\|U^\top_{\!\beta_j}GV_{c}\|_F^2
 	\nonumber\\
 &+\sum_{l=1}^{r-1}\frac{\|U_{\!a_l}^\top GV_{\!c}\|_F^2}{\nu_l(X)}+\sum_{l=1}^{r-1}
 	\frac{2\|[\mathcal{S}(U^{\top}\!G V_{1})]_{a_{l}\beta_{0}}\|_F^2}{\nu_l(X)\!-\!\nu_{r}(X)};
 \end{align*} 
 when $r= s(X)\!+\!1$, $d^2\Psi_{\kappa}(X|\Gamma)(G)$ is equal to  the following sum
 \begin{align*}
 &\sum_{l=1}^{r-1}\sum_{j=1}^{q}\frac{2(1\!-\!\zeta_{j}(\Gamma))}{\nu_l(X)}\| [\mathcal{S}(U^{\top}\!GV_{1})]_{a_{l}\beta_{j}}\|_F^2+\sum_{l=1}^{r-1}\sum_{l'=1}^{r-1}
 \frac{2\|[\mathcal{T}(U^{\top}\!GV_{1})]_{a_{l}a_{l'}}\|_F^2}{\nu_l(X)\!+\!\nu_{l'}(X)}\nonumber\\
 &+\!\sum_{l=1}^{r-1}\!\frac{2\| [\mathcal{S}(U^{\top}\!GV_{\!1})]_{a_{l}\beta_{0}}\|_F^2}{\nu_l(X)}+\!\sum_{l=1}^{r-1}\frac{\|U^\top_{a_l}GV_{c}\|_F^2}{\nu_l(X)}+\!\sum_{l=1}^{r-1}\sum_{j=0}^{q}
 \frac{2(1\!+\!\zeta_{j}(\Gamma))}{\nu_l(X)}\|[\mathcal{T}(U^{\top}\!GV_{1})]_{a_{l}\beta_{j}}\|_F^2. 	
 \end{align*} 
\end{proposition} 
\begin{proof}
 We first consider the case $r\in s(X)$. Fix any $G\!\in\mathcal{C}_{\Psi_{\!\kappa}}(X,\Gamma)$. For each $l\in[ s(X)]$, by the definition of $\Xi_{a_l}$ in \eqref{Xial} and the eigenvalue decomposition of $\mathcal{B}(X)$ in \eqref{BX-svd}, we get
 \begin{align*}
  \Xi_{a_l}(X,G)&=2[\mathcal{S}(U^{\top}\!GV_{1})]_{a_l}^\top (\nu_l(X)I_{n}-{\rm Diag}(\sigma(X)))^{\dagger}[\mathcal{S}(U^{\top}\!G {V_{1}})]_{a_l}\\
  &\quad+2[\mathcal{T}(U^{\top}\!GV_{1})]_{a_l}^\top (\nu_l(X) I_{n}+{\rm Diag}(\sigma(X)))^{\dagger} [\mathcal{T}(U^{\top}\!GV_{1})]_{a_l}\\
  &\quad+\!\frac{1}{\nu_l(X)}U_{\!a_l}^\top GV_{\!c} V_{\!c}^{\top}G^{\top}U^\top_{\!a_l},
 \end{align*}
 which implies that  
 \begin{align}\label{GammTraceTemp1}
 \sum_{l=1}^{r}{\rm tr}\big(\Xi_{a_l}(X,G)\big)
 &=\sum_{l=1}^{r}\sum_{l\ne l'=1}^{ s(X)+1}
   \frac{2\|[\mathcal{S}(U^{\top}\!G V_{1})]_{a_{l}a_{l'}}\|_F^2}{\nu_l(X)-\nu_{l'}(X)}+\sum_{l=1}^{r}\frac{1}{\nu_l(X)}\|U_{\!a_l}^\top GV_{\!c}\|_F^2\nonumber\\
 &\quad+\sum_{l=1}^{r}\sum_{l'=1}^{ s(X)+1}
	\frac{2\|[\mathcal{T}(U^{\top}\!G V_{1})]_{a_{l}a_{l'}}\|_F^2}{\nu_l(X)+\nu_{l'}(X)}.
 \end{align} 
 From $(U,V)\in\mathbb{O}^{n,m}(X)\cap\mathbb{O}^{n,m}(\Gamma)$, we have $U_{\!a_{r}}^{\top}\big(\Gamma\!-\!\sum_{l=1}^{r-1} U_{\!a_l}V^{\top}_{\!a_l}\big)V_{\!a_{r}}={\rm Diag}(\sigma_{\!a_{ r}}(\Gamma))$, so
 \begin{align*}
 &\big\langle U_{\!a_{ r}}^{\top}\big(\Gamma-{\textstyle\sum_{l=1}^{ r-1}} U_{\!a_l}V^{\top}_{\!a_l}\big)V_{\!a_{ r}}, \Xi_{a_{ r}}(X,G)\rangle\nonumber\\	
 &=2\langle{\rm Diag}(\sigma_{\!a_{ r}}(\Gamma)),  P^{\top}_{\!a_r}\mathcal{B}(G)(\nu_r(X) I_{n+m}-\mathcal{B}(X))^{\dagger}\mathcal{B}(G) P_{\!a_{ r}}\rangle \nonumber \\
 &=\sum_{j=1}^{q}\zeta_{j}(\Gamma){\rm tr}(2 P^{\top}_{\!\beta_j}\mathcal{B}(G)(\nu_r(X)I_{n+m}-\mathcal{B}(X))^{\dagger}\mathcal{B}(G)P_{\!\beta_j})\nonumber\\
 &=\sum_{j=1}^{q}\sum_{r\ne l'=1}^{ s(X)+1}
	\frac{2\zeta_j(\Gamma)\| [\mathcal{S}(U^{\top}\! GV_{1})]_{\beta_{j}a_{l'}}\|_F^2}{\nu_r(X)-\nu_{l'}(X)}
   +\sum_{j=1}^{q}\frac{\zeta_{j}(\Gamma)}{\nu_{ r}(X)}\|U^\top_{\!\beta_j}GV_{c}\|_F^2\nonumber\\
 &\quad+\sum_{j=1}^{q}\sum_{l'=1}^{ s(X)+1}
  \frac{2	\zeta_{j}(\Gamma)\|[\mathcal{T}(U^{\top}\! GV_{1})]_{\beta_ja_{l'}}\|_F^2}{\nu_r(X)+\nu_{l'}(X)}.
 \end{align*} 
 Combining the above two equations with the previous \eqref{KNormSubDer1} leads to
 \begin{align}\label{temp-equa31}
 d^2\Psi_{\kappa}(X|\Gamma)(G)
 &=\!\sum_{l=1}^{ r-1}\sum_{l\ne l'=1}^{ s(X)+1}
	\frac{2\|[\mathcal{S}(U^{\top}\!G V_{1})]_{a_{l}a_{l'}}\|_F^2}{\nu_l(X)\!-\!\nu_{l'}(X)}+\!\sum_{j=1}^{q}\sum_{r\ne l'=1}^{ s(X)+1}\frac{2\zeta_{j}(\Gamma)\| [\mathcal{S}(U^{\top}\! GV_{1})]_{\beta_{j}a_{l'}}\|_F^2}{\nu_r(X)\!-\!\nu_{l'}(X)}\nonumber\\
 &\quad +\sum_{l=1}^{ r-1}\sum_{l'=1}^{ s(X)+1}
		\frac{2\|[\mathcal{T}(U^{\top}\!G V_{\!1})]_{a_{l}a_{l'}}\|_F^2}{\nu_l(X)+\nu_{l'}(X)}+\sum_{j=1}^{q}\frac{\zeta_j(\Gamma)}{\nu_r(X)}\|U^\top_{\!\beta_j}GV_{c}\|_F^2
		\nonumber\\
 &\quad+\sum_{l=1}^{ r-1}\frac{1}{\nu_l(X)}\|U_{\!a_l}^\top GV_{c}\|_F^2+\sum_{j=1}^{q}\sum_{l'=1}^{ s(X)+1}
		\frac{2	\zeta_{j}(\Gamma)\| [\mathcal{T}(U^{\top}\! GV_{1})]_{\beta_ja_{l'}}\|_F^2}{\nu_r(X)+\nu_{l'}(X)}.
 \end{align}
 The sum of the first two terms on the right hand side of \eqref{temp-equa31} is equal to 
 \begin{align*}
 &\sum_{l=1}^{ r-1}
 \frac{2\|[\mathcal{S}(U^{\top}\!G V_{\!1})]_{a_{l}a_{ r}}\|_F^2}{\nu_l(X)-\nu_{ r}(X)}+\sum_{l=1}^{ r-1}\sum_{l'= r+1}^{ s(X)+1}
 \frac{2\|[\mathcal{S}(U^{\top}\!G V_{\!1})]_{a_{l}a_{l'}}\|_F^2}{\nu_l(X)-\nu_{l'}(X)} \\
 &-\sum_{j=1}^{q}\sum_{l'=1}^{ r-1}
 \frac{2\zeta_{j}(\Gamma)\| [\mathcal{S}(U^{\top}\! GV_{1})]_{\beta_{j}a_{l'}}\|_F^2}{\nu_{l'}(X)-\nu_{ r}(X)}+\sum_{j=1}^{q}\sum_{l'= r+1}^{
  s(X)+1}\frac{2	\zeta_j(\Gamma)\| [\mathcal{S}(U^{\top}\! GV_{\!1})]_{\beta_{j}a_{l'}}\|_F^2}{\nu_{ r}(X)-\nu_{l'}(X)}\\
 &=\sum_{l=1}^{r-1}\sum_{j=0}^{q}
 \frac{2\|[\mathcal{S}(U^{\top}\!G V_{\!1})]_{a_{l}\beta_{j}}\|_F^2}{\nu_l(X)-\nu_{r}(X)}+\sum_{l=1}^{r-1}\sum_{l'=r+1}^{s(X)+1}
 \frac{2\|[\mathcal{S}(U^{\top}\!GV_{1})]_{a_{l}a_{l'}}\|_F^2 }{\nu_l(X)-\nu_{l'}(X)}\\
 &\quad-\sum_{j=1}^{q}\sum_{l'=1}^{ r-1}
 \frac{2\zeta_{j}(\Gamma)\| [\mathcal{S}(U^{\top}\! GV_{1})]_{\beta_{j}a_{l'}}\|_F^2}{\nu_{l'}(X)-\nu_{ r}(X)}+\sum_{j=1}^{q}\sum_{l'= r+1}^{
 	s(X)+1}\frac{2	\zeta_j(\Gamma)\| [\mathcal{S}(U^{\top}\! GV_{\!1})]_{\beta_{j}a_{l'}}\|_F^2}{\nu_{ r}(X)-\nu_{l'}(X)} \\
 &=\sum_{l=1}^{r-1}\sum_{j=1}^{q}
 \frac{2(1-\zeta_j(\Gamma))\|[\mathcal{S}(U^{\top}\!G V_{1})]_{a_{l}\beta_{j}}\|_F^2}{\nu_l(X)-\nu_{r}(X)}
 +\sum_{l=1}^{r-1}\sum_{l'=r+1}^{s(X)+1}
 \frac{2\|[\mathcal{S}(U^{\top}\!GV_{1})]_{a_{l}a_{l'}}\|_F^2 }{\nu_l(X)-\nu_{l'}(X)}\\
 &\quad+\sum_{l=1}^{r-1}
 \frac{2\|[\mathcal{S}(U^{\top}\!GV_{1})]_{a_{l}\beta_{0}}\|_F^2}{\nu_l(X)-\nu_{r}(X)} 	+\sum_{j=1}^{q}\sum_{l'=r+1}^{
 	s(X)+1}\frac{2\zeta_{j}(\Gamma)\|[\mathcal{S}(U^{\top}\! GV_{1})]_{\beta_{j}a_{l'}}\|_F^2}{\nu_{r}(X)-\nu_{l'}(X)}.  
 \end{align*}
 Along with the last four terms on the right hand side of \eqref{temp-equa31}, we obtain the result. 
 
 Next we focus on the case that $ r= s(X)\!+\!1$. From $(U,V)\in\mathbb{O}^{n,m}(X)\cap\mathbb{O}^{n,m}(\Gamma)$ and Lemma \ref{KnorSubdifLemma} (ii), it follows $\Gamma-\sum_{l=1}^{r-1}U_{\!a_l}V_{\!a_l}^\top=U_{b}{\rm Diag}(\zeta(\Gamma)) V^\top_{b}$. Then, 
 \begin{align*}
 &2\big\langle\Gamma-{\textstyle\sum_{l=1}^{r-1}}U_{\!a_l}V_{\!a_l}^\top,G V_{\!a}\Sigma_a^{-1}U^\top_{\!a} G \big\rangle
 =2\langle U_{b}{\rm Diag}(\zeta(\Gamma)) V^\top_{b},G V_{\!a}\Sigma_a^{-1}U^\top_{\!a} G \rangle \nonumber \\
 &=\langle {\rm Diag}(\zeta(\Gamma)),U^\top_{b} G V_{\!a} \Sigma_a^{-1}U^\top_{\!a} G V_{b}+V^\top_{b} G^\top U_{\!a}\Sigma_a^{-1}V^\top_{\!a} G^\top U_{b}\rangle.
 \end{align*}
 Since now $\beta_0,\ldots,\beta_{q}$ is a partition of the index set $b$, we have $U_{b}=[U_{\!\beta_{1}}\,\cdots\,U_{\!\beta_{q}}\ U_{\!\beta_0}]$
 and $V_{b}=[V_{\!\beta_1}\,\cdots\,V_{\!\beta_{q}}\ V_{\!\beta_0}]$. Then, from the above equality, an elementary calculation yields
 \begin{align*}
 &2\big\langle\Gamma-{\textstyle\sum_{l=1}^{r-1}}U_{\!a_l}V_{\!a_l}^\top,G V_{\!a}\Sigma_a^{-1}U^\top_{\!a} G \big\rangle\\
 &=\sum_{j=0}^{q}\zeta_{j}(\Gamma){\rm tr} (U^\top_{\beta_{j}} G V_{\!a} \Sigma_a^{-1}U^\top_{\!a} G V_{\beta_{j}}+V^\top_{\beta_{j}} G^\top U_{\!a} \Sigma_a^{-1}V^\top_{\!a} G^\top U_{\!\beta_{j}})\nonumber\\
 &=\sum_{j=1}^{q}\sum_{l=1}^{r-1}\frac{2\zeta_{j}(\Gamma)}{\nu_{l}(X)}\big\langle[\mathcal{S}(U^{\top}\! G {V_{1}})]_{a_{l}\beta_{j}},[\mathcal{S}(V_{1}^\top G^{\top} {U})]_{\beta_{j}a_{l}}\big\rangle\nonumber\\
 &=\sum_{l=1}^{r-1}\sum_{j=1}^{q}\frac{2\zeta_{j}(\Gamma)}{\nu_{l}(X)}\big(\| [\mathcal{S}(U^{\top}\! G {V_{1}})]_{a_{l}\beta_{j}}\|_F^2-\| [\mathcal{T}(U^{\top}GV_{1})]_{a_{l}\beta_{j}}\|_F^2\big),
 \end{align*}
 where the third equality is due to the definition of the linear mappings $\mathcal{S}$ and $\mathcal{T}$. Combining the above equation with the previous \eqref{GammTraceTemp1} and \eqref{KNormSubDer2} results in
 \begin{align*}
 d^2\Psi_{\!\kappa}(X|\Gamma)(G)
 &=\!\sum_{l=1}^{r-1}\sum_{l\ne l'=1}^{r}\frac{2\|[\mathcal{S}(U^{\top}\!G V_{1})]_{a_{l}a_{l'}}\|_F^2}{\nu_l(X)\!-\!\nu_{l'}(X)}+\sum_{l=1}^{r-1}\sum_{l'=1}^{r}
\frac{2\|[\mathcal{T}(U^{\top}\!GV_{1})]_{a_{l}a_{l'}}\|_F^2}{\nu_l(X)\!+\!\nu_{l'}(X)}\\
 &\quad-\sum_{l=1}^{r-1}\sum_{j=1}^{q}\frac{2\zeta_{j}(\Gamma)}{\nu_{l}(X)}\big(\| [\mathcal{S}(U^{\top}\! G {V_{1}})]_{a_{l}\beta_{j}}\|_F^2-\| [\mathcal{T}(U^{\top}GV_{1})]_{a_{l}\beta_{j}}\|_F^2\big)\\
 &\quad+\sum_{l=1}^{r-1}\frac{1}{\nu_l(X)}\|U_{\!a_l}^\top GV_{\!c}\|_F^2.
 \end{align*}
 For the sum of the first three terms on the right hand side, we calculate that
 \begin{align*}
  &\sum_{l=1}^{r-1}\frac{2\|[\mathcal{S}(U^{\top}\!G V_{1})]_{a_{l}b}\|_F^2}{\nu_l(X)}+\sum_{l=1}^{r-1}\sum_{l'=1}^{r-1}    	\frac{2\|[\mathcal{T}(U^{\top}\!GV_{1})]_{a_{l}a_{l'}}\|_F^2}{\nu_l(X)\!+\!\nu_{l'}(X)}+\sum_{l=1}^{r-1}\frac{2\|[\mathcal{T}(U^{\top}\!GV_{1})]_{a_{l}b}\|_F^2}{\nu_l(X)}\\
  &  -\!\sum_{l=1}^{r-1}\sum_{j=1}^{q}\frac{2\zeta_{j}(\Gamma)}{\nu_{l}(X)}\big(\| [\mathcal{S}(U^{\top}\! GV_{1})]_{a_{l}\beta_{j}}\|_F^2-\!\| [\mathcal{T}(U^{\top}\! GV_{1})]_{a_{l}\beta_{j}}\|_F^2\big)\nonumber\\
  &=\sum_{l=1}^{r-1}\sum_{j=1}^{q} 		\frac{2(1-\!\zeta_{j}(\Gamma))}{\nu_l(X)}\|[\mathcal{S}(U^{\top}\! GV_{1})]_{a_{l}\beta_{j}}\|_F^2+\sum_{l=1}^{r-1}\frac{2}{\nu_l(X)}\|[\mathcal{S}(U^{\top}\! GV_{1})]_{a_{l}\beta_{0}}\|_F^2\\
  &\quad +\sum_{l=1}^{r-1}\sum_{l'=1}^{r-1} 		\frac{2\|[\mathcal{T}(U^{\top}\!GV_{1})]_{a_{l}a_{l'}}\|_F^2}{\nu_l(X)\!+\!\nu_{l'}(X)}+\sum_{l=1}^{r-1}\sum_{j=0}^{q} 		\frac{2(1\!+\!\zeta_{j}(\Gamma))}{\nu_l(X)}\|[\mathcal{T}(U^{\top}\!GV_{1})]_{a_{l}\beta_{j}}\|_F^2.
  \end{align*}
 The above two equations implies the desired result. The proof is completed.	
 \end{proof}

From Proposition \ref{MainTh2}, we obtain the second subderivative of the nuclear norm $\|\cdot\|_*$.
 \begin{corollary}\label{Nnorm-subderiv}
  Fix any $ X\in \mathbb{R}^{n\times m}$ and $ \Gamma\in\partial\| X\|_{*}$. Pick $(U,V)\in\mathbb{O}^{n,m}( X) \cap \mathbb{O}^{n,m}( \Gamma)$ with $V=[V_{1}\ \ V_{c}]$ for $V_{1}\in\mathbb{O}^{m\times n}$. Let $a_l=a_l(X)$ for each $l\in[s(X)+1]$ with $a_{l}(X)$ defined by \eqref{al}, and let $r\in[s(X)+1]$ be such that $n\in a_{r}$. Let $\zeta_1(\Gamma)>\cdots>\zeta_{q}(\Gamma)$ be the nonzero distinct entries of $\big\{\sigma_i(\Gamma)\ |\ i\in a_{r}\big\}$, and write
  $\beta_{l}(\Gamma)\!:=\big\{i\in a_{r}\,|\,\sigma_i(\Gamma)=\zeta_{l}(\Gamma)\big\}$ for $l\in[q]$ and $\beta_{0}(\Gamma)\!:=\big\{i\in a_{r}\,|\,\sigma_i(\Gamma)=0\big\}$. Then, for any $G\in \mathcal{C}_{\|\cdot\|_{*}}( X, \Gamma)$, if ${\rm rank}(X)=n$,  	
  \begin{equation*}
  d^2\|\cdot\|_{*}( X| \Gamma)(G)
  =\sum_{l=1}^{s(X)}\sum_{l'=1}^{s(X)}\frac{2\|[\mathcal{T}(U^{\top}\!GV_{\!1})]_{a_{l}a_{l'}}\|_F^2}{\nu_l(X)\!+\!\nu_{l'}(X)} +\sum_{l=1}^{s(X)}\frac{1}{\nu_l(X)}\|U^\top_{\!a_l}G{V_{\!c}}\|_F^2;
 \end{equation*} 
 and if ${\rm rank}(X)<n$, 
 \begin{align*}
  &d^2\|\cdot\|_{*}( X| \Gamma)(G)
  =\!\sum_{l=1}^{s(X)}\sum\limits_{l'=1}^{s(X)}
  \frac{2\|[\mathcal{T}(U^{\top}\!GV_{\!1})]\|_F^2}{\nu_l(X)\!+\!\nu_{l'}(X)}+\sum_{l=1}^{s(X)}\frac{1}{\nu_l(X)}\|U_{\!a_l}^\top GV_c\|_F^2\\
  &\qquad+\sum_{l=1}^{s(X)}\sum_{j=0}^{q}\!\Big[\frac{2(1\!-\zeta_{j}(\Gamma))}{\nu_l(X)}
 	\|[\mathcal{S}(U^{\top}\!G V_{\!1})]_{a_{l}\beta_{j}}\|_F^2 \!+\!\frac{2(1\!+\zeta_{j}(\Gamma))}{\nu_l(X)}\|[\mathcal{T}(U^{\top}\!G V_{1})]_{a_{l}\beta_{j}}\|_F^2\Big].
 \end{align*} 
 \end{corollary}
\begin{proof}
 When ${\rm rank}(X)=n$, invoking Proposition \ref{MainTh2} with $r=s(X)$ and noting that $a_{s(X)+1}=\emptyset,q=1$ and $\beta_{1}=a_{s(X)}$, we obtain 
 \begin{align*}
 d^2\|\cdot\|_{*}( X| \Gamma)(G)&=\sum_{l=1}^{r-1}\sum_{l'=1}^{ s(X)}\frac{2\|[\mathcal{T}(U^{\top}\!G V_{1})]_{a_{l}a_{l'}}\|_F^2}{\nu_l(X)\!+\!\nu_{l'}(X)}+\sum_{l'=1}^{ s(X)}
 \frac{2\|[\mathcal{T}(U^{\top}\! GV_{1})]_{a_{r}a_{l'}}\|_F^2}{\nu_{r}(X)+\nu_{l'}(X)}\\
 &\quad+\sum_{j=1}^{q}\frac{\zeta_{j}(\Gamma)}{\nu_{r}(X)}\|U^\top_{\!\beta_j}GV_{c}\|_F^2+\sum_{l=1}^{r-1}\frac{\|U_{\!a_l}^\top GV_{\!c}\|_F^2}{\nu_l(X)}\\
 &=\sum_{l=1}^{s(X)}\sum_{l'=1}^{s(X)}\frac{2\|[\mathcal{T}(U^{\top}\!GV_{\!1})]_{a_{l}a_{l'}}\|_F^2}{\nu_l(X)\!+\!\nu_{l'}(X)}+\sum_{l=1}^{s(X)}\frac{1}{\nu_{l}(X)}\|U^\top_{a_{l}}GV_{c}\|_F^2.
 \end{align*}
 When ${\rm rank}(X)<n$, the conclusion follows Proposition \ref{MainTh2} with $r=s(X)+1$.
 \end{proof}

Using Proposition \ref{MainTh2} for $r=1$ yields the second subderivative of the spectral norm.
\begin{corollary}\label{spectral-subderiv}
 Fix any $ X\in \mathbb{R}^{n\times m}$ and $ \Gamma\in\partial\|X\|$. Pick $(U,V)\in\mathbb{O}^{n,m}( X) \cap \mathbb{O}^{n,m}( \Gamma)$ with $V=[V_{1}\ \ V_{c}]$ for $V_{1}\in\!\mathbb{O}^{m\times n}$. Let $a_l=a_l(X)$ for each $l\in[s(X)+1]$ with $a_{l}(X)$ defined by \eqref{al}. Let $\zeta_1(\Gamma)>\cdots>\zeta_{q}(\Gamma)$ be the nonzero distinct entries of the set $\big\{\sigma_i(\Gamma)\ |\ i\in a_1\big\}$, and write $\beta_{0}(\Gamma)\!:=\big\{i\in a_1\,|\,\sigma_i(\Gamma)=0\big\}$ and $\beta_l(\Gamma)\!:=\big\{i\in a_1\,|\,\sigma_i(\Gamma)=\zeta_{l}(\Gamma)\big\}$ for $l\in[q]$. Then, for any $G\in\mathcal{C}_{\|\cdot\|}( X, \Gamma)$, 
  \begin{align*}
   d^2\|\cdot\|( X| \Gamma)(G) 
   &=\sum_{j=1}^{q}\!\sum_{l'=2}^{s(X)+1}\frac{2\zeta_{j}(\Gamma)\| [\mathcal{S}(U^{\top}\! GV_{\!1})]_{\beta_{j}a_{l'}}\|_F^2}{\nu_{1}(X)-\nu_{l'}(X)}
   	+\sum_{j=1}^{q}\frac{\zeta_{j}(\Gamma)}{\nu_{1}(X)}\|U^\top_{\!\beta_j}GV_{\!c}\|_F^2\\
   &\quad+\sum_{j=1}^{q}\!\sum_{l'=1}^{s(X)+1}\frac{2\zeta_{j}(\Gamma)\|[\mathcal{T}(U^{\top}\! GV_{1})]_{\beta_ja_{l'}}\|_F^2}{\nu_{1}(X)+\nu_{l'}(X)}.
  \end{align*} 
 \end{corollary}

 Notice that $d^2\Psi_{\kappa}(X|\Gamma)$ is always nonnegative by the convexity of {\color{blue}$\Psi_{\kappa}$}. Together with its expression in Proposition \ref{MainTh2}, we have the following conclusion.
 \begin{corollary}\label{corollary3.3}
 Fix any $X\in\mathbb{R}^{n\times m}$ and $\Gamma\in\partial\Psi_{\kappa}(X)$. Pick $(U,V)\in\mathbb{O}^{n,m}(X)\cap\mathbb{O}^{n,m}(\Gamma)$ with $V=[V_{1}\ \ V_{c}]$ for $V_{1}\in\mathbb{O}^{m\times n}$. For each $l\in[s(X)+1]$, let $a_l=a_l(X)$ with $a_l(X)$ defined by \eqref{al}, and let $r\in[s(X)\!+\!1]$ be such that $\kappa\in a_{r}$. Write
 $\alpha\!:=\bigcup_{l=1}^{r-1}a_{l},\beta\!:=a_{r}$, $ \gamma\!:=\bigcup_{l=r+1}^{s(X)+1}a_l$ and $\beta_{+}\!:=\!\beta\backslash(\beta_1\cup\beta_0)$
 with $\beta_1\!:=\beta_1(\Gamma)$ and $\beta_0\!:=\beta_0(\Gamma)$.
 Then, for any $G\in\mathcal{C}_{\Psi_{\kappa}}(X,\Gamma)$, 
 when $r\in[s(X)]$,  $d^2\Psi_{\kappa}(X|\Gamma)(G)=0$ if and only if
 \begin{subnumcases}{}
  (U^{\top}\!GV_1)_{(\alpha\cup\beta_1\cup\beta_+)(\alpha\cup\beta_1\cup\beta_+)}\in \mathbb{S}^{|\alpha|+|\beta_1|+|\beta_+|}, \\
  (U^{\top}\!GV_{1})_{\beta_1\beta_0}=(U^{\top}\!G V_{1})_{\beta_0\beta_1}^{\top},\,(U^{\top}\!G V_{1})_{\beta_+\beta_0}=(U^{\top}\!GV_{1})_{\beta_0\beta_+}^\top, \\
  (U^{\top}\!GV_{1})_{\alpha\beta_+}=(U^{\top}\! G V_{1})_{\beta_+\alpha}^{\top}=0,\,(U^{\top} G V_{1})_{\alpha\beta_0}=(U^{\top}GV_{1})_{\beta_0\alpha}^\top=0,\nonumber\\
  (U^{\top}\!GV_1)_{\alpha\gamma}=(U^{\top}\!G V_1)_{\gamma\alpha}^\top=0,\\
  (U^{\top}\!G V_{1})_{\beta_1\gamma}=(U^{\top}\! G V_{1})_{\gamma\beta_1}^{\top}=0,\,
  (U^{\top}\!GV_{1})_{\beta_+\gamma}=(U^{\top}\! G V_{1})_{\gamma\beta_+}^{\top}=0,\nonumber\\
  (U^{\top}\!GV_{1})_{\alpha c}=0,\,
  (U^{\top}\!GV_{1})_{\beta_1 c}=0,\,
 (U^{\top}\!GV_{1})_{\beta_+ c}=0;
 \end{subnumcases}
 and when $r=s(X)+1$, $d^2\Psi_{\kappa}(X|\Gamma)(G)=0$ if and only if 
 \begin{subnumcases}{}
 (U^{\top}\!GV_{1})_{\alpha\alpha}\in \mathbb{S}^{|\alpha|},\,
 (U^{\top}\!GV_{1})_{\alpha\beta_1}=(U^{\top}\!G V_{1})_{\beta_1\alpha}^{\top},\\
 (U^{\top}\!GV_{1})_{\alpha c}=0,\,
 (U^{\top}\!GV_{1})_{\alpha\beta_+}=(U^{\top}\!G V_{1})_{\beta_+\alpha}^\top=0, \\
 (U^{\top}\!GV_{1})_{\alpha\beta_0}=(U^{\top}\!G V_{1})_{\beta_0\alpha}^{\top}=0.
 \end{subnumcases}
\end{corollary}  
 \section{Characterization of tilt stability}\label{sec4}
  
 To provide a specific characterization of tilt stability for problem \eqref{KnormRegular}, we need the following technical lemma to present the critical cone of $\Psi_{\kappa}$ at any point $(X,\Gamma)\in{\rm gph}\,\partial\Psi_{\kappa}$.
 \begin{lemma}\label{lemma-CCone}
 (see \cite[Propositon 10]{Ding17}) Consider any $X\in\mathbb{R}^{n\times m}$ and  $\Gamma\in\partial\Psi_{\kappa}(X)$. Pick $(U,V)\in\mathbb{O}^{n,m}(X) \cap \mathbb{O}^{n,m}(\Gamma)$. Let $a_l=a_l(X)$ for each $l\in[ s(X)+1]$ with $a_{l}(X)$ defined by \eqref{al}, let $r\in[ s(X)+\!1]$ be such that $\kappa\in a_{r}$, and let $\beta_l=\beta_{l}(\Gamma)$ for $l=0,1,\ldots, q$ with $\beta_{l}(\Gamma)$ defined by \eqref{beta_l} and $\beta_+\!:=\bigcup_{j=2}^q\beta_j(\Gamma)$. Then the following assertions hold. 
 \begin{itemize}
 \item [(i)] When $r\in [ s(X)]$, $G\in\mathcal{C}_{\Psi_{\kappa}}(X,\Gamma)$ if and only if there exists $\varpi\in \mathbb{R}$ such that
  \begin{align*}
  \lambda_1\big(\mathcal{S}(U^{\top}_{\!\beta_{0}} G V_{\!\beta_{0}})\big)\leq \varpi\leq \lambda_{|\beta_1|}\big(\mathcal{S}(U^{\top}_{\!\beta_1} G V_{\!\beta_1})\big),\qquad\\	
   \mathcal{S}(U^{\top}_{\!a_{r}} G V_{\!a_{r}})=\left[\begin{matrix}
   	\mathcal{S}(U^{\top}_{\!\beta_1} GV_{\!\beta_1})& 0 & 0\\
   	0 & \varpi I_{|\beta_{+}|} & 0 \\
   	0& 0 & 	\mathcal{S}(U^{\top}_{\!\beta_{0}} G V_{\!\beta_{0}})
   \end{matrix}\right].
  \end{align*}

 \item [(ii)] When $r=\! s(X)+1$ and $\|\Gamma\|_*\!<\kappa$, $G\in\!\mathcal{C}_{\Psi_{\kappa}}(X,\Gamma)$ iff 
 $\mathcal{S}(U^{\top}_{\!\beta_1} G V_{\!\beta_1})\in \mathbb{S}_+^{|\beta_1|}$ and 
 \[
  [U_{\!a_{r}}^\top GV_{\!a_{r}}\ \ U_{\!a_{r}}^\top G V_{c}]=\begin{pmatrix}
	\mathcal{S}(U^{\top}_{\beta_1} GV_{\beta_1})& 0 & 0& 0\\
	0 & 0 & 0&0 \\
	0& 0 & 	0& 0
  \end{pmatrix}.
 \]
  
 \item [(iii)] When $r=s(X)+1$ and $\|\Gamma\|_*\!=\kappa$, $G\in\mathcal{C}_{\Psi_{\kappa}}(X,\Gamma)$ iff there is $\varpi\in\mathbb{R}$ such that
  \begin{align*}
  \sigma_1\big([U_{\beta_0}^\top GV_{\beta_0}\quad U_{\beta_0}^\top G V_{c}]\big)\leq\varpi\leq \lambda_{|\beta_1|}\big(\mathcal{S}(U^{\top}_{\beta_1} G V_{\beta_1})\big),\qquad\qquad\\
  [U_{\!a_{r}}^\top G V_{\!a_{r}}\ \ U_{\!a_{r}}^\top G V_{\!c}]=\begin{pmatrix}
  \mathcal{S}(U^{\top}_{\beta_1} GV_{\beta_1})& 0 & 0& 0\\
   0 & \varpi I_{|\beta_{+}|} & 0&0 \\
   0& 0 & U_{\beta_0}^\top GV_{\!\beta_0}& U_{\beta_0}^\top G V_{c}
  \end{pmatrix}.
  \end{align*}
 \end{itemize}
 \end{lemma}
 \begin{theorem}\label{ThTilt}
  Let $\overline{X}$ be a local optimal solution of problem \eqref{KnormRegular} and let $\overline{\Gamma}\!:=-\nu\nabla\vartheta(\overline{X})$. Suppose that $\nabla^2\vartheta(\cdot)$ is positive semidefinite on an open neighborhood $\mathcal{N}$ of $\overline{X}$. For each $l\in[ s(\overline{X})+\!1]$, let $\overline{a}_l=a_l(\overline{X})$ with $a_l(\overline{X})$ defined by \eqref{al} for $X=\overline{X}$, and let $\overline{r}\in[ s(\overline{X})\!+\!1]$ be such that $\kappa\in\overline{a}_{\overline{r}}$. Write
  $\alpha\!:={\textstyle\bigcup}_{l=1}^{\overline{r}-1}\overline{a}_{l},\beta\!:=\overline{a}_{\overline{r}}, \gamma\!:={\textstyle\bigcup}_{l=\overline{r}+1}^{ s(\overline{X})+1}\overline{a}_l$ and $ \beta_{+}\!:=\beta\backslash(\beta_1\cup\beta_0)$ with $\beta_1\!:=\big\{i\in\beta\,|\,\sigma_i(\overline{\Gamma})=1\}$ and $\beta_0\!:=\big\{i\in\beta\,|\,\sigma_i(\overline{\Gamma})=0\big\}$.
 Then $\overline{X}$ is a tilt-stable solution of problem \eqref{KnormRegular} if and only if ${\rm Ker}\,\nabla^2\vartheta(\overline{X})\cap \Upsilon=\{0\}$ where, if $\overline{r}\in[ s(\overline{X})]$,  
 \begin{align*}
  \Upsilon=\bigg\{G\!\in\mathbb{R}^{n\times m}\,\bigg|\,\exists
  (\overline{U},\overline{V})\!\in\mathbb{O}^{n,m}(\overline{X})\cap\mathbb{O}^{n,m}(\overline{\Gamma}),D\in \mathbb{R}^{|\beta_0|\times|\beta_0|},	
  \begin{pmatrix}
  	A &  B \\
   B^{\top} & C
  \end{pmatrix}\!\in\mathbb{S}^{|\alpha|+|\beta_1|},\\
  \left.\lambda_1(\mathcal{S}(D))\le\varpi\le \lambda_{|\beta_1|}(C)\ {\rm and}\  
  	\begin{pmatrix}
  		E_{11} &   E_{12} \\
  		E_{21} & E_{22}
  	\end{pmatrix}\in \mathbb{R}^{(|\beta_0|+|\gamma|)\times (|\gamma|+|c|)}\right.\\
  	\left. {\rm such\ that}\ \overline{U}^{\top}\!G\overline{V}=\begin{pmatrix}
  		A & B & 0 & 0 & 0 & 0\\
  		B^\top  &  C & 0 &0 & 0 & 0\\
  		0 & 0& \varpi I_{|\beta_+|} & 0 & 0  & 0\\
  		0 & 0 & 0 & D & E_{11} &   E_{12} \\
  		0 & 0 & 0 & 0 & E_{21} & E_{22}
  	\end{pmatrix}\right\};\qquad     	 
  \end{align*}
  if $\overline{r}= s(\overline{X})+1$ and $\|\overline{\Gamma}\|_*<\kappa$,  
  \begin{align*}
  	\Upsilon=\bigg\{G\!\in\mathbb{R}^{n\times m}\,\bigg|\,\exists
  	(\overline{U},\overline{V})\!\in\mathbb{O}^{n,m}(\overline{X})\cap\mathbb{O}^{n,m}(\overline{\Gamma}), A\in \mathbb{S}^{|\alpha|}, B\in \mathbb{R}^{|\alpha|\times |\beta_1|}\ {\rm and}\\
  	\left. C\in \mathbb{S}^{|\beta_1|}\ {\rm such\ that}\ \overline{U}^{\top}\!G\overline{V}=\begin{pmatrix}
  			A & B & 0 & 0 & 0 \\
  		B^\top  &  C & 0 &0 & 0 \\
  		0 & 0& 0 & 0 & 0  \\
  		0 & 0 & 0 & 0&0
  	\end{pmatrix}\right\};     	 
  \end{align*}
  and if $\overline{r}= s(\overline{X})+1$ and $\|\overline{\Gamma}\|_*=\kappa$, 
  \begin{align*}
  	\Upsilon=\bigg\{G\!\in\mathbb{R}^{n\times m}\,\bigg|\,\exists
  	(\overline{U},\overline{V})\!\in\mathbb{O}^{n,m}(\overline{X})\cap\mathbb{O}^{n,m}(\overline{\Gamma}), A\in \mathbb{S}^{|\alpha|}, B\in \mathbb{R}^{|\alpha|\times |\beta_1|},C\in \mathbb{S}^{|\beta_1|}, \\
  	\left. D\in \mathbb{R}^{|\beta_0|\times |\beta_0|},E\in \mathbb{R}^{|\beta_0|\times |c|}\ {\rm and}\ \sigma_1([D\ \ E])\le\varpi\le \lambda_{|\beta_1|}(C)\right.\quad\\
  	\left. {\rm such\ that}\ \overline{U}^{\top}\!G\overline{V}=\begin{pmatrix}
  		A & B & 0 & 0 & 0 \\
  	B^\top  &  C & 0 &0 & 0 \\
  	0 & 0& \varpi I_{|\beta_+|} & 0 & 0  \\
  	0 & 0 & 0 & 	D& E 
  	\end{pmatrix}\right\}.    	 
  \end{align*}
 \end{theorem}
 \begin{proof}
 From Proposition \ref{TiltProp2} with $\mathbb{X}=\mathbb{R}^{n\times m},\varphi=\nu\vartheta$ and $g=\Psi_{\kappa}$, it follows that $\overline{X}$ is a tilt-stable solution of \eqref{KnormRegular}  if and only if
 ${\rm Ker}\,\nabla^2\vartheta(\overline{X})\cap \mathcal{G}=\{0\}$ with 
 \begin{align*}
 \mathcal{G}=\big\{G\in \mathbb{R}^{n\times m}\ |\ \exists (X^k,\Gamma^k)\in {\rm gph}\,\partial \Psi_{\kappa}\ {\rm and}\ G^k\in \mathbb{R}^{n\times m}\ \  {\rm such\ that}\qquad\qquad\qquad\\
 \lim_{k\to \infty}(X^k,\Gamma^k,G^k)= (\overline{X},\overline{\Gamma},G)\ {\rm and}\ \lim_{k\to \infty} d^2 \Psi_{\kappa}(X^k|\Gamma^k)(G^k)=0 \big\}.
 \end{align*}
 Therefore, it suffices to prove that $\mathcal{G}=\Upsilon$. We first argue that $\mathcal{G}\subset\Upsilon$. Pick any $G\in \mathcal{G}$. Then, there exist sequences $\{(X^k,\Gamma^k)\}_{k\in\mathbb{N}}\subset{\rm gph}\, \partial \Psi_{\kappa}$ and $\{G^k\}_{k\in\mathbb{N}}\subset\mathbb{R}^{n\times m}$ such that
 \begin{equation}\label{SecDerCase1}
 \lim_{k\to \infty}(X^k,\Gamma^k,G^k)= (\overline{X},\overline{\Gamma},G)\ \ {\rm and}\ \lim_{k\to \infty} d^2 \Psi_{\kappa}(X^k|\Gamma^k)(G^k)=0. 
 \end{equation}  
 By Proposition \ref{tepi-diff}, for each $k\in\mathbb{N}$, ${\rm dom}\,d^2\Psi_{\kappa}(X^k|\Gamma^k)=\mathcal{C}_{\Psi_{\kappa}}(X^k,\Gamma^k)$. The second limit in \eqref{SecDerCase1} implies $G^k\in\mathcal{C}_{\Psi_{\kappa}}(X^k,\Gamma^k)$ for each $k\in\mathbb{N}$. From $\Gamma{^k}\in \partial \Psi_{\kappa}(X{^k})$, the convexity of $\Psi_k$ and \cite[Corollary 2.5]{Lewis95}, for each $k\in\mathbb{N}$, there exists $(U^k,V^k)\in \mathbb{O}^{n,m}(X^k)\cap \mathbb{O}^{n,m}(\Gamma^k)$. Note that $\{(U{^k},V{^k})\}_{k\in\mathbb{N}}$ is bounded, and the multifunction $\mathbb{R}^{n\times m}\ni Z\rightrightarrows\mathbb{O}^{n,m}(Z)$ is outer semicontinuous by \cite[Lemma 2.1]{Cui2017}. There exist an infinite index set $\mathcal{K}\subset \mathbb{N}$ such that  $\lim_{\mathcal{K}\ni k\to\infty}(U{^k},V{^k})=(\overline{U},\overline{V})$ for some $(\overline{U},\overline{V})\in \mathbb{O}^{n,m}(\overline{X})\cap \mathbb{O}^{n,m}(\overline{\Gamma})$. Obviously, 
\begin{equation*}
  \overline{X}=\overline{U}
  \big[{\rm Diag}(\sigma(\overline{X}))\ \ 0\big]\overline{V}^\top
  \text{\; and \;}\overline{\Gamma}=\overline{U}\big[{\rm Diag}(\sigma(\overline{\Gamma}))\ \ 0\big]\overline{V}^\top
 \end{equation*}
 with $\overline{V}=[\overline{V}_1\ \ \overline{V}_{\!c}]$ for $\overline{V}_1\in\mathbb{O}^{m\times n}$. In the following,  we argue that $G\in\Upsilon$ by three cases. 
 
 \noindent
 {\bf Case 1: $\overline{r}\in [ s(\overline{X})]$}. Recall that $\overline{\Gamma}\in\partial\Psi_{\!\kappa}(\overline{X})$. By Lemma \ref{KnorSubdifLemma} (i) with $(X,\Gamma)=(\overline{X},\overline{\Gamma})$, there exist integers $\overline{\kappa}_0$ and $\overline{\kappa}_1$ with $0\le \overline{\kappa}_0\le\kappa-1$ and $\kappa\le\overline{\kappa}_1\le n$ such that  
 \begin{align}\label{case1-relation1}
 \sigma_1(\overline{X})\geq \cdots\geq \sigma_{\overline{\kappa}_0}(\overline{X})>\sigma_{\overline{\kappa}_0+1}(\overline{X})=\cdots=\sigma_{\kappa}(\overline{X})=\cdots=\sigma_{\overline{\kappa}_1}(\overline{X})\qquad\\
  >\sigma_{\overline{\kappa}_1+1}(\overline{X})\geq \cdots \geq \sigma_n(\overline{X})\geq 0,\qquad\\	
  \label{case1-relation2}
  \sigma_{\alpha}(\overline{\Gamma})={\color{blue}e_{\alpha}},\ {\textstyle\sum_{i\in \beta}}\,\sigma_i(\overline{\Gamma})=\kappa-\overline{\kappa}_0\ {\rm with}\ 0\leq\sigma_{\beta}(\overline{\Gamma})\leq {\color{blue}e_{\beta}},\ \ {\rm and}\ \ \sigma_{\gamma}(\overline{\Gamma})=0,
 \end{align}
 where $[\overline{\kappa}_0]=\alpha, \{\overline{\kappa}_0\!+\!1,\ldots,\overline{\kappa}_1\}=\beta$ and $ \{\overline{\kappa}_1\!+\!1,\ldots,n\}=\gamma$.
 Since $\lim_{\mathcal{K}\ni k\to\infty}\sigma(X^k)=\sigma(\overline{X})$, there must exist integers $\kappa_0$ and $\kappa_1$ with $\overline{\kappa}_0\leq \kappa_0\le\kappa\!-\!1$ and $\kappa\leq \kappa_1\leq\overline{\kappa}_1$ and an infinite index set $\widehat{\mathcal{K}}\subset\mathcal{K}$ such that for all $k\in\widehat{\mathcal{K}}$, 
  \begin{align}\label{XjDecrease1}
  \sigma_1(X^k)\geq \cdots\geq \sigma_{\kappa_0}(X^k)>\sigma_{\kappa_0+1}(X^k)=\cdots\!=\!\sigma_\kappa(X^k)=\cdots=\sigma_{\kappa_1}(X^k)\qquad\qquad\nonumber\\
  >\sigma_{\kappa_1+1}(X^k)\geq \cdots \geq \sigma_n(X^k)\geq 0.\qquad\qquad\qquad
 \end{align}
 Together with $\Gamma{^k}\in\partial \Psi_{\kappa}(X^k)$ and Lemma \ref{KnorSubdifLemma} (i) for $(X,\Gamma)=(X^k,\Gamma^k)$, for each $k\in\widehat{\mathcal{K}}$,
 \begin{align}\label{GamajDecrease1}
  \sigma_{\widehat{\alpha}}(\Gamma^k)={\color{blue}e_{\widehat{\alpha}}},\ 
 {\textstyle\sum_{i\in\widehat{\beta}}}\,\sigma_i(\Gamma^k)=\kappa-\kappa_0\ \ {\rm with}\ 0\leq \sigma_{\widehat{\beta}}(\Gamma^k)\leq {\color{blue}e_{\widehat{\beta}}},\ {\rm and}\ \ \sigma_{\widehat{\gamma}}(\Gamma^k)=0
 \end{align}
 where $\widehat{\alpha}\!:=\{1,\ldots,\kappa_0\},\widehat{\beta}\!:=\!\big\{\kappa_0+\!1,\ldots,\kappa_1\big\}$ and $\widehat{\gamma}\!:=\{\kappa_1\!+\!1,\ldots,n\}$.
 Then, for all $k\in\widehat{\mathcal{K}}$,  $\big\{i\in\beta\,|\,\sigma_i(X^k)>\sigma_{\kappa}(X^k)\big\}\!=\{\overline{\kappa}_0\!+\!1,\ldots,\kappa_0\}\!:=\eta^1,\big\{i\in\beta\,|\,\sigma_i(X^k)=\sigma_{\kappa}(X^k)\big\}=\widehat{\beta}$ and $\big\{i\in\beta\,|\,\sigma_i(X^k)<\sigma_{\kappa}(X^k)\big\}=\big\{\kappa_1\!+\!1,\ldots,\overline{\kappa}_1\big\}\!:=\eta^3$. Moreover, if necessary by taking an infinite subset of $\widehat{\mathcal{K}}$, the index sets $\{i\in\widehat{\beta}\,|\, \sigma_{i}(\Gamma^k)=1\}, \{i\in\widehat{\beta}\,|\, 0<\sigma_{i}(\Gamma^k)<1\}$ and $\{i\in\widehat{\beta}\,|\, \sigma_{i}(\Gamma^k)=0\}$ for all $k\in\widehat{\mathcal{K}}$ are independent of $k$. For convenience, write  
 \begin{equation*}
 \widehat{\beta}_1:=\{i\in \widehat{\beta}\,|\, \sigma_{i}(\Gamma^k)=1\},\widehat{\beta}_+:=\{i\in \widehat{\beta}\,|\, 0<\sigma_{i}(\Gamma^k)<1\},\ 
 \widehat{\beta}_0:=\{i\in \widehat{\beta}\,|\, \sigma_{i}(\Gamma^k)=0\}.
 \end{equation*} 	
 Let $\widehat{\beta}_{+}^1\!:=\{i\in\widehat{\beta}_+\,|\, \sigma_{i}(\overline{\Gamma})=1\}, \widehat{\beta}_{+}^0\!:=\{i\in\widehat{\beta}_+\,|\, \sigma_{i}(\overline{\Gamma})=0\},\beta_{+}\!:=\{i\in\widehat{\beta}_+\,|\, \sigma_{i}(\overline{\Gamma})\in(0,1)\}$. 
 From the definitions of the above index sets, it is not difficult to infer that
 \begin{align}\label{relation41}
  \widehat{\alpha}=\alpha\cup \eta^1,\ 
  \widehat{\gamma}=\gamma\cup \eta^3,\ 
  \beta=\beta_1\cup\beta_0\cup\beta_{+}=\eta^1\cup\widehat{\beta}\cup\eta^3,\
  \widehat{\beta}=\widehat{\beta}_1\cup\widehat{\beta}_{+}\cup\widehat{\beta}^0,\\
  \label{relation42}
  \widehat{\beta}_+= \widehat{\beta}_{+}^1\cup \beta_+\cup \widehat{\beta}_{+}^0,\
  \beta_1=\eta^1\cup \widehat{\beta}_1\cup \widehat{\beta}_{+}^1,\ \beta_0=\widehat{\beta}_{+}^0\cup \widehat{\beta}_0\cup\eta^3.\qquad\qquad
 \end{align}
 From \eqref{XjDecrease1}, for each $k\in\widehat{\mathcal{K}}$, $s(X^k)$ is independent of $k$ and so is $a_l(X^k)$ for each $l\in\![s(X^k)]$, where $a_l(X^k)$ is the index set defined by \eqref{al} with $X=X^k$. Write $\widehat{a}_l:=a_l(X^k)$ for all $k\in\widehat{\mathcal{K}}$ and each $l\in[s(X^k)]$, and let $r\in[s(X^k)+1]$ be such that $\kappa\in\widehat{a}_{r}$. From the above \eqref{XjDecrease1}, obviously, $\widehat{a}_{r}=\widehat{\beta}$. Recalling that $\overline{r}\in s(\overline{X})$ and $\kappa\in \overline{a}_{\overline{r}}$, we infer that $r\in[s(X^k)]$. For each $k\in\widehat{\mathcal{K}}$, let $\zeta_1(\Gamma^k)>\cdots>\zeta_{q}(\Gamma^k)$ be the nonzero distinct entries in the set $\{\sigma_i(\Gamma^k)\ |\ i\in \widehat{a}_r\}$, and for each $l\in\{2,\ldots,q\}$, let $\beta_q(\Gamma^k)$ be defined by \eqref{beta_l} with $\Gamma=\Gamma^k$, which is also independent of $k$. Let $\widehat{\beta}_q:=\beta_q(\Gamma^k)$. Clearly, $\widehat{\beta}=\bigcup_{j=0}^q\widehat{\beta}_j$. 
  From Proposition \ref{MainTh2} with $(X,\Gamma)=(X^k,\Gamma^k)$, for every $k\in\widehat{\mathcal{K}}$, $d^2\Psi_{\!\kappa}(X^k|\Gamma^k)(G^k)$ equals
  \begin{align}\label{LimitCase1}
   &\sum_{l=1}^{r-1}\sum_{l'=r+1}^{ s(X^k)+1}
  	\!\frac{2\|[\mathcal{S}((U^k)^{\top}\!G^k V_{1}^k)]_{\widehat{a}_{l}\widehat{a}_{l'}}\|_F^2}{\nu_l(X^k)\!-\!\nu_{l'}(X^k)}+\!\sum_{l=1}^{r-1}\sum_{j=1}^{q}\!\frac{2(1\!-\!\zeta_j(\Gamma^k))}{\nu_l(X^k)\!-\!\nu_{r}(X^k)}\| [\mathcal{S}((U^k)^{\top}\!G^kV_{1}^k)]_{\widehat{a}_{l}\widehat{\beta}_{j}}\|_F^2\nonumber\\
  & +\sum_{j=1}^{q}\sum_{l'=r+1}^{ s(X^k)+1}\!\frac{2\zeta_{j}(\Gamma^k)\| [\mathcal{S}((U^k)^{\top}\! G^k V_1^k)]_{\widehat{\beta}_{j}\widehat{a}_{l'}}\|_F^2}{\nu_{r}(X^k)\!-\!\nu_{l'}(X^k)}+\sum_{l=1}^{r-1}\frac{\|(U_{\!\widehat{a}_l}^k)^\top G^kV_{c}^k\|_F^2}{\nu_l(X^k)} \nonumber\\
  & +\sum_{j=1}^{q}\sum_{l'=1}^{ s(X^k)+1}
  	\frac{2\zeta_{j}(\Gamma^k)\|[\mathcal{T}((U^k)^{\top}\! G^kV_1^k)]_{\widehat{\beta}_j\widehat{a}_{l'}}\|_F^2}{\nu_{r}(X^k)+\nu_{l'}(X^k)}+\sum_{j=1}^{q}\frac{\zeta_{j}(\Gamma^k)}{\nu_{r}(X^k)}\|(U_{\!\widehat{\beta}_j}^k)^\top G^kV_{c}^k\|_F^2
  	\nonumber\\
  &+\sum_{l=1}^{r-1}\sum_{l'=1}^{ s(X^k)+1}\!\frac{2\|[\mathcal{T}((U^k)^{\top}\!G^k V_1^k)]_{\widehat{a}_{l}\widehat{a}_{l'}}\|_F^2}{\nu_l(X^k)\!+\!\nu_{l'}(X^k)}+\sum_{l=1}^{r-1}
  	\frac{2\|[\mathcal{S}((U^k)^{\top}\!G^k V_{1}^k)]_{\widehat{a}_{l}\widehat{\beta}_{0}}\|_F^2}{\nu_l(X^k)\!-\!\nu_{r}(X^k)}.
 \end{align}
 Together with  $\lim_{\widehat{\mathcal{K}}\ni k\to\infty}d^2\Psi_{\kappa}(X^k|\Gamma^k)(G^k)=0$ by \eqref{SecDerCase1}, we have the following relations   
  \begin{subnumcases}{}\label{R1Stru1}
   (\overline{U}^{\top}\!G \overline{V}_1)_{(\widehat{\alpha}\cup\widehat{\beta}_1\cup\widehat{\beta}^1_{+}\cup \beta_{+})(\widehat{\alpha}\cup\widehat{\beta}_1\cup\widehat{\beta}^1_{+}\cup \beta_{+})}\in \mathbb{S}^{|\widehat{\alpha}|+|\widehat{\beta}_1|+|\widehat{\beta}^1_+|+|\beta_+|},\\
  (\overline{U}^{\top}\!G \overline{V}_1)_{(\widehat{\beta}_1\cup\widehat{\beta}^1_{+}\cup\beta_+)(\widehat{\beta}^0_{+}\cup\widehat{\beta}_0)}=(\overline{U}^{\top}\!G \overline{V}_{\!1})_{(\widehat{\beta}^0_{+}\cup\widehat{\beta}_0)(\widehat{\beta}_1\cup\widehat{\beta}^1_{+}\cup \beta_+)}^\top,\label{R1Stru2}\\
  (\overline{U}^{\top} G \overline{V}_1)_{\widehat{\alpha}(\beta_+\cup \widehat{\beta}^0_{+})}=(\overline{U}^{\top}\!G\overline{V}_1)_{(\beta_+\cup \widehat{\beta}^0_{+})\widehat{\alpha}}^\top=0,\label{R1Stru3} \\
  (\overline{U}^{\top} G \overline{V}_1)_{\widehat{\alpha}\widehat{\beta}_0}=(\overline{U}^{\top}\!G \overline{V}_1)_{\widehat{\beta}_0\widehat{\alpha}}^\top=0,
  	\label{R1Stru4}\\
  (\overline{U}^{\top} G \overline{V}_1)_{\widehat{\alpha}\widehat{\gamma}}=(\overline{U}^{\top}\!G \overline{V}_1)_{\widehat{\gamma}\widehat{\alpha}}^\top=0,
  	\label{R1Stru5}\\
  (\overline{U}^{\top} G \overline{V}_1)_{(\widehat{\beta}_{1}\cup \widehat{\beta}^1_{+}\cup\beta_+)\widehat{\gamma}}=(\overline{U}^{\top}\!G \overline{V}_1)_{\widehat{\gamma}(\widehat{\beta}_{1}\cup\widehat{\beta}^1_{+}\cup\beta_{+})}^\top=0,\label{R1Stru6}
  	\\ 
  (\overline{U}^{\top}\!G \overline{V}_{\!1})_{\widehat{\alpha} c}=0,\,
  (\overline{U}^{\top}\!G \overline{V}_{\!1})_{\widehat{\beta}_1 c}=0,\,
  (\overline{U}^{\top}\!G \overline{V}_{\!1})_{(\widehat{\beta}^1_{+}\cup \beta_+) c}=0,
  \label {R1Stru7}
 \end{subnumcases}
 and the detailed arguments are put in Appendix A. In addition, for each $k\in \widehat{\mathcal{K}}$, since $r\in[s(X^k)]$ and $G{^k}\in\mathcal{C}_{\Psi_{\kappa}}(X^k,\Gamma^k)$, by Lemma \ref{lemma-CCone} (i) there exists $\varpi_k\in \mathbb{R}$ such that
 \begin{align*}
 \lambda_1\big[\mathcal{S}\big((U_{\widehat{\beta}_{0}}^k)^{\top}\!G^k V_{\widehat{\beta}_{0}}^k\big)\big]\leq\varpi_k\leq \lambda_{|\widehat{\beta}_1|}\big[\mathcal{S}\big((U_{\widehat{\beta}_1}^k)^{\top}\!G^kV_{\widehat{\beta}_1}^k\big)\big],\qquad\qquad\qquad\qquad\\
 \mathcal{S}\big((U_{\widehat{\beta}}^k)^{\top}\!G^k V_{\widehat{\beta}}^k\big)=\mathcal{S}\big((U_{\widehat{a}_r}^k)^{\top}\!G^k V_{\widehat{a}_r}^k\big)=\left[\begin{matrix}
    	\mathcal{S}\big((U_{\widehat{\beta}_1}^k)^{\top} G^k V_{\widehat{\beta}_1}^k\big)& 0 & 0\\
    	0 & \varpi_k I_{|\widehat{\beta}_{+}|} & 0 \\
    	0& 0 & 	\mathcal{S}\big((U_{\widehat{\beta}_{0}}^k)^{\top}\! G^kV_{\widehat{\beta}_{0}}^k\big)
 \end{matrix}\right].
 \end{align*}
 If necessary by taking a subsequence, we assume $\lim_{\widehat{\mathcal{K}}\ni k\to\infty}\varpi_k=\varpi$. Then, it holds
 \begin{align}\label{temp0-varpi}
 \lambda_1\big([\mathcal{S}(\overline{U}^{\top}\!G\overline{V}_{\!1})]_{\widehat{\beta}_0\widehat{\beta}_0}\big)\leq\varpi\leq \lambda_{|\widehat{\beta}_1|}\big([\mathcal{S}(\overline{U}^{\top}\!G\overline{V}_{\!1})]_{\widehat{\beta}_1\widehat{\beta}_1}\big),\qquad\qquad\quad\\ 	
 \label{CriticalConeCase1}
  \lim_{\widehat{\mathcal{K}}\ni k\to \infty}\mathcal{S}\big((U_{\widehat{\beta}}^k)^{\top}\! G^kV_{\widehat{\beta}}^k\big)
  =\left[\begin{matrix}
 	[\mathcal{S}(\overline{U}^{\top}\!G\overline{V}_{\!1})]_{\widehat{\beta}_1\widehat{\beta}_1}& 0 & 0\\
   	0 & \varpi I_{|\widehat{\beta}_{+}^1\cup \beta_+\cup \widehat{\beta}_{+}^0|} & 0 \\
   	0& 0 & 	[\mathcal{S}(\overline{U}^{\top}\!G\overline{V}_{\!1})]_{\widehat{\beta}_0\widehat{\beta}_0}
    \end{matrix}\right].
 \end{align}
 Take $C=(\overline{U}^{\top}\!G\overline{V}_{\!1})_{(\eta^1\cup \widehat{\beta}_1\cup \widehat{\beta}_{+}^1)(\eta^1\cup \widehat{\beta}_1\cup \widehat{\beta}_{+}^1)}$ and $D=\begin{pmatrix}
 \varpi I_{|\widehat{\beta}_{+}^0|} & 0\\
 0 & (\overline{U}^{\top}\!G\overline{V}_{\!1})_{(\widehat{\beta}_0\cup\eta^3)(\widehat{\beta}_0\cup\eta^3)} \\
 \end{pmatrix}$. Note that $\sigma_i\big[(U_{\widehat{\beta}_0\cup\eta^3}^k)^{\top}\!G^k V_{\widehat{\beta}_0\cup\eta^3}^k\big]$ for $i\in[|\widehat{\beta}_0|]$ is not less than $\sigma_i\big[(U_{\widehat{\beta}_0\cup\eta^3}^k)^{\top}\!G^k V_{\widehat{\beta}_0\cup\eta^3}^k\big]$ for $i=|\widehat{\beta}_0|+1,\ldots,|\widehat{\beta}_0\cup\eta^3|$. Hence, $\lambda_1\big[\mathcal{S}\big((U_{\widehat{\beta}_0\cup\eta^3}^k)^{\top}\!G^k V_{\widehat{\beta}_0\cup\eta^3}^k\big)\big]=\lambda_1\big[\mathcal{S}\big((U_{\widehat{\beta}_0}^k)^{\top}\!G^k V_{\widehat{\beta}_0}^k\big)\big]$ for all $k\in\widehat{\mathcal{K}}$. 
 Passing the limit $\widehat{\mathcal{K}}\ni k\to\infty$ and using the expression of $D$ leads to $\lambda_1(\mathcal{S}(D))\le\lambda_1\big([\mathcal{S}(\overline{U}^{\top}\!G\overline{V}_{\!1})]_{\widehat{\beta}_0\widehat{\beta}_0}\big)$. Together with the expression of $C$ and the above \eqref{temp0-varpi}, we have $\lambda_1(\mathcal{S}(D))\le\varpi\le\lambda_{|\beta_1|}(C)$. 
 Take $A=(\overline{U}^{\top}\!G\overline{V}_{\!1})_{\alpha\alpha}$, $B=  (\overline{U}^{\top}\!G\overline{V}_{\!1})_{\alpha(\eta^1\cup \widehat{\beta}_1\cup \widehat{\beta}_{+}^1)}$ and 
 \begin{align*}
  [E_{11} \ \ E_{12}]&=
 	\begin{pmatrix}
 		0_{|\widehat{\beta}_{+}^0|} & 0\\
 		(\overline{U}^{\top}\!G\overline{V}_{\!1})_{(\widehat{\beta}_0\cup\eta^3)\gamma}& (\overline{U}^{\top}\!G\overline{V}_{\!1})_{(\widehat{\beta}_0\cup\eta^3)c}\end{pmatrix},\\
 [E_{21} \ \ E_{22}]&=[(\overline{U}^{\top}\!G\overline{V}_{\!1})_{\gamma\gamma}\ \ (\overline{U}^{\top}\!G\overline{V}_{\!1})_{\gamma c}].
 \end{align*}
 Along with \eqref{R1Stru1}-\eqref{R1Stru7} and \eqref{CriticalConeCase1}, the relations $\widehat{\alpha}=\alpha\cup\eta^1, \beta_1=\eta^1\cup\widehat{\beta}_1\cup\widehat{\beta}_{+}^1$ and $\beta_0=\widehat{\beta}_{+}^0\cup \widehat{\beta}_0\cup\eta^3$ in \eqref{relation41}-\eqref{relation42}, and the expression of $\Upsilon$ in this case, we have $G\in\Upsilon$.
       
 \noindent     
 {\bf Case 2: $\overline{r}= s(\overline{X})+1$ and $\|\overline{\Gamma}\|_*<\kappa$.} 
 By invoking Lemma \ref{KnorSubdifLemma} (ii) with $(X,\Gamma)=(\overline{X},\overline{\Gamma})$, there exists an integer $\overline{\kappa}_0$ with $0\le \overline{\kappa}_0\le\kappa-1$ such that  
 \begin{align}\label{case2-Xbar}
 \sigma_1(\overline{X})\geq \cdots\geq \sigma_{\overline{\kappa}_0}(\overline{X})>\sigma_{\overline{\kappa}_0+1}(\overline{X})=\cdots=\sigma_{\kappa}(\overline{X})=\cdots=\sigma_n(\overline{X})=0,\\
 \label{case2-Gbar}
 \sigma_{\alpha}(\overline{\Gamma})={\color{blue}e_{\alpha}}\ \ {\rm and}\ \ 
  {\textstyle\sum_{i\in\beta}}\,\sigma_i(\overline{\Gamma})\leq \kappa-\overline{\kappa}_0\ \ {\rm with}\ \ 0\leq\sigma_{\beta}(\overline{\Gamma})\le{\color{blue} e_{\beta}},\qquad
 \end{align}
 where $[\overline{\kappa}_0]=\alpha$ and $\{\overline{\kappa}_0+1,\ldots,n\}=\beta$. Now there exists an infinite index set $\widetilde{\mathcal{K}}\subset\mathcal{K}$ such that $\sigma_{\kappa}(X^k)=0$ for each $k\in\widetilde{\mathcal{K}}$. If not, there will exist  $\overline{k}\in\mathbb{N}$ such that $\sigma_{\kappa}(X^k)>0$ for all $\mathcal{K}\ni k>\overline{k}$. By Lemma \ref{KnorSubdifLemma} (i) with $(X,\Gamma)=(X^k,\Gamma^k)$ for each $k\in\mathcal{K}$, we have $\|\Gamma^k\|_*=\kappa$ for all $\mathcal{K}\ni k>\overline{k}$. On the other hand, from $\lim_{\mathcal{K}\ni k\to\infty}\|\Gamma^k\|_*=\|\overline{\Gamma}\|_*<\kappa$, there exists $\widehat{k}\in\mathbb{N}$ such that $\|\Gamma^k\|_*<\kappa$ for all $\mathcal{K}\ni k>\widehat{k}$. Then, for all $\mathcal{K}\ni k>\max\{\overline{k},\widehat{k}\}$, $\kappa=\|\Gamma^k\|_*<\kappa$, which is impossible. Since $\lim_{\widetilde{\mathcal{K}}\ni k\to\infty}\sigma(X^k)=\sigma(\overline{X})$, there is an integer $\kappa_0$ with $\overline{\kappa}_0\leq \kappa_0\leq \kappa-1$ and an infinite index set $\widehat{\mathcal{K}}\subset\widetilde{\mathcal{K}}$ such that for all $k\in\widehat{\mathcal{K}}$, 
 \begin{align}\label{XjDecrease2}
 \sigma_1(X^k)\geq \cdots\geq \sigma_{\kappa_0}(X^k)>\sigma_{\kappa_0+1}(X^k)=\cdots=\sigma_{\kappa}(X^k)=\cdots= \sigma_n(X^k)= 0,\\
 \sigma_{\widehat{\alpha}}(\Gamma^k)={\color{blue}e_{\widehat{\alpha}}}\ \ {\rm and}\ \ 
 {\textstyle\sum_{i\in \widehat{\beta}}}\,\sigma_i(\Gamma^k)\leq \kappa-\kappa_0\ \ {\rm with}\  \  0\leq \sigma_{\widehat{\beta}}(\Gamma^k)\le{\color{blue} e_{\widehat{\beta}}},\qquad
 \label{XjDecrease3}
 \end{align}
 where $\widehat{\alpha}\!:=\![\kappa_0]$ and $\widehat{\beta}\!:=\!\{\kappa_0\!+1,\ldots,n\}$. 
 Then, following the same arguments as those for Case 1, the relations in \eqref{relation41}-\eqref{relation42} still hold with $\widehat{\gamma}=\emptyset$ and $\eta^3=\emptyset$, i.e.,
 \begin{align}\label{relation421}
  \widehat{\alpha}=\alpha\cup \eta^1,\ 
  \beta=\beta_1\cup\beta_0\cup\beta_{+}=\eta^1\cup\widehat{\beta},\
  \widehat{\beta}=\widehat{\beta}_1\cup\widehat{\beta}_{+}\cup\widehat{\beta}^0,\\
  \label{relation422}
  \widehat{\beta}_+= \widehat{\beta}_{+}^1\cup \beta_+\cup \widehat{\beta}_{+}^0,\
  \beta_1=\eta^1\cup \widehat{\beta}_1\cup \widehat{\beta}_{+}^1,\ \beta_0=\widehat{\beta}_{+}^0\cup \widehat{\beta}_0.
  \end{align} 
 From \eqref{XjDecrease2}, for each $k\in\widehat{\mathcal{K}}$, $s(X^k)$ is indenpendent of $k$ and so is $a_l(X^k)$ for each $l\in[ s(X^k)]$, where $a_l(X^k)$ is the index set defined by \eqref{al} with $X=X^k$. For each $l\in[ s(X^k)]$, we write $\widehat{a}_l:=a_l(X^k)$. From \eqref{XjDecrease2} and $\sigma_{\kappa}(X^k)=0$ for each $k\in\widehat{\mathcal{K}}$, we have $\kappa\in\widehat{a}_{s(X^k)+1}$. Let $r:=s(X^k)+1$. Let $\zeta_1(\Gamma^k)>\cdots>\zeta_{q}(\Gamma^k)$ be the nonzero distinct entries of the set $\{\sigma_i(\Gamma^k)\ |\ i\in \widehat{a}_r\}$, and for each $l\in\{2,\ldots,q\}$, let  $\beta_q(\Gamma^k)$ be defined by \eqref{beta_l} with $\Gamma=\Gamma^k$, which is also independent of $k$. Write $\widehat{\beta}_q:=\beta_q(\Gamma^k)$. Clearly, $\widehat{\beta}=\bigcup_{j=0}^q\widehat{\beta}_j$.
  
 For each $k\in\widehat{\mathcal{K}}$, by Proposition \ref{MainTh2} with $(X,\Gamma)=(X^k,\Gamma^k)$,  $d^2\Psi_{\kappa}(X^k|\Gamma^k)(G^k)$ equals
 \begin{align}\label{case2-3}
  &\sum_{l=1}^{r-1}\sum_{j=1}^{q}\frac{2(1\!-\!\zeta_{j}(\Gamma^k))}{\nu_l(X^k)}\| [\mathcal{S}((U^k)^{\top}\!G^kV_{1}^k)]_{\widehat{a}_{l}\widehat{\beta}_{j}}\|_F^2+\sum_{l=1}^{r-1}\sum_{l'=1}^{r-1}
 	\frac{2\|[\mathcal{T}((U^k)^{\top}\!G^kV_{1}^k)]_{\widehat{a}_{l}\widehat{a}_{l'}}\|_F^2}{\nu_l(X^k)\!+\!\nu_{l'}(X^k)}\nonumber\\
  &+\!\sum_{l=1}^{r-1}\!\frac{2\| [\mathcal{S}((U^k)^{\top}\!G^kV_{1}^k)]_{\widehat{a}_{l}\widehat{\beta}_{0}}\|_F^2}{\nu_l(X^k)}+\!\sum_{l=1}^{r-1}\frac{\|(U^k)^\top_{\widehat{a}_l}G^kV_{c}^k\|_F^2}{\nu_l(X^k)}\\
  & +\!\sum_{l=1}^{r-1}\sum_{j=0}^{q}
 	\frac{2(1\!+\!\zeta_{j}(\Gamma^k))}{\nu_l(X^k)}\|[\mathcal{T}((U^k)^{\top}\!G^kV_{1}^k)]_{\widehat{a}_{l}\widehat{\beta}_{j}}\|_F^2. \nonumber	
 \end{align} 
 Together with $\lim_{\widehat{\mathcal{K}}\ni k\to\infty}d^2\Psi_{\kappa}(X^k|\Gamma^k)(G^k)=0$ by \eqref{SecDerCase1}, we have the following relations
 \begin{subnumcases}{}\label{R2Stru1}
 (\overline{U}^{\top} G \overline{V}_1)_{\widehat{\alpha}\widehat{\alpha}}\in \mathbb{S}^{|\alpha|},\,
  (\overline{U}^{\top}\!G \overline{V}_1)_{\widehat{\alpha}(\widehat{\beta}_1\cup \widehat{\beta}_{+}^1)}=(\overline{U}^{\top}\!G\overline{V}_1)_{(\widehat{\beta}_1\cup \widehat{\beta}_{+}^1)\widehat{\alpha}}^\top,\\ \label{R2Stru2}
  (\overline{U}^{\top}\!G\overline{V}_{\!1})_{\widehat{\alpha} c}=0,\,
  (\overline{U}^{\top}\!G\overline{V}_{\!1})_{\widehat{\alpha}(\beta_+\cup \widehat{\beta}_{+}^0)}=(\overline{U}^{\top}\!G\overline{V}_{\!1})_{(\beta_+\cup \widehat{\beta}_{+}^0)\widehat{\alpha}}^\top=0, \\ 
  \label{R2Stru3}
  (\overline{U}^{\top}\!G\overline{V}_{\!1})_{\widehat{\alpha}\widehat{\beta}_0}=(\overline{U}^{\top}\!G\overline{V}_{\!1})_{\widehat{\beta}_0\widehat{\alpha}}^\top=0,
  \end{subnumcases}
  whose proof is included in Appendix B. In addition, for each $k\in\widehat{\mathcal{K}}$, since $\|\Gamma^k\|_*<\kappa$ and $G^k\in\mathcal{C}_{\Psi_{\kappa}}(X^k,\Gamma^k)$, according to Lemma \ref{lemma-CCone} (ii), we have $\mathcal{S}((U_{\widehat{\beta}_1}^k)^{\top}\!G^kV^k_{\widehat{\beta}_1})\in \mathbb{S}_+^{|\widehat{\beta}_1|}$ and
  \begin{equation*}
   \big[(U_{\widehat{\beta}}^k)^{\top}\!G^k V_{\widehat{\beta}}^k\ \ (U_{\widehat{\beta}}^k)^\top\!G^k V_c^k\big]=\!\left[\begin{matrix}
  	\mathcal{S}((U_{\widehat{\beta}_1}^k)^{\top}\!G^kV^k_{\widehat{\beta}_1})& 0 & 0 & 0\\
  	0 & 0 & 0 & 0 \\
  	0& 0 & 	0 & 0
  \end{matrix}\right].
  \end{equation*}
 Passing the limit $\widehat{\mathcal{K}}\ni k\to\infty$ to the above inclusion and equality leads to 
 \begin{equation}\label{CriticalConeCase2}
  \mathcal{S}(\overline{U}_{\widehat{\beta}_1}^{\top}\!G\overline{V}_{\!\widehat{\beta}_1})\in \mathbb{S}_+^{|\widehat{\beta}_1|}\ \ {\rm and}\ \ \big[\overline{U}_{\widehat{\beta}}^{\top}\!G \overline{V}_{\!\widehat{\beta}}\ \ \overline{U}_{\widehat{\beta}}^\top G \overline{V}_{\!c}\big]
 =\!\left[\begin{matrix}
  		[\mathcal{S}(\overline{U}^{\top}\!G\overline{V}_{\!1})]_{\widehat{\beta}_1\widehat{\beta}_1}& 0 & 0 & 0\\
  		0 & 0 & 0 &0 \\
  		0& 0 & 	0 & 0
  	\end{matrix}\right].
  \end{equation}
  Take $A=(\overline{U}^{\top}\!G\overline{V}_{\!1})_{\alpha\alpha},B=  (\overline{U}^{\top}\!G\overline{V}_{\!1})_{\alpha(\eta^1\cup \widehat{\beta}_1\cup \widehat{\beta}_{+}^1)}$ and $C=(\overline{U}^{\top}\!G\overline{V}_{\!1})_{(\eta^1\cup \widehat{\beta}_1\cup \widehat{\beta}_{+}^1)(\eta^1\cup \widehat{\beta}_1\cup \widehat{\beta}_{+}^1)}$. Along with equations \eqref{R2Stru1}-\eqref{R2Stru3} and \eqref{CriticalConeCase2}, the relations $\widehat{\alpha}=\alpha\cup\eta^1,\beta_1=\eta^1\cup\widehat{\beta}\cup\widehat{\beta}_{+}^1$ and $\beta_0=\widehat{\beta}_{+}^0\cup \widehat{\beta}_0$ in  \eqref{relation421}-\eqref{relation422}, and the definition of $\Upsilon$ in this case, we have $G\in\Upsilon$.
  
 \noindent
 {\bf Case 3: $\overline{r}= s(\overline{X})\!+1$ and $\|\overline{\Gamma}\|_*\!=\kappa$.} In this case, $\gamma=\emptyset$. By Lemma \ref{KnorSubdifLemma} (ii) with $(X,\Gamma)=(\overline{X},\overline{\Gamma})$, there exists an integer $\overline{\kappa}_0$ with $0\le \overline{\kappa}_0\le\kappa\!-\!1$ such that the previous \eqref{case2-Xbar}-\eqref{case2-Gbar} hold. 
 We proceed the proof by the following two cases.
  
 \noindent
 {\bf Subcase 3.1: there is an infinite index set $\widehat{\mathcal{K}}$ such that $\sigma_\kappa(X^k)>0$ for all $k\in\widehat{\mathcal{K}}$.} Note that equations \eqref{XjDecrease1}-\eqref{GamajDecrease1}  and the discussions after them with $\gamma=\emptyset$ are applicable to this case. By combining equation \eqref{LimitCase1} with  $\lim_{\widehat{\mathcal{K}}\ni k\to\infty}d^2\Psi_{\kappa}(X^k|\Gamma^k)(G^k)=0$ leads to \eqref{R1Stru1}-\eqref{R1Stru7}. In addition, from the arguments after \eqref{R1Stru1}-\eqref{R1Stru7}, there exists $\varpi\in\mathbb{R}$ such that equation \eqref{temp0-varpi} holds. Take $A=(\overline{U}^{\top}\!G\overline{V}_{\!1})_{\alpha\alpha},B=  (\overline{U}^{\top}\!G\overline{V}_{\!1})_{\alpha(\eta^1\cup \widehat{\beta}_1\cup \widehat{\beta}_{+}^1)}$, 
 \begin{align*}
 C=(\overline{U}^{\top}\!G\overline{V}_{\!1})_{(\eta^1\cup \widehat{\beta}_1\cup \widehat{\beta}_{+}^1)(\eta^1\cup \widehat{\beta}_1\cup \widehat{\beta}_{+}^1)}\ \ {\rm and}\ \ [D\ \ E]=\begin{pmatrix}
  \varpi I_{|\widehat{\beta}_{+}^0|} & 0\\
  0 & (\overline{U}^{\top}\!G\overline{V})_{(\widehat{\beta}_0\cup\eta^3)(\widehat{\beta}_0\cup\eta^3\cup c)}
 \end{pmatrix}.
\end{align*}
 Note that for each $k\in\widehat{\mathcal{K}}$,  $\lambda_1\big(\mathcal{S}([(U^k)^{\top}G^kV^k]_{\widehat{\beta}_0\widehat{\beta}_0})\big)=\sigma_1\big([(U^k)^{\top}G^kV^k]_{(\widehat{\beta}_0\cup\eta^3)(\widehat{\beta}_0\cup\eta^3\cup c)}\big)$. Passing the limit $\widehat{\mathcal{K}}\ni k\to\infty$ leads to $\lambda_1\big(\mathcal{S}([\overline{U}^{\top}G\overline{V}]_{\widehat{\beta}_0\widehat{\beta}_0})\big)=\sigma_1\big([\overline{U}^{\top}G\overline{V}]_{(\widehat{\beta}_0\cup\eta^3)(\widehat{\beta}_0\cup\eta^3\cup c)}\big)$. Along with the previous \eqref{temp0-varpi} and the expressions of $C$ and $[D\ \ E]$, we have $\sigma_1([D\ \ E])\le\varpi\le\lambda_{|\beta_1|}(C)$.
 Combining this inequality with equations \eqref{R1Stru1}-\eqref{R1Stru7} and \eqref{CriticalConeCase1}, the relations $\widehat{\alpha}=\alpha\cup\eta^1, \beta_1=\eta^1\cup\widehat{\beta}_1\cup\widehat{\beta}_{+}^1,\beta_0=\widehat{\beta}_{+}^0\cup \widehat{\beta}_0\cup\eta^3$ in \eqref{relation41}-\eqref{relation42}, and the definition of $\Upsilon$ in this case, we conclude that $G\in\Upsilon$.
  
 \noindent
 {\bf Subcase 3.2: there is an infinite index set $\widehat{\mathcal{K}}\subset\mathcal{K}$ such that $\sigma_\kappa(X^k)=0$ for all $k\in\widehat{\mathcal{K}}$.} Note that equations \eqref{XjDecrease2}-\eqref{XjDecrease3} and the arguments after them are applicable to this case. By combining equation \eqref{case2-3} with  $\lim_{\widehat{\mathcal{K}}\ni k\to\infty}d^2\Psi_{\kappa}(X^k|\Gamma^k)(G^k)=0$ leads to 
 \begin{subnumcases}{}\label{case32-equa1}
 (\overline{U}^{\top} G \overline{V}_1)_{\widehat{\alpha}\widehat{\alpha}}\in \mathbb{S}^{|\alpha|},\,
 (\overline{U}^{\top}\!G \overline{V}_1)_{\widehat{\alpha}(\widehat{\beta}_1\cup \widehat{\beta}_{+}^1)}=(\overline{U}^{\top}\!G\overline{V}_1)_{(\widehat{\beta}_1\cup \widehat{\beta}_{+}^1)\widehat{\alpha}}^\top,\\ 
 \label{case32-equa2}
 (\overline{U}^{\top}\!G\overline{V}_{\!1})_{\widehat{\alpha} c}=0,\,
 (\overline{U}^{\top}\!G\overline{V}_{\!1})_{\widehat{\alpha}(\beta_+\cup \widehat{\beta}_{+}^0)}=(\overline{U}^{\top}\!G\overline{V}_{\!1})_{(\beta_+\cup \widehat{\beta}_{+}^0)\widehat{\alpha}}^\top=0, \\ 
 \label{case32-equa3}
 (\overline{U}^{\top}\!G\overline{V}_{\!1})_{\widehat{\alpha}\widehat{\beta}_0}=(\overline{U}^{\top}\!G\overline{V}_{\!1})_{\widehat{\beta}_0\widehat{\alpha}}^\top=0. 
 \end{subnumcases}
 In addition, from the arguments after equations \eqref{R2Stru1}-\eqref{R2Stru3}, equation \eqref{CriticalConeCase2} holds. Take $A=(\overline{U}^{\top}\!G\overline{V}_{\!1})_{\alpha\alpha},B=  (\overline{U}^{\top}\!G\overline{V}_{\!1})_{\alpha(\eta^1\cup \widehat{\beta}_1\cup \widehat{\beta}_{+}^1)},C=(\overline{U}^{\top}\!G\overline{V}_{\!1})_{(\eta^1\cup \widehat{\beta}_1\cup \widehat{\beta}_{+}^1)(\eta^1\cup \widehat{\beta}_1\cup \widehat{\beta}_{+}^1)}$, $D=0$ and $E=0$. Together with \eqref{case32-equa1}-\eqref{case32-equa3} and \eqref{CriticalConeCase2}, the relations $\widehat{\alpha}=\alpha\cup\eta^1$, $\beta_1=\eta^1\cup\widehat{\beta}\cup\widehat{\beta}_{+}^1$ and $\beta_0=\widehat{\beta}_{+}^0\cup \widehat{\beta}_0$ in  \eqref{relation421}-\eqref{relation422}, and the definition of $\Upsilon$ in this case, we have $G\in\Upsilon$.
 
 The above arguments establish the inclusion $\mathcal{G}\subset\Upsilon$. In what follows, we prove the converse inclusion $\Upsilon\subset\mathcal{G}$. Pick any $G\in\Upsilon$. We claim that $G\in \mathcal{G}$ by three cases. 
 
 \noindent
 {\bf Case 1: $\overline{r}\in[s(\overline{X})]$}. From $G\in\Upsilon$ and the definition of $\Upsilon$ in this case,  there exist $(U,V)\in\mathbb{O}^{n,m}(\overline{X})\cap\mathbb{O}^{n,m}(\overline{\Gamma}),\begin{pmatrix}
 	A &  B \\
 	B^{\top} & C
 \end{pmatrix}\in\mathbb{S}^{|\alpha|+|\beta_1|}, D\in \mathbb{R}^{|\beta_0|\times |\beta_0|},\begin{pmatrix}
 	E_{11} &   E_{12} \\
 	E_{21} & E_{22}
 \end{pmatrix}\in\mathbb{R}^{(|\beta_0|+|\gamma|)\times (|\gamma|+|c|)}$ and $\varpi\in\mathbb{R}$ such that  
\begin{equation}\label{temp-equa43}
 \lambda_1(\mathcal{S}(D))\le\varpi\le\lambda_{|\beta_1|}(C)\ \ {\rm and}\ \ 	
 U^\top\! GV=
 \begin{pmatrix}
 A & B & 0 & 0 & 0 & 0\\
 B^\top  &  C & 0 &0 & 0 & 0\\
 0 & 0& \varpi I_{|\beta_+|} & 0 & 0  & 0\\
 0 & 0 & 0 & D & E_{11} &   E_{12} \\
 0 & 0 & 0 & 0 & E_{21} & E_{22}
 \end{pmatrix}.
 \end{equation}
 Since $\overline{\Gamma}\in\partial\Psi_{\kappa}(\overline{X})$, there exist integers $\overline{\kappa}_0$ and $\overline{\kappa}_1$ with $0\le \overline{\kappa}_0\le\kappa-1<\kappa\le\overline{\kappa}_1\le n$ such that the previous \eqref{case1-relation1} and \eqref{case1-relation2} hold. For each $k\in\mathbb{N}$, take $\Gamma^k:=\overline{\Gamma}, G^k:=G$ and 
 \begin{equation*}
 X^k:=\overline{X}+U\begin{pmatrix}
  0_{\alpha\alpha} & 0 & 0\\
  0 & \frac{1}{k}I_{|\beta_1|} & 0\\
  0 & 0 & 0
 \end{pmatrix}V^\top.
\end{equation*}
Obviously, for sufficiently large $k$, $\sigma_i(X^k)=\sigma_i(\overline{X})+\frac{1}{k}$ for $i\in\beta_1$, and $\sigma_i(X^k)=\sigma_i(\overline{X})$ for $i\in[n]\backslash\beta_1$. Note that, when $k$ is large enough, the set $\big\{i\in\!\beta\,|\,\sigma_i(X^k)=\sigma_{\kappa}(X^k)\big\}$ is independent of $k$, so we denote it as $\widehat{\beta}$. Write $\widehat{\beta}_1\!:=\{i\in \widehat{\beta}\,|\, \sigma_{i}(\Gamma^k)=1\},\,\widehat{\beta}_+\!:=\{i\in \widehat{\beta}\,|\, 0<\sigma_{i}(\Gamma^k)<1\}$ and $\widehat{\beta}_0\!:=\{i\in \widehat{\beta}\,|\, \sigma_{i}(\Gamma^k)=0\}$. In view of Lemma \ref{KnorSubdifLemma} (i), for sufficiently large $k$, $(X^k,\Gamma^k)\in {\rm gph}\,\partial \Psi_{\kappa}$. Let $r\in[s(X^k)]$ be the integer such that $\kappa\in\widehat{a}_{r}:=a_{r}(X^k)$. We claim that for sufficiently large $k$, $G^k\in \mathcal{C}_{\Psi_{\kappa}}(X^k,\Gamma^k)$. Indeed, when $\kappa\notin\beta_1$, we have $\widehat{\beta}=\beta_{+}\cup\beta_0$, which along with $\Gamma^k\equiv\overline{\Gamma}$ implies  $\widehat{\beta}_{+}=\beta_{+}$ and $\widehat{\beta}_{0}=\beta_{0}$. Together with  $\widehat{\beta}\subset\beta=\beta_{1}\cup\beta_{+}\cup\beta_{0}$, we infer that $\widehat{\beta}_{1}\subset\beta_{1}=\emptyset$. This along with $G^k\equiv G$ and the above \eqref{temp-equa43} implies that $D=U_{\widehat{\beta}_0}^{\top}G^kV_{\widehat{\beta}_0}$, $\lambda_1(\mathcal{S}(D))\le\varpi$ and $ \mathcal{S}(U^{\top}_{\!\widehat{a}_{r}}G^kV_{\!\widehat{a}_{r}})=\left[\begin{matrix}
 	  \varpi I_{|\widehat{\beta}_{+}|} & 0 \\
 	  0 & \mathcal{S}(D)
 \end{matrix}\right]$ for sufficiently large $k$. Invoking Lemma \ref{lemma-CCone} (i) with $(X,\Gamma)=(X^k,\Gamma^k)$ for large enough $k$ yields the claimed inclusion. When $\kappa\in\beta_1$, we have $\widehat{\beta}=\beta_1$, which along with $\Gamma^k\equiv\overline{\Gamma}$ implies $\widehat{\beta}_{+}=\emptyset=\widehat{\beta}_0$, so  $\beta_1=\widehat{\beta}_1$. Together with the above \eqref{temp-equa43}, it follows $C=U_{\widehat{\beta}_1}^{\top}G^kV_{\widehat{\beta}_1}\in\mathbb{S}^{|\widehat{\beta}_1|}$, $\varpi\le\lambda_{|\widehat{\beta}_1|}(C)$ and  $\mathcal{S}(U^{\top}_{\!\widehat{a}_{r}}G^kV_{\!\widehat{a}_{r}})=C$. Invoking Lemma \ref{lemma-CCone} (i) with $(X,\Gamma)=(X^k,\Gamma^k)$ yields the claimed inclusion. Now using Corollary \ref{corollary3.3} with $(X,\Gamma,G)=(X^k,\Gamma^k,G^k)$ for sufficiently large $k$ and the expression of $U^{\top}G^kV$ in \eqref{temp-equa43}, we infer that $d^2\Psi_{\kappa}(X^k|\Gamma^k)(G^k)=0$ for large enough $k$. Along with $(X^k,\Gamma^k,G^k)\to (\overline{X},\overline{\Gamma},G)$ and $(X^k,\Gamma^k)\in {\rm gph}\,\partial \Psi_{\kappa}$ for sufficiently large $k$, it holds $G\in \mathcal{G}$.

 \noindent	
 {\bf Case 2: $\overline{r}= s(\overline{X})+1$ and $\|\overline{\Gamma}\|_*<\kappa$}. From $G\in\Upsilon$ and the definition of $\Upsilon$ for this case,  there exist $(U,V)\in\mathbb{O}^{n,m}(\overline{X})\cap\mathbb{O}^{n,m}(\overline{\Gamma}),A\in\mathbb{S}^{|\alpha|}, B\in \mathbb{R}^{|\alpha|\times |\beta_1|}$ and $C\in \mathbb{S}^{|\beta_1|}$ such that
 \begin{equation}\label{temp-equa44}
 U^\top GV=\begin{pmatrix}
 			A & B & 0 & 0 & 0 \\
 			B^\top  &  C & 0 &0 & 0 \\
 			0 & 0& 0 & 0 & 0  \\
 			0 & 0 & 0 & 	0& 0 
 		\end{pmatrix}.
 \end{equation} 
 Since $\overline{\Gamma}\in\partial\Psi_{\kappa}(\overline{X})$, there exist integers $\overline{\kappa}_0$ and $\overline{\kappa}_1$ with $0\le \overline{\kappa}_0\le\kappa-1<\kappa\le\overline{\kappa}_1\le n$ such that the previous \eqref{case2-Xbar} and \eqref{case2-Gbar} hold.
 For each $k\in\mathbb{N}$, we take $(X^k,\Gamma^k,G^k)$ in the same way as in Case 1. Let $\widehat{\beta}_1, \widehat{\beta}_{+}$ and $\widehat{\beta}_0$ be the index sets defined as in Case 1. Let $r\in[s(X^k)+1]$ be the integer such that $\kappa\in \widehat{a}_{r}:=a_{r}(X^k)$. We claim that for sufficiently large $k$, $G^k\in \mathcal{C}_{\Psi_{\kappa}}(X^k,\Gamma^k)$. Indeed, when $\kappa\notin\beta_1$, in view of the arguments for Case 1, we have  $\widehat{\beta}_1=\emptyset,\widehat{\beta}_{+}=\beta_{+}$ and $\widehat{\beta}_{0}=\beta_{0}$, which implies $r=s(X^k)+1$. Together with $G^k\equiv G$ and the above \eqref{temp-equa44}, we have $[U^{\top}_{\!\widehat{a}_{r}}G^kV_{\!\widehat{a}_{r}}\ \ U^{\top}_{\!\widehat{a}_{r}}G^kV_{c}]=0$. According to Lemma \ref{lemma-CCone} (ii) with $(X,\Gamma)=(X^k,\Gamma^k)$ for sufficiently large $k$, the claimed conclusion holds. When $\kappa\in\beta_1$, in view of the arguments for Case 1, we have $\widehat{\beta}_{+}=\emptyset=\widehat{\beta}_0$ and $\widehat{\beta}=\widehat{\beta}_1=\beta_1$, which implies $r\in[s(X^k)]$. Together with $G^k\equiv G$ and the above \eqref{temp-equa44}, we have $U^{\top}_{\!\widehat{a}_{r}}G^kV_{\!\widehat{a}_{r}}=C\in\mathbb{S}^{|\widehat{\beta}_1|}$. Take $\varpi=\lambda_{|\widehat{\beta}_1|}(C)$. Then, according to Lemma \ref{lemma-CCone} (i) with $(X,\Gamma)=(X^k,\Gamma^k)$ for sufficiently large $k$, the claimed inclusion also holds. Now using Corollary \ref{corollary3.3} with $(X,\Gamma,G)=(X^k,\Gamma^k,G^k)$ for sufficiently large $k$ and noting that $U^{\top}G^kV$ has the expression in \eqref{temp-equa44}, we infer that $d^2\Psi_{\kappa}(X^k|\Gamma^k)(G^k)=0$ for sufficiently large $k$. Along with $(X^k,\Gamma^k,G^k)\to (\overline{X},\overline{\Gamma},G)$ and $(X^k,\Gamma^k)\in {\rm gph}\,\partial \Psi_{\kappa}$ for sufficiently large $k$, we conclude that $G\in \mathcal{G}$.

 \noindent
 {\bf Case 3: $\overline{r}= s(\overline{X})\!+\!1$ and $\|\overline{\Gamma}\|_*=\kappa$}. From $G\in\Upsilon$ and the definition of $\Upsilon$ for this case, there exist $(U,V)\!\in\mathbb{O}^{n,m}(\overline{X})\cap\mathbb{O}^{n,m}(\overline{\Gamma}), A\in \mathbb{S}^{|\alpha|},B\!\in\!\mathbb{R}^{|\alpha|\times |\beta_1|},C\in \mathbb{S}^{|\beta_1|}, D\in\!\mathbb{R}^{|\beta_0|\times |\beta_0|}, E\in\! \mathbb{R}^{|\beta_0|\times |c|}$ and $\varpi\in\mathbb{R}$ such that
  \begin{equation}\label{temp-equa45}
  \sigma_1([D\ \ E])\le\varpi\le\lambda_{|\beta_1|}(C)\ \ {\rm and}\ \ 
  U^\top GV=
  \begin{pmatrix}
  	A & B & 0 & 0 & 0 \\
  	B^\top  &  C & 0 &0 & 0 \\
  	0 & 0& \varpi I_{|\beta_+|} & 0 & 0  \\
  	0 & 0 & 0 & 	D& E 
  \end{pmatrix}.
  \end{equation}
 Since $\overline{\Gamma}\in\partial\Psi_{\kappa}(\overline{x})$, there exist integers $\overline{\kappa}_0$ and $\overline{\kappa}_1$ with $0\le \overline{\kappa}_0\le\kappa-1<\kappa\le\overline{\kappa}_1\le n$ such that the previous \eqref{case2-Xbar} and \eqref{case2-Gbar} hold.
 For each $k\in\mathbb{N}$, we take $(X^k,\Gamma^k,G^k)$ in the same way as in Case 1. Let $\widehat{\beta}_1, \widehat{\beta}_{+}$ and $\widehat{\beta}_0$ be the index sets defined in the same way as in Case 1. Let $r\in[s(X^k)+1]$ be such that $\kappa\in\widehat{a}_{r}:=a_{r}(X^k)$. We claim that for sufficiently large $k$, $G^k\in \mathcal{C}_{\Psi_{\kappa}}(X^k,\Gamma^k)$. Indeed, when $\kappa\notin\beta_1$, in view of the arguments for Case 1, we have $\widehat{\beta}_1=\emptyset,\widehat{\beta}_{+}=\beta_{+}$ and $\widehat{\beta}_{0}=\beta_{0}$, which implies $r=s(X^k)+1$. Together with $G^k\equiv G$ and the above \eqref{temp-equa45}, we have $\sigma_1([D\ \ E])\le\varpi$ and $
  [U^{\top}_{\!\widehat{a}_{r}}G^kV_{\!\widehat{a}_{r}}\ \ U^{\top}_{\!\widehat{a}_{r}}G^kV_c]=\left[\begin{matrix}
  	\varpi I_{|\widehat{\beta}_{+}|} & 0 & 0 \\
  	0 & D & E
 \end{matrix}\right]$. According to Lemma \ref{lemma-CCone} (iii), the claimed inclusion holds. When $\kappa\in\beta_1$, in view of the arguments for Case 1, we have $\widehat{\beta}_{+}=\emptyset=\widehat{\beta}_0$ and $\widehat{\beta}=\widehat{\beta}_1=\beta_1$, which implies $r\in[s(X^k)]$. Together with $G^k=G$ and the above \eqref{temp-equa45}, we have $\varpi\le\lambda_{|\beta_1|}(C)$ and $\mathcal{S}(U^{\top}_{\!\widehat{a}_{r}}G^kV_{\!\widehat{a}_{r}})=C\in\mathbb{S}^{|\widehat{\beta}_1|}$. According to Lemma \ref{lemma-CCone} (i) with $(X,\Gamma)=(X^k,\Gamma^k)$ for sufficiently large $k$, the claimed inclusion holds. Now using Corollary \ref{corollary3.3} with $(X,\Gamma,G)=(X^k,\Gamma^k,G^k)$ for sufficiently large $k$ and noting that $U^{\top}G^kV$ has the expression in \eqref{temp-equa45}, we infer that $d^2\Psi_{\kappa}(X^k|\Gamma^k)(G^k)=0$ for sufficiently large $k$. Along with $(X^k,\Gamma^k,G^k)\to (\overline{X},\overline{\Gamma},G)$ and $(X^k,\Gamma^k)\in {\rm gph}\,\partial \Psi_{\kappa}$ for large enough $k$, we conclude that $G\in \mathcal{G}$. 
 \end{proof}
 
 As a byproduct of Theorem \ref{ThTilt}, we get the sufficient and necessary condition for the tilt-stability of the nuclear norm regularized problem established in \cite[Theorem 4.6]{Nghia2025}. 
 \begin{corollary}\label{Example1}
 Let $\overline{X}$ be a local optimal solution of problem \eqref{KnormRegular} with $\kappa=n$, and write $\overline{\Gamma}\!:=-\nu\nabla\vartheta(\overline{X})$. Suppose that $\nabla^2\vartheta(\cdot)$ is  positive semidefinite  on an open neighborhood $\mathcal{N}$ of $\overline{X}$. Then, $\overline{X}$ is a tilt-stable solution of \eqref{KnormRegular} if and only if ${\rm Ker}\,\nabla^2\vartheta(\overline{X})\cap \Upsilon=\{0\}$ with 
 \[
  \Upsilon=\bigg\{\overline{U}
 	\begin{pmatrix}
 		Z & 0\\
 		0  &  0
 	\end{pmatrix}\overline{V}^\top\,|\, (\overline{U},\overline{V})\in\mathbb{O}^{n,m}(\overline{X})\cap\mathbb{O}^{n,m}(\overline{\Gamma}),Z\in \mathbb{S}^{|\alpha\cup\beta_1|}\bigg\}.
  \] 
 \end{corollary}
 \begin{proof}
 Note that $\overline{\Gamma}\in\partial\Psi_{\kappa}(\overline{X})$. If $\sigma_n(\overline{X})>0$, by Lemma \ref{KnorSubdifLemma} (i), $\|\overline{\Gamma}\|_*=n$, $\alpha\cup\beta_1=[n],\beta_1=\beta$ and $\beta_{+}=\beta_0=\gamma=\emptyset$. The desired result follows by  the first part of Theorem \ref{ThTilt}. If $\sigma_n(\overline{X})=0$ and $\|\overline{\Gamma}\|_*<n$, then $\gamma=\emptyset$ and  the desired result follows the second part of Theorem \ref{ThTilt}. If $\sigma_n(\overline{X})=0$ and $\|\overline{\Gamma}\|_*=n$, by Lemma \ref{KnorSubdifLemma} (ii),  $\beta_+=\beta_0=\emptyset=\gamma$, so $\alpha\cup \beta_1=[n]$. The desired result follows the last part of Theorem \ref{ThTilt}.	
 \end{proof}

 \begin{corollary}\label{corllary42}
 If $\kappa=1$ and $\overline{\Gamma}\in \partial \|\overline{X}\|$, then $\overline{X}$ is a tilt-stable solution of problem \eqref{KnormRegular} if and only if
 ${\rm Ker}\,\nabla^2\vartheta(\overline{X})\cap \Upsilon=\{0\}$	with 
 \begin{align*}
 &\Upsilon:=\left\{W= \overline{U}\left[
 \begin{matrix}
  Z& 0 &0\\
  0  & D & E\\
  0  & 0 & F 
  \end{matrix}
 \right]\overline{V}^\top\ \bigg|\ D\in \mathbb{R}^{|\beta_0|\times |\beta_0|}, \begin{pmatrix}
 	E\\ F
 \end{pmatrix}\in \mathbb{R}^{{|\beta_0|+|\gamma|}\times (|\gamma|+|c|)}, Z\in\mathbb{R}\ {\rm if}\right.\\ &\qquad\qquad\qquad\qquad\qquad\qquad\qquad\quad  \beta_1\ne\emptyset,{\rm otherwise}\ Z=\varpi I_{|\beta_{+}|}\ {\rm with}\ \lambda_1(\mathcal{S}(D))\le\varpi\bigg\}.
 \end{align*}
 \end{corollary}
 \begin{remark}\label{main-remark41}
 Note that the conjugate $\Psi^*_{\kappa}$ of $\Psi_{\kappa}$ is $\mathcal{C}^2$-cone reducible by \cite{Cui2017}. One referee asked if ${\rm par\,}\partial\Psi^*_{\kappa}(-\nu\nabla\! \vartheta(\overline{X}))$ is exactly the set $\Upsilon$ in Theorem 4.1. Now we take $\kappa=1$ for example to show that the set $\Upsilon$ in Theorem 4.1 will be strictly smaller than ${\rm par\,}\partial\Psi^*_{\kappa}(-\nu\nabla\! \vartheta(\overline{X}))$ in some cases. When $\kappa=1$, $\Psi_{\kappa}^*(\cdot)=\delta_{C}(\cdot)$, the indictor of the set $C\!:=\{Z\in\mathbb{R}^{n\times m}\,|\, \|Z\|_*\leq 1\}$. Let $\overline{\Gamma}:=-\nu\nabla\vartheta(\overline{X})$. Then, in view of \cite[Theorem 10.3]{RW98}, 
 \[
 \partial\Psi_{\kappa}^*(\overline{\Gamma})
 =\left\{\begin{array}{cl}
 	\partial\|\overline{\Gamma}\|_{*} &{\rm if}\ \|\overline{\Gamma}\|_*=1,\\
 	\{0\} &{\rm if}\ \|\overline{\Gamma}\|_*<1. 	
 \end{array}\right. 
 \] 
 When $\overline{\Gamma}$ satisfies $\|\overline{\Gamma}\|_*<1$, it is obvious that ${\rm par\,}\partial \Psi^*_{\kappa}(\overline{\Gamma})=\{0\}$. In this case, we have $\overline{X}=0$ by $\overline{X}\in\partial\Psi^*_{\kappa}(\overline{\Gamma})$, so Theorem 4.1 for the case $\overline{r}=s(\overline{X})+1$ and $\|\overline{\Gamma}\|_*<1$ implies $\Upsilon=\{0\}$ due to $\alpha=\beta_1=\emptyset$. That is, when $\overline{\Gamma}$ satisfies $\|\overline{\Gamma}\|_*<1$, ${\rm par\,}\partial \Psi^*_{\kappa}(\overline{\Gamma})=\Upsilon$. Next we take a look at the case that $\|\overline{\Gamma}\|_*=1$. Now the set $\partial\Psi_{\kappa}^*(\overline{\Gamma})=\partial\|\overline{\Gamma}\|_*$ takes the form of
 \[
 \partial\|\overline{\Gamma}\|_*=\left\{[\overline{U}_{+}\ \ \overline{U}_0]\begin{pmatrix}
 	I_{r} & 0\\
 	0  & W
 \end{pmatrix}[\overline{V}\ \ \overline{V}_0]^{\top}\,|\,\|W\|\le 1\right\}\ \ {\rm with}\ r={\rm rank}(\overline{\Gamma}).
 \]
 While the set $\Upsilon$ has the form in Corollary \ref{corllary42}. Obviously, when $\beta_1=\emptyset$, the set $\Upsilon$ is not a subspace in $\mathbb{R}^{n\times m}$, but ${\rm par}\,\partial\Psi_{\kappa}^*(\overline{\Gamma})$ is necessarily a subspace. Then, it is impossible for $\Upsilon$ to be the same as in ${\rm par}\,\partial\Psi_{\kappa}^*(\overline{\Gamma})$. It is not hard to check that $\Upsilon$ is a subset of ${\rm par}\,\partial\Psi_{\kappa}^*(\overline{\Gamma})$. This shows that $\Upsilon$ is strictly contained in ${\rm par}\,\partial\Psi_{\kappa}^*(\overline{\Gamma})$. To sum up, for $\kappa=1$, 
 \begin{align*}
 {\rm Ker}\,\nabla^2\vartheta(\overline{X})\cap {\rm par}\,\partial\Psi_{\kappa}^*(\overline{\Gamma})=\{0\} \Longrightarrow{\rm Ker}\,\nabla^2\vartheta(\overline{X})\cap \Upsilon=\{0\}\Longrightarrow \overline{X}\ {\rm is\ tilt\ stable}\qquad\qquad\\
 \Longrightarrow{\rm Ker}\,\nabla^2\vartheta(\overline{X})\cap {\rm par}\,\partial\Psi_{\kappa}^*(\overline{\Gamma})=\{0\},\qquad
 \end{align*}
 so our condition for $\overline{X}$ to be a tilt-stable solution is weaker than the one obtained in \cite{Cui2024}.
 \end{remark}
 \section{Conclusion}\label{sec5}

 We established a necessary and sufficient condition for the tilt stability of a local minimizer of problem \eqref{prob} by operating directly the second subderivative of $g$. Our condition removes the task involved in the condition of \cite{Nghia2025} to find a set $\mathcal{M}$ such that $g$ satifies the quadratic-growth for it and $\partial g$ has a relative approximations onto this set. In particular, by applying the obtained sufficient and necessary condition and leveraging the second subderivative of the Ky-Fan $\kappa$-norm, we derive a checkable criterion to identify the tilt stability of a local minimizer for problem \eqref {KnormRegular}. This criterion is also demonstrated to be weaker than the one obtained in \cite{Cui2024} in some cases.


 \bigskip
 \noindent
 {\bf\large Appendix A: Proof of equations \eqref{R1Stru1}-\eqref{R1Stru7}.} Recall that $\lim_{\widehat{\mathcal{K}}\ni k\to\infty}\sigma(X^k)=\sigma(\overline{X})$. By the definition of $\widehat{a}_l$ and equation \eqref{XjDecrease1}, if necessary by taking an infinite subset of $\widehat{\mathcal{K}}$, there must exist integer $r_0$ and $r_1$ with $1\leq r_0<r<r_1\leq  s(X^k)+1$ such that 
 \begin{equation}\label{AlphGammTemp1}
 \bigcup\limits_{l'=1}^{\overline{r}-1} \overline{a}_{l'}\!=\!\alpha\!=\!\bigcup\limits_{l=1}^{r_0} \widehat{a}_{l}, \eta^1\!=\!\bigcup\limits_{l=r_0+1}^{r-1} \widehat{a}_{l},
 \eta^3\!=\!\bigcup\limits_{l=r+1}^{r_1} \widehat{a}_{l},\  
 \bigcup\limits_{l'=\overline{r}+1}^{ s(\overline{X})+1}\overline{a}_{l'}\!=\gamma=\!\bigcup\limits_{l=r_1+1}^{ s(X^k)+1} \widehat{a}_{l}.
 \end{equation}
 Using equation \eqref{AlphGammTemp1}, the previous \eqref{LimitCase1}, and $\lim_{\widehat{\mathcal{K}}\ni k\to\infty}d^2\Psi_{\kappa}(X^k|\Gamma^k)(G^k)=0$ yields
 \begin{align*}
 0&=\lim_{\widehat{\mathcal{K}}\ni k\to \infty}\sum_{l=1}^{r-1}\sum_{l'=1}^{ s(X^k)+1}\frac{\|[\mathcal{T}((U^k)^{\top}\!G^k V_1^k)]_{\widehat{a}_{l}\widehat{a}_{l'}}\|_F^2}{\nu_l(X^k)\!+\!\nu_{l'}(X^k)}\\
 &=\lim_{\widehat{\mathcal{K}}\ni k\to \infty}\sum_{i\in\alpha\cup \eta^1}\sum_{j\in [n]}\frac{\|[\mathcal{T}((U^k)^{\top}\!G^k V_1^k)]_{ij}\|_F^2}{\sigma_i(X^k)\!+\!\sigma_j(X^k)}=\!\sum_{i\in \alpha\cup \eta^1}\sum_{j\in [n]}\frac{\|[\mathcal{T}(\overline{U}^{\top}\!G \overline{V}_{\!1})]_{ij}\|_F^2}{\sigma_i(\overline{X})\!+\!\sigma_j(\overline{X})}
\end{align*}
 and
 \begin{align*}
 0&=\lim_{\widehat{\mathcal{K}}\ni k\to \infty}\sum_{j=1}^{q}\sum_{l'=1}^{s(X^k)+1} 	\frac{\zeta_{j}(\Gamma^k)\|[\mathcal{T}((U^k)^{\top}\! G^kV_1^k)]_{\widehat{\beta}_j \widehat{a}_{l'}}\|_F^2}{\nu_{r}(X^k)+\nu_{l'}(X^k)}\\
 &=\lim_{\widehat{\mathcal{K}}\ni k\to \infty}\sum_{j\in \widehat{\beta}\setminus  \widehat{\beta}_0}\sum_{i\in [n]} 	\frac{\zeta_{j}(\Gamma^k)\|[\mathcal{T}((U^k)^{\top}\! G^kV_1^k)]_{ji}\|_F^2}{\nu_{r}(X^k)+\sigma_i(X^k)}\\
 &=\!\sum_{j\in \widehat{\beta}\setminus \widehat{\beta}_0}\!\sum_{i\in [n]} 	\frac{\zeta_{j}(\overline{\Gamma})\|[\mathcal{T}(\overline{U}^{\top}\! G\overline{V}_{\!1})]_{ji}\|_F^2}{\nu_{\overline{r}}(\overline{X})+\sigma_{i}(\overline{X})}.
 \end{align*} 
 Along with $\widehat{\beta}\setminus \widehat{\beta}_0\!=\!\widehat{\beta}_1\cup (\widehat{\beta}^1_+\cup \beta_+\cup \widehat{\beta}^0_+)$ and  $[n]=\widehat{\alpha}\cup\widehat{\beta}\cup\widehat{\gamma}$, we get \eqref{R1Stru1}-\eqref{R1Stru2}, and
 \begin{align}\label{SysBolockTemp1}
 [\mathcal{T}(\overline{U}^{\top}\! G\overline{V}_{\!1})]_{\widehat{\alpha}(\widehat{\beta}^0_+\cup\widehat{\beta}_0\cup \widehat{\gamma})}=0\ \ {\rm and}\ \ [\mathcal{T}(\overline{U}^{\top}\! G\overline{V}_{\!1})]_{(\widehat{\beta}_1\cup \widehat{\beta}^1_+\cup \beta_+)\widehat{\gamma}}=0.
 \end{align}
 By using \eqref{AlphGammTemp1}, the previous \eqref{LimitCase1} and the limit $\lim_{\widehat{\mathcal{K}}\ni k\to\infty}d^2\Psi_{\kappa}(X^k|\Gamma^k)(G^k)=0$ again,  
 \begin{align*}
 0&=\lim_{\widehat{\mathcal{K}}\ni k\to \infty}\sum_{l=1}^{r-1}\sum_{l'=r+1}^{ s(X^k)+1}
 	\frac{\|[\mathcal{S}((U^k)^{\top}\!G^k V_1^k)]_{\widehat{a}_{l}\widehat{a}_{l'}}\|_F^2}{\nu_l(X^k)\!-\!\nu_{l'}(X^k)}\\
 	&=\lim_{\widehat{\mathcal{K}}\ni k\to \infty}\!\bigg[\sum_{l=1}^{r_0}\!\sum_{l'=r+1}^{r_1}\!   \frac{\|[\mathcal{S}((U^k)^{\top}\!G^k V_1^k)]_{\widehat{a}_{l}\widehat{a}_{l'}}\|_F^2}{\nu_l(X^k)\!-\!\nu_{l'}(X^k)}\!+\!\sum_{l=1}^{r_0}\!\sum_{l'=r_1+1}^{ s(X^k)+1}\!
 	\frac{\|[\mathcal{S}((U^k)^{\top}\!G^k V_1^k)]_{\widehat{a}_{l}\widehat{a}_{l'}}\|_F^2}{\nu_l(X^k)\!-\!\nu_{l'}(X^k)}\bigg]\\
 	&\qquad+\lim_{\widehat{\mathcal{K}}\ni k\to \infty}\sum_{l=r_0+1}^{r-1}\!\sum_{l'=r+1}^{r_1}
 	\!\frac{\|[\mathcal{S}((U^k)^{\top}\!G^k V_1^k)]_{\widehat{a}_{l}\widehat{a}_{l'}}\|_F^2}{\nu_l(X^k)\!-\!\nu_{l'}(X^k)}\\
 	&\qquad+\lim_{\widehat{\mathcal{K}}\ni k\to \infty}\sum_{l=r_0+1}^{r-1}\!\sum_{l'=r_1+1}^{ s(X^k)+1}
 	\frac{\|[\mathcal{S}((U^k)^{\top}\!G^k V_1^k)]_{\widehat{a}_{l}\widehat{a}_{l'}}\|_F^2}{\nu_l(X^k)\!-\!\nu_{l'}(X^k)}\\
 	&=\sum_{l=1}^{\overline{r}-1}  
 	\frac{\|[\mathcal{S}(\overline{U}^{\top}\!G \overline{V}_{\!1})]_{\overline{a}_{l}\eta^3}\|_F^2}{\nu_l(\overline{X})\!-\!\nu_{r}(\overline{X})}+\sum_{l=1}^{\overline{r}-1}\sum_{{\color{blue}l'=\overline{r}+1}}^{ s(\overline{X})+1}
 	\frac{\|[\mathcal{S}(\overline{U}^{\top}\!G {\overline{V}}_{1})]_{\overline{a}_{l}\overline{a}_{l'}}\|_F^2}{\nu_l(\overline{X})\!
 		-\!\nu_{l'}(\overline{X})}\\
 	&\quad+\lim_{\widehat{\mathcal{K}}\ni k\to \infty}\sum_{l=r_0+1}^{r-1}\!\sum_{l'=r+1}^{r_1}
 	\!\frac{\|[\mathcal{S}((U^k)^{\top}\!G^k V_1^k)]_{\widehat{a}_{l}\widehat{a}_{l'}}\|_F^2}{\nu_l(X^k)\!-\!\nu_{l'}(X^k)}+\sum_{{\color{blue}l'=\overline{r}+1}}^{ s(\overline{X})+1}
 	\frac{\|[\mathcal{S}(\overline{U}^{\top}\!G\overline{V}_{\!1})]_{\eta^1 \overline{a}_{l'}}\|_F^2}{\nu_r(\overline{X})\!
 		-\!\nu_{l'}(\overline{X})},
 \end{align*}
 which implies that $[\mathcal{S}(\overline{U}^{\top}\!G \overline{V}_{\!1})]_{\alpha\eta^3}=0, [\mathcal{S}(\overline{U}^{\top}\!G \overline{V}_{\!1})]_{\alpha\gamma}=0, [\mathcal{S}(\overline{U}^{\top}\!G \overline{V}_{\!1})]_{\eta^1\gamma}=0$ and
 \[
 \lim_{\widehat{\mathcal{K}}\ni k\to \infty}\sum_{l=r_0+1}^{r-1}\!\sum_{l'=r+1}^{r_1}
 \!\frac{\|[\mathcal{S}((U^k)^{\top}\!G^k V_1^k)]_{\widehat{a}_{l}\widehat{a}_{l'}}\|_F^2}{\nu_l(X^k)-\nu_{l'}(X^k)}=0. 
 \]
 Note that
 $\lim_{\widehat{\mathcal{K}}\ni k\to \infty}\sigma_l(X^k)=\lim_{\widehat{\mathcal{K}}\ni k\to\infty}\sigma_{l'}(X^k)=\sigma_{r}(\overline{X})$ for all $l,l'\in \eta^1\cup\eta^3$. The above limit implies
 $0=\lim_{\widehat{\mathcal{K}}\ni k\to \infty}\|[\mathcal{S}((U^k)^{\top}\!G^k V_1^k)]_{\eta^1\eta^3}\|_F^2=[\mathcal{S}(\overline{U}G\overline{V}_1)]_{\eta^1\eta^3}=0$. Then, $[\mathcal{S}(\overline{U}^{\top}\!G \overline{V}_{\!1})]_{\widehat{\alpha}\widehat{\gamma}}=0$, which along with the first equality of \eqref{SysBolockTemp1} leads to \eqref{R1Stru5}. Similarly, using \eqref{AlphGammTemp1}, the previous \eqref{LimitCase1}, and $\lim_{\widehat{\mathcal{K}}\ni k\to\infty}d^2\Psi_{\kappa}(X^k|\Gamma^k)(G^k)=0$ yields  
 \begin{align*}
 0&=\lim_{\widehat{\mathcal{K}}\ni k\to \infty}\sum_{l=1}^{r-1}\sum_{j=0}^{q}	\frac{1\!-\!\zeta_j(\Gamma^k)}{\nu_l(X^k)\!-\!\nu_{r}(X^k)}\|[\mathcal{S}((U^k)^{\top}\!G^kV_1^k)]_{\widehat{a}_{l}\widehat{\beta}_{j}}\|_F^2\\
 &=\lim_{\widehat{\mathcal{K}}\ni k\to \infty}\sum_{l=1}^{r_0}\sum_{j=0}^{q} 	\frac{1\!-\!\zeta_j(\Gamma^k)}{\nu_l(X^k)\!-\!\nu_{r}(X^k)}\|[\mathcal{S}((U^k)^{\top}\!G^kV_1^k)]_{\widehat{a}_{l}\widehat{\beta}_{j}}\|_F^2\\
 &\quad+\lim_{\widehat{\mathcal{K}}\ni k\to \infty}\sum_{l=r_0+1}^{r-1}\sum_{j=0}^{q} 	\frac{1\!-\!\zeta_j(\Gamma^k)}{\nu_l(X^k)\!-\!\nu_{r}(X^k)}\| [\mathcal{S}((U^k)^{\top}\!G^kV_1^k)]_{\widehat{a}_{l}\widehat{\beta}_{j}}\|_F^2\\
 &=\sum_{i\in \alpha}\sum_{j\in \widehat{\beta}}
 \frac{1\!-\!\sigma_j(\overline{\Gamma})}{\sigma_i(\overline{X})\!-\!\nu_{r}(\overline{X})}\| [\mathcal{S}(\overline{U}^{\top}\!G \overline{V}_{\!1})]_{ij}\|_F^2\\
 &\quad+\lim_{\widehat{\mathcal{K}}\ni k\to \infty}\sum_{i\in \eta^1}\sum_{j\in \widehat{\beta}\setminus\widehat{\beta}_1} 	\frac{1\!-\!\sigma_j(\Gamma^k)}{\sigma_i(X^k)\!-\!\nu_{r}(X^k)}\| [\mathcal{S}((U^k)^{\top}\!G^kV_1^k)]_{ij}\|_F^2,
 \end{align*}
 which means that $[\mathcal{S}(\overline{U}^{\top}\!G\overline{V}_{\!1})]_{\alpha[\widehat{\beta}\setminus(\widehat{\beta}_1\cup \widehat{\beta}^+_1)]}=0$ and 
 $[\mathcal{S}(\overline{U}^{\top}\!G\overline{V}_{\!1})]_{ \eta^1[\widehat{\beta}\setminus(\widehat{\beta}_1\cup \widehat{\beta}^+_1)]}=0$, i.e.,
 $[\mathcal{S}(\overline{U}^{\top}\!G\overline{V}_{\!1})]_{ \widehat{\alpha}( \beta_+\cup\widehat{\beta}^0_+\cup \widehat{\beta}_0)}\!=\!0$. Together with the first equality of \eqref{SysBolockTemp1}, we obtain \eqref{R1Stru3}-\eqref{R1Stru4}. In addition, using \eqref{AlphGammTemp1} and $\lim_{\widehat{\mathcal{K}}\ni k\to\infty}d^2\Psi_{\kappa}(X^k|\Gamma^k)(G^k)=0$ also implies that
 \begin{align*}
 0&=\lim_{\widehat{\mathcal{K}}\ni k\to \infty}\sum_{j=1}^{q}\sum_{l'=r+1}^{ s(X^k)+1}\frac{\zeta_{j}(\Gamma^k)\| [\mathcal{S}((U^k)^{\top}\! G^k V_1^k)]_{\widehat{\beta}{j}\widehat{a}_{l'}}\|_F^2}{\nu_{r}(X^k)\!-\!\nu_{l'}(X^k)}\\
 &=\lim_{\widehat{\mathcal{K}}\ni k\to \infty}\sum_{j=1}^{q}\sum_{l'=r+1}^{r_1}\frac{\zeta_{j}(\Gamma^k)\| [\mathcal{S}((U^k)^{\top}\! G^k V_1^k)]_{\widehat{\beta}{j}\widehat{a}_{l'}}\|_F^2}{\nu_{r}(X^k)\!-\!\nu_{l'}(X^k)}\\
 &\quad+\lim_{\widehat{\mathcal{K}}\ni k\to \infty}\sum_{j=1}^{q}\sum_{l'=r_1+1}^{ s(X^k)+1}\frac{\zeta_{j}(\Gamma^k)\| [\mathcal{S}((U^k)^{\top}\! G^k V_1^k)]_{\widehat{\beta}{j}\widehat{a}_{l'}}\|_F^2}{\nu_{r}(X^k)\!-\!\nu_{l'}(X^k)}\\
 &=\lim_{\widehat{\mathcal{K}}\ni k\to \infty}\!\!\sum_{j\in \widehat{\beta}\setminus \widehat{\beta}_0}\!\sum_{i\in \eta^3}\frac{\sigma_{j}(\Gamma^k)\| [\mathcal{S}((U^k)^{\top}\! G^kV_1^k)]_{ji}\|_F^2}{\nu_{r}(X^k)\!-\!\sigma_i(X^k)}
 +\!\!\sum_{j\in \widehat{\beta}\setminus \widehat{\beta}_0}\!\sum_{i\in \gamma}\frac{\sigma_{j}(\overline{\Gamma})\|[\mathcal{S}(\overline{U}^{\top}\! G \overline{V}_{\!1})]_{ji}\|_F^2}{\nu_{\overline{r}}(\overline{X})\!-\!\sigma_i(\overline{X})},
 \end{align*}
 which implies that 
 $[\mathcal{S}(\overline{U}^{\top}\! G \overline{V}_{1})]_{[\widehat{\beta}\setminus (\widehat{\beta}_0\cup \widehat{\beta}^+_0)][\eta^3\cup \gamma]}=0$, i.e.,
 $[\mathcal{S}(\overline{U}^{\top}\! G \overline{V}_{1})]_{[\widehat{\beta}_1\cup \widehat{\beta}^1_+\cup \beta_+]\widehat{\gamma}}=0$. Together with the second equality in \eqref{SysBolockTemp1}, we obtain \eqref{R1Stru6}. Finally, we also have
 \begin{align*}
  0&=\lim_{\widehat{\mathcal{K}}\ni k\to \infty}\sum_{l=1}^{r-1}\frac{\|(U^k)_{\widehat{a}_l}^\top G^kV^k_{c}\|_F^2}{\nu_l(X^k)}=\lim_{\widehat{\mathcal{K}}\ni k\to \infty}\sum_{i\in \alpha\cup \eta^1}\frac{\|(U^k)_{i}^\top G^kV^k_{\!c}\|_F^2}{\sigma_i(X^k)}
 	=\sum_{i\in \alpha\cup \eta^1}\frac{\|\overline{U}_{i}^\top G\overline{V}_{\!c}\|_F^2}{\sigma_i(\overline{X})},\\
 0&\!=\!\lim_{\widehat{\mathcal{K}}\ni k\to \infty}\sum_{j=1}^{q}\frac{\zeta_{j}(\Gamma^k)\|U^\top_{\!\widehat{\beta}_j^k}GV^k_{c}\|_F^2}{\nu_{r}(X^k)}\!=\!\lim_{\widehat{\mathcal{K}}\ni k\to \infty}\sum_{j\in \widehat{\beta}}\frac{\sigma_{j}(\Gamma^k)\|(U_j^k)^{\top} GV^k_{c}\|^2}{\nu_{r}(X^k)}\!=\!\sum_{j\in \widehat{\beta}}\frac{\sigma_{j}(\overline{\Gamma})\|{\color{blue}\overline{U}^\top_{j}}G\overline{V}_{c}\|^2}{\nu_{\overline{r}}(\overline{X})},
 \end{align*}
 which implies that $[{\color{blue}\overline{U}}^\top G\overline{V}]_{[\widehat{\beta}_1\cup\widehat{\beta}^1_+\cup \beta_+]c}=0$ and $[\overline{U}^{\top} G \overline{V}_{\!1}]_{\widehat{\alpha} c}=0$, i.e.,  \eqref{R1Stru7} holds.
 
 \medskip
 \noindent
 {\bf\large Appendix B: Proof of equations \eqref{R2Stru1}-\eqref{R2Stru3}.}
 Recall that $\lim_{\widehat{\mathcal{K}}\ni k\to\infty}\sigma(X^k)=\sigma(\overline{X})$. By the definition of $\widehat{a}_l$ and equation \eqref{XjDecrease2}, if necessary by taking an infinite subset of $\widehat{\mathcal{K}}$, there must exist integer $r_0$ with $1\leq r_0<r\leq  s(X^k)+1$ such that 
 \begin{equation}\label{AlphGammTemp2}
 \!{\textstyle\bigcup}_{l'=1}^{\overline{r}-1} \overline{a}_{l'}=\alpha={\textstyle\bigcup}_{l=1}^{r_0} \widehat{a}_{l}\ \ {\rm and}\ \  \eta^1={\textstyle\bigcup}_{l=r_0+1}^{r-1} \widehat{a}_{l}.
 \end{equation}
 Using \eqref{AlphGammTemp2}, the previous \eqref{case2-3}, and the limit $\lim_{\widehat{\mathcal{K}}\ni k\to\infty}d^2\Psi_{\kappa}(X^k|\Gamma^k)(G^k)=0$ leads to 
 \begin{align*}
 0&=\lim_{\widehat{\mathcal{K}}\ni k\to\infty}\sum_{l=1}^{r-1}\sum_{j=0}^{q}
 \frac{(1+\zeta_{j}(\Gamma^k))}{\nu_l(X^k)}\|[\mathcal{T}((U^k)^{\top}\!G^kV_{1}^k)]_{\widehat{a}_{l}\widehat{\beta}_{j}}\|_F^2\\
 &=\lim_{\widehat{\mathcal{K}}\ni k\to\infty}\sum_{l=1}^{r_0}\sum_{j=0}^{q}
 \frac{(1+\zeta_{j}(\Gamma^k))}{\nu_l(X^k)}\|[\mathcal{T}((U^k)^{\top}\!G^kV_{1}^k)]_{\widehat{a}_{l}\widehat{\beta}_{j}}\|_F^2\\
 &\qquad+\lim_{\widehat{\mathcal{K}}\ni k\to\infty}\sum_{l=r_0+1}^{r-1}\sum_{j=0}^{q} 	\frac{(1+\zeta_{j}(\Gamma^k))}{\nu_l(X^k)}\|[\mathcal{T}((U^k)^{\top}\!G^kV_{1}^k)]_{\widehat{a}_{l}\widehat{\beta}_{j}}\|_F^2\\
 &=\sum_{i\in \alpha}\sum_{j\in \widehat{\beta}}
 	\frac{(1+\zeta_{j}(\overline{\Gamma}))}{\sigma_i(\overline{X})}\|[\mathcal{T}(\overline{U}^{\top}\!G\overline{V}_{1})]_{ij}\|_F^2\\
 &\quad+\lim_{\widehat{\mathcal{K}}\ni k\to\infty}\sum_{i\in \eta^1}\sum_{j\in \widehat{\beta}}	\frac{(1\!+\!\zeta_{j}(\Gamma^k))\|[\mathcal{T}((U^k)^{\top}\!G^kV_{1}^k)]_{\widehat{a}_{l}\widehat{\beta}_{j}}\|_F^2}{\nu_i(X^k)}
 \end{align*}
 which implies that  $\lim_{\widehat{\mathcal{K}}\ni k\to\infty}\sum_{i\in \eta^1}\sum_{j\in \widehat{\beta}}	\frac{(1+\zeta_{j}(\Gamma^k))\|[\mathcal{T}((U^k)^{\top}\!G^kV_{1}^k)]_{\widehat{a}_{l}\widehat{\beta}_{j}}\|_F^2}{\nu_i(X^k)}=0$ and
 \[ [\mathcal{T}(\overline{U}^{\top}\!G\overline{V}_{1})]_{\alpha\widehat{\beta}}=0.
 \]
 Recall that $\lim_{\widehat{\mathcal{K}}\ni k\to\infty}\nu_i(X^k)=0$ for all $i\in\eta^1$. From the above limit, we deduce that 
 \begin{align*}
  0&=\lim_{\widehat{\mathcal{K}}\ni k\to\infty}\sum_{i\in \eta^1}\sum_{j\in \widehat{\beta}}(1\!+\!\zeta_{j}(\Gamma^k))\|[\mathcal{T}((U^k)^{\top}\!G^kV_{1}^k)]_{\widehat{a}_{l}\widehat{\beta}_{j}}\|_F^2\\
  &=\sum_{i\in \eta^1}\sum_{j\in \widehat{\beta}}(1\!+\!\zeta_{j}(\overline{\Gamma}))\|[\mathcal{T}((\overline{U})^{\top}\!G\overline{V}_{1})]_{\widehat{a}_{l}\widehat{\beta}_{j}}\|_F^2.
  \end{align*}
  The above two equations imply  $[\mathcal{T}(\overline{U}^{\top}\!G\overline{V}_{1})]_{(\alpha\cup\eta^1)\widehat{\beta}}=0$. Recall that
 $\widehat{\alpha}=\alpha\cup \eta^1$ and 
 $\widehat{\beta}=\widehat{\beta}_1\cup \widehat{\beta}_+\cup \widehat{\beta}_0$ and $\widehat{\beta}_+=\widehat{\beta}^1_+\cup\beta_+\cup\widehat{\beta}^0_+$. Then, it holds that
 \begin{align}\label{AppdexBTemp1}
 [\mathcal{T}(\overline{U}^{\top}\!G\overline{V}_{1})]_{\widehat{\alpha}(\widehat{\beta}_1\cup \widehat{\beta}^1_+\cup\beta_+\cup\widehat{\beta}^0_+\cup \widehat{\beta}_0)}=0,
 \end{align}
 which implies that the second equality holds in \eqref{R2Stru1}.
 Similarly, by using \eqref{AlphGammTemp2}, the previous \eqref{case2-3}, and the limit $\lim_{\widehat{\mathcal{K}}\ni k\to\infty}d^2\Psi_{\kappa}(X^k|\Gamma^k)(G^k)=0$, we also have 
 \begin{align*}
 0&=\lim_{\widehat{\mathcal{K}}\ni k\to\infty}\sum_{l=1}^{r-1}\sum_{l'=1}^{r-1}
 	\frac{\|[\mathcal{T}((U^k)^{\top}\!G^kV_{1}^k)]_{\widehat{a}_{l}\widehat{a}_{l'}}\|_F^2}{\nu_l(X^k)\!+\!\nu_{l'}(X^k)}\\
 &=\lim_{\widehat{\mathcal{K}}\ni k\to\infty}\bigg[\sum_{l=1}^{r_0}\sum_{l'=1}^{r-1}
 	\frac{\|[\mathcal{T}((U^k)^{\top}\!G^kV_{1}^k)]_{\widehat{a}_{l}\widehat{a}_{l'}}\|_F^2}{\nu_l(X^k)\!+\!\nu_{l'}(X^k)}+\sum_{l=r_0+1}^{r-1}\sum_{l'=1}^{r_0}
 	\frac{\|[\mathcal{T}((U^k)^{\top}\!G^kV_{1}^k)]_{\widehat{a}_{l}\widehat{a}_{l'}}\|_F^2}{\nu_l(X^k)\!+\!\nu_{l'}(X^k)}\bigg]\\
 &\qquad+\lim_{\widehat{\mathcal{K}}\ni k\to\infty}\sum_{l=r_0+1}^{r-1}\sum_{l'=r_0+1}^{r-1}
 	\frac{\|[\mathcal{T}((U^k)^{\top}\!G^kV_{1}^k)]_{\widehat{a}_{l}\widehat{a}_{l'}}\|_F^2}{\nu_l(X^k)\!+\!\nu_{l'}(X^k)}\\
 &=\sum_{i\in \alpha}\sum_{j\in \alpha\cup \eta^1}
 	\frac{\|[\mathcal{T}(\overline{U}^{\top}\!G\overline{V}_{1})]_{ij}\|_F^2}{\sigma_i(\overline{X})\!+\!\sigma_{j}(\overline{X})}+\sum_{i\in \eta^1}\sum_{j\in \alpha}
 	\frac{\|[\mathcal{T}(\overline{U}^{\top}\!G\overline{V}_{1})]_{ij}\|_F^2}{\sigma_{j}(\overline{X})}\\
 &\qquad+\lim_{\widehat{\mathcal{K}}\ni k\to\infty}\sum_{i\in \eta^1}\sum_{j\in \eta^1}
 	\frac{\|[\mathcal{T}((U^k)^{\top}\!G^kV_{1}^k)]_{\widehat{a}_{l}\widehat{a}_{l'}}\|_F^2}{\nu_l(X^k)\!+\!\nu_{l'}(X^k)},
 \end{align*}
 which implies $[\mathcal{T}(\overline{U}^{\top}\!G\overline{V}_{1})]_{\alpha\widehat{\alpha}}=0,[\mathcal{T}(\overline{U}^{\top}\!G\overline{V}_{1})]_{\eta^1\alpha}=0$ and the following limit 
 \[
   \lim_{\widehat{\mathcal{K}}\ni k\to\infty}\sum_{i\in \eta^1}\sum_{j\in \eta^1}
   \frac{\|[\mathcal{T}((U^k)^{\top}\!G^kV_{1}^k)]_{\widehat{a}_{l}\widehat{a}_{l'}}\|_F^2}{\nu_l(X^k)\!+\!\nu_{l'}(X^k)}=0. 
 \]
 This limit, along with $\lim_{\widehat{\mathcal{K}}\ni k\to\infty}\nu_{l}(X^k)=0$, implies $[\mathcal{T}(\overline{U}^{\top}\!G\overline{V}_{1})]_{\eta^1\eta^1}=0$. Thus, the first equation in \eqref{R2Stru1} holds, and consequently, the two equalities in \eqref{R2Stru1} hold.
 Using the above \eqref{AlphGammTemp2} and  $\lim_{\widehat{\mathcal{K}}\ni k\to\infty}d^2\Psi_{\kappa}(X^k|\Gamma^k)(G^k)=0$ again leads to
 \begin{align*}
 0&=\lim_{\widehat{\mathcal{K}}\ni k\to\infty}\sum_{l=1}^{r-1}\sum_{j=1}^{q}\frac{2(1\!-\!\zeta_{j}(\Gamma^k))}{\nu_l(X^k)}\| [\mathcal{S}((U^k)^{\top}\!G^kV_{1}^k)]_{\widehat{a}_{l}\widehat{\beta}_{j}}\|_F^2\\
 &=\lim_{\widehat{\mathcal{K}}\ni k\to\infty}\sum_{l=1}^{r_0}\sum_{j=1}^{q}\frac{2(1\!-\!\zeta_{j}(\Gamma^k))}{\nu_l(X^k)}\| [\mathcal{S}((U^k)^{\top}\!G^kV_{1}^k)]_{\widehat{a}_{l}\widehat{\beta}_{j}}\|_F^2\\
 &\quad+\lim_{\widehat{\mathcal{K}}\ni k\to\infty}\sum_{l=r_0+1}^{r-1}\sum_{j=1}^{q}\frac{2(1\!-\!\zeta_{j}(\Gamma^k))}{\nu_l(X^k)}\| [\mathcal{S}((U^k)^{\top}\!G^kV_{1}^k)]_{\widehat{a}_{l}\widehat{\beta}_{j}}\|_F^2\\
 &=\sum_{i\in \alpha}\sum_{j\in \widehat{\beta}\setminus \widehat{\beta}_0}\frac{2(1\!-\!\zeta_{j}(\overline{\Gamma}))}{\sigma_i(\overline{X})}\| [\mathcal{S}(\overline{U}^{\top}\!G\overline{V}_{1})]_{ij}\|_F^2\\
 &\quad+\lim_{\widehat{\mathcal{K}}\ni k\to\infty}\sum_{i\in \eta^1}\sum_{j\in \widehat{\beta}\setminus \widehat{\beta}_0}\frac{2(1\!-\!\zeta_{j}(\Gamma^k))\| [\mathcal{S}((U^k)^{\top}\!G^kV_{1}^k)]_{ij}\|_F^2}{\sigma_i(X^k)}
 \end{align*}
 which implies $\lim_{\widehat{\mathcal{K}}\ni k\to\infty}\sum_{i\in \eta^1}\sum_{j\in \widehat{\beta}\setminus \widehat{\beta}_0}\frac{2(1\!-\!\zeta_{j}(\Gamma^k))\| [\mathcal{S}((U^k)^{\top}\!G^kV_{1}^k)]_{ij}\|_F^2}{\sigma_i(X^k)}=0$ and
 \[
 [\mathcal{S}(\overline{U}^{\top}\!G\overline{V}_{1})]_{\alpha, \widehat{\beta}\setminus (\widehat{\beta}_0\cup\widehat{\beta}_1\cup\widehat{\beta}^1_+)}\!=\!0.
 \]
 The former, along with $\lim_{\widehat{\mathcal{K}}\ni k\to\infty}\sigma_i(X^k)=0$ for all $i\in\eta^1$, implies 
 \begin{equation*}
 0=\lim_{\widehat{\mathcal{K}}\ni k\to\infty}\sum_{i\in \eta^1}\sum_{j\in \widehat{\beta}\setminus \widehat{\beta}_0}(1\!-\!\zeta_{j}(\Gamma^k))\| [\mathcal{S}((U^k)^{\top}\!G^kV_{1}^k)]_{ij}\|_F^2=(1\!-\!\zeta_{j}(\overline{\Gamma}))\| [\mathcal{S}((\overline{U})^{\top}G\overline{V}_{1})]_{ij}\|_F.
 \end{equation*}
 The above two equations imply that    $[\mathcal{S}(\overline{U}^{\top}\!G\overline{V}_{1})]_{\alpha\cup \eta^1, \widehat{\beta}\setminus (\widehat{\beta}_0\cup\widehat{\beta}_1\cup\widehat{\beta}^1_+)}\!=\!0$. Together with $\widehat{\beta}\setminus(\widehat{\beta}_0\cup\widehat{\beta}_1\cup\widehat{\beta}^1_+)=\beta_+\cup \widehat{\beta}^0_+$ and $\alpha\cup \eta^1=\widehat{\alpha}$, we have $[\mathcal{S}(\overline{U}^{\top}\!G\overline{V}_{1})]_{\widehat{\alpha},\beta_+\cup \widehat{\beta}^0_+}\!=\!0$, which along with \eqref{R2Stru2} implies that the second and third equalities in \eqref{R2Stru2} hold. In addition, 
 \begin{align*} 
  0&=\lim_{\widehat{\mathcal{K}}\ni k\to\infty}\sum_{l=1}^{r-1}\frac{\|(U^k)^\top_{\widehat{a}_l}G^kV_{c}^k\|_F^2}{\nu_l(X^k)}\\
  &=\lim_{\widehat{\mathcal{K}}\ni k\to\infty}\sum_{l=1}^{r_0}\frac{\|(U^k)^\top_{\widehat{a}_l}G^kV_{c}^k\|_F^2}{\nu_l(X^k)}+\lim_{\widehat{\mathcal{K}}\ni k\to\infty}\sum_{l=r_0+1}^{r-1}\frac{\|(U^k)^\top_{\widehat{a}_l}G^kV_{c}^k\|_F^2}{\nu_l(X^k)}\\
  &=\sum_{i\in \alpha}\frac{\|\overline{U}^\top_{i}G\overline{V}_{c}\|_F^2}{\sigma_i(\overline{X})}+\lim_{\widehat{\mathcal{K}}\ni k\to\infty}\sum_{i\in \eta^1}\frac{\|(U^k)^\top_{i}G^kV_{c}^k\|_F^2}{\sigma_i(X^k)},
 \end{align*}
 which implies that $\overline{U}^\top_{\alpha\cup \eta^1}G\overline{V}_{c}=0$, i.e., the first equality in \eqref{R2Stru2} holds,  and
 \begin{align*}
 0 &=\lim_{\widehat{\mathcal{K}}\ni k\to\infty}\!\sum_{l=1}^{r-1}\!\frac{\| [\mathcal{S}((U^k)^{\top}\!G^kV_{1}^k)]_{\widehat{a}_{l}\widehat{\beta}_0}\|_F^2}{\nu_l(X^k)}\\
 &=\lim_{\widehat{\mathcal{K}}\ni k\to\infty}\!\sum_{l=1}^{r_0}\!\frac{\| [\mathcal{S}((U^k)^{\top}\!G^kV_{1}^k)]_{\widehat{a}_{l}\widehat{\beta}_0}\|_F^2}{\nu_l(X^k)}+\lim_{\widehat{\mathcal{K}}\ni k\to\infty}\!\sum_{l=r_0+1}^{r-1}\!\frac{\| [\mathcal{S}((U^k)^{\top}\!G^kV_{1}^k)]_{\widehat{a}_{l}\widehat{\beta}_0}\|_F^2}{\nu_l(X^k)}\\
 &=\sum_{i\in \alpha}\!\frac{\| [\mathcal{S}(\overline{U}^{\top}\!G\overline{V}_{1})]_{i\widehat{\beta}_0}\|_F^2}{\sigma_i(\overline{X})}+\lim_{\widehat{\mathcal{K}}\ni k\to\infty}\!\sum_{i\in \eta^1}\!\frac{\| [\mathcal{S}((U^k)^{\top}\!G^kV_{1}^k)]_{i\widehat{\beta}_0}\|_F^2}{\sigma_i(X^k)},	
 \end{align*} 
 which implies that $[\mathcal{S}(\overline{U}^{\top}\!G\overline{V}_{1})]_{\alpha\cup \eta^1,\widehat{\beta}_0}=0$. Together with \eqref{AppdexBTemp1}, we obtain \eqref{R2Stru3}.
 
 \end{document}